\documentclass[a4paper,11pt]{article}
\usepackage{amssymb,amsthm,graphicx,inputenc,mathtools,verbatim}
\usepackage{hyperref,url,enumitem,bm,esint,color}
\usepackage[margin=30mm]{geometry}

\usepackage{tikz}
\usepackage{cite}

\allowdisplaybreaks

\definecolor{rosso}{rgb}{0.85,0,0}
\definecolor{azzurro}{rgb}{0.13, 0.67, 0.8}

\newcommand{\ov}[1]{\overline{#1}}

\renewcommand{\hat}[1]{\widehat{#1}}
\let\badeps\epsilon
\newcommand{\eps}{\varepsilon}

\newcommand{\pd}{\partial}

\newcommand{\Lpanarh}{\mathcal L}
\newcommand{\Lpanarv}{\tilde{{\mathcal L}}}

\newcommand{\dn}{\partial_{\bnn}}
\def\bnn{{\boldsymbol n}}
\def\bnu{{\boldsymbol \nu}}

\def\<#1>{\mathopen\langle #1\mathclose\rangle}

\renewcommand{\tilde}{\widetilde}

\def\non{\notag}

\def\erre{{\mathbb{R}}}
\def\enne{{\mathbb{N}}}

\newcommand{\bigchi}{\ensuremath{\mathrm{\mathcal{X}}}}
\newcommand{\charfcn}[1]{\bigchi_{#1}} 

\theoremstyle{plain}

\newtheorem{lem}{Lemma}[section]
\newtheorem{prop}{Proposition}[section]
\newtheorem{cor}{Corollary}[section]
\newtheorem{remark}{Remark}[section]

\numberwithin{equation}{section}

\def\Lip{Lip\-schitz}

\def\jump#1{[#1]^+_-}
\def\Sp{{S_+}}
\def\Sm{{S_-}}

\def\th{\vartheta}

\def\Lm{{\lambda}_-}
\def\Lp{{\lambda}_+}
\def\Lmr{({\lambda}_- r)}
\def\LmR{({\lambda}_- R)}
\def\Lmqs{({\lambda}_- \qstar)}
\def\Lpqs{({\lambda}_+ \qstar)}
\def\Lmq{({\lambda}_- q)}
\def\Lpq{({\lambda}_+ q)}
\def\Lpr{({\lambda}_+ r)}
\def\LpR{({\lambda}_+ R)}

\def\Il{\mathbb{I}_\ell}
\def\Kl{\mathbb{K}_\ell}
\def\Iz{\mathbb{I}_0}
\def\Kz{\mathbb{K}_0}

\def\rad{Y}

\def\multibold #1{\def\arg{#1}%
  \ifx\arg\pto \let\next\relax
  \else
  \def\next{\expandafter
    \def\csname #1#1\endcsname{{\bf #1}}%
    \multibold}%
  \fi \next}

\def\pto{.}

\def\multimathbb #1{\def\arg{#1}%
  \ifx\arg\pto \let\next\relax
  \else
  \def\next{\expandafter
    \def\csname #1#1#1\endcsname{{\mathbb #1}}%
    \multimathbb}%
  \fi \next}

\def\multical #1{\def\arg{#1}%
  \ifx\arg\pto \let\next\relax
  \else
  \def\next{\expandafter
    \def\csname cal#1\endcsname{{\cal #1}}%
    \multical}%
  \fi \next}

\def\multimathop #1 {\def\arg{#1}%
  \ifx\arg\pto \let\next\relax
  \else
  \def\next{\expandafter
    \def\csname #1\endcsname{\mathop{\rm #1}\nolimits}%
    \multimathop}%
  \fi \next}

\multibold
qwertyuiopasdfhjklzxcvbnmQWERTYUIOPASDFGHJKLZXCVBNM.

\multimathbb
QWERTYUIOPASDFGHJKLZXCVBNM.

\multical
QWERTYUIOPASDFGHJKLZXCVBNM.

\multimathop
diag dist div dom mean meas sign supp .

\def\nn{\boldsymbol{n}}

\def\bmin{b_{d,-}^{{\rm in}}}
\def\bmex{b_{d,-}^{{\rm ext}}}

\def\cmex{c_{d,-}^{{\rm ext}}}

\def\emin{e_-^{{\rm in}}}
\def\emex{e_-^{{\rm ext}}}

\def\mumin{\mu_-^{{\rm in}}}
\def\mumin{\mu_-^{{\rm in}}}
\def\mumex{\mu_-^{{\rm ext}}}

\def\qud{(q_1,q_2)}
\def\qudst{(\qstaru,\qstard)}

\def\qstar{{q^\star}}
\def\qstaru{{q^\star_1}}
\def\qstard{{q^\star_2}}

\def\mustin{\mu_-^{\star, {\rm in}}}
\def\mustpl{\mu_+^{\star}}
\def\mustext{\mu_-^{\star, {\rm ext}}}

\def\uin{u_-^{{\rm in}}}
\def\upl{u_+}
\def\uext{u_-^{{\rm ext}}}
\def\Uin{U_-^{{\rm in}}}
\def\Upl{U_+}
\def\Uext{U_-^{{\rm ext}}}
\def\vin{v_-^{{\rm in}}}

\def\vext{v_-^{{\rm ext}}}
\def\Uin{U_-^{{\rm in}}}
\def\Upl{U_+}

\def\Uext{U_-^{{\rm ext}}}

\def\Iell{{\Il}}
\def\Kell{{\Kl}}

\def\dplus{{\frac{S_+}{\rho_+}}}
\def\dmin{{\frac{S_-}{\rho_-}}}

\title{
	On a  Mullins--Sekerka model for the growth of active droplets modelling protocells: Stability analysis and numerical computations
}

\author{Harald Garcke \footnotemark[1] \and Kei Fong Lam \footnotemark[2] \and Robert N\"urnberg \footnotemark[3] \and Andrea Signori \footnotemark[4]}

 \date{}

\begin{document}
	
\maketitle

\begin{abstract}
\noindent
Mullins--Sekerka models with chemical reactions can lead to scenarios where drop\-lets grow, become unstable, split, grow {and undergo further division}. These grow and division cycles have been proposed as a model for protocells and {are} {believed} to play a {fundamental}  role in living systems by providing chemical compartments which are important in the organization of living systems.
	{This} paper analyses  chemically active Mullins--Sekerka models. Existence of radially symmetric solutions is shown and  a detailed stability analysis in radial as well as planar situations is given. In particular, we also analyze multilayered solutions leading to shell-type situations. {Finally, we introduce a numerical method based on a parametric finite element approach that explicitly accounts for topological changes, thereby allowing} for droplet splitting and merging.
	Several numerical simulations verify the findings of the theoretical stability analysis and show complex {dynamical behavior, including multiple} instabilities, splittings of droplets and appearance of shell-type solutions.
	
\noindent
\end{abstract}



\renewcommand{\thefootnote}{\fnsymbol{footnote}}
\footnotetext[1]{Fakult{\"a}t f\"ur Mathematik, Universit{\"a}t Regensburg, 93040 Regensburg, Germany
({\texttt harald.garcke@ur.de}).}
\footnotetext[2]{Department of Mathematics, Hong Kong Baptist University, Kowloon Tong, Hong Kong ({\texttt akflam@hkbu.edu.hk}).}
\footnotetext[3]{Department of Mathematics, University of Trento, 38123 Trento, Italy
({\texttt robert.nurnberg@unitn.it}).}
\footnotetext[4]{Department of Mathematics, Politecnico di Milano, 20133 Milano, Italy ({\texttt andrea.signori@polimi.it}), Alexander von Humboldt Research Fellow.}

\renewcommand{\thefootnote}{\arabic{footnote}}

\section{Introduction}

	In phase separating systems  with chemical reactions{, growth regimes may arise in which Ostwald ripening does not occur}. Instead, droplets only grow
to a certain size and then coexist with the surrounding phase, see \cite{zwickerostwald}.  The formation of such droplets has been suggested to play {a} crucial  role in living systems by providing biochemical compartments which are important in the spatial organization of biological processes.  For {certain parameter}  regimes  describing surface tension, diffusion and chemical reactions{,} growing droplets can become unstable.  {This may trigger droplet splitting, after which} the smaller droplets  grow and later  split again. This spontaneous droplet division can hence lead to growth and division cycles and   it has been suggested in \cite{Active_drops} as a model for the formation of protocells. 
Recently, the present authors in \cite{OOL} derived a new Mullin{--}Sekerka model with chemical reactions which appears as the sharp interface limit of
a Cahn--Hilliard equation with {fast} chemical reactions. The latter Cahn--Hilliard model was in fact used in \cite{Active_drops} as a diffuse interface model which  models   chemically active droplet systems with growth and splitting scenarios. It is the goal of this paper to systematically analyze the new chemically active Mullins--Sekerka model and to present sharp interface computations of the model which also allow for topological changes. This enables us to compute growth and splitting scenarios within a sharp interface {context}.


%
In \cite{OOL}, we provided a detailed mathematical analysis of the Cahn--Hilliard model introduced in \cite{Active_drops}. 
{After carefully formulating the model, we} established a well-posedness result for the {resulting} system. We then applied formally matched asymptotic expansions to connect the diffuse interface model to a sharp interface that turned out to be of  Mullins--Sekerka type. Although a sharp interface model was already proposed in \cite{Active_drops}, our analysis {reveals} that the asymptotic expansions lead to a quasi-static diffusion problem that can {take} additional source terms {arising from chemical reactions at the interface. Namely, t}he model involves  the following quasi-static diffusion equations with chemical reactions in the two phases $\Omega^\pm${:}
$$- m_\pm \Delta \mu_\pm = S_\pm - \rho_\pm \mu_\pm  \quad \text{in $\Omega^\pm$,}$$
coupled {with} equations on the free boundary relating 
{the chemical potentials $\mu_\pm$} to the  mean curvature of the interface and an equation for the normal velocity of the interface. In particular, we allow for sources and sinks due to reactions at the interface.

In this paper, we  focus on the Mullins--Sekerka-type free boundary problem.  We prove that{,} for sufficiently strong sinks at the interface, droplets necessarily shrink. Specifically, we analyze the existence of radial solutions and examine their stability {with respect to both radial and} non-radial perturbations. 
We demonstrate that, for certain parameter {regimes}, two distinct radial stationary solutions exist, one stable and one unstable to radial perturbations. We then consider non-radial perturbations and can identify parameter regimes
for which stationary radial solutions {lose stability}. For some of theses regimes the two-mode perturbation can be shown to be {the} most unstable which {can}  then lead to the splitting of one {droplet} into two. Additionally, we explore the evolution of shell-like solutions and {carry out} an involved stability analysis. Th{is} analysis  is non-classical as perturbations at two disjoint interfaces influence each other through a coupling via the bulk. Finally, we present several numerical simulations which {{confirm}
	 the theoretical} findings.
{Overall,}  the present paper provide{s} a deeper insight into the dynamics of droplets in phase-separating systems influenced by chemical reactions{, within a} sharp interface {framework}.

{Similar unstable growth phenomena have previously  been observed  in Cahn--Hilliard and Mullins--Sekerka type problems with source terms in tumour growth.
We refer to the works of {Cristini} and  Lowengrub together with their co-authors \cite{CristiniLN03,CristiniLLW09}, Friedman and Hu \cite{FriedmanHu06}, Escher and Matioc \cite{EscherM11} and \cite{Hawkins-DaarudZO12, GarckeLSS16, GarckeLNS18}
for further details. However, in these models so far no multiple splitting scenarios and shell type solutions have been observed,
which therefore represent a new feature of the model investigated in this paper. The pioneering work of Brangwynne and Hyman \cite{BrangwynneHyman09,BrangwynneHyman11} demonstrated that droplets originating from phase separation play a {pivotal} role in living cells.
They discovered a completely new physical mechanism  for cellular interactions between proteins and other biomolecules  in the absence  of membranes. The researchers described dynamic, fluid-like droplets that form rapidly through phase separation --- similar to oil droplets in water --- creating temporary structures that are protected from the outside. Since their discovery, they and others have shown that such membrane-less liquid condensates play a role in numerous cellular processes, such as cellular signalling, cell division, the nested structure of nucleoli in the cell nucleus, and DNA regulation. Their discovery represents  a fundamental advance in understanding cellular organisation.  For reviews focusing on the role of non-membrane-bound compartments 
 that arise  as phase-separated droplets in biology we refer to the reviews \cite{HymanWJ14,Brangwynne13}.
It has  also been  suggested that the segregation  of molecules by phase separation has played an important role during the early stages in the origin of life
\cite{Active_drops}.
More recently, both theoretical and experimental studies have shown that such systems can also form stable spherical shells
  \cite{bauermann2023formation, BergmannBBetal23}.
However, a thorough mathematical analysis of the underlying models is missing and the present paper aims to initiate   such an analysis.}%

This paper is organized as follows. In Section 2{,} we precisely formulate the Mullins--Sekerka  system with chemical reactions, non-dimensionalize 
the problem, identify regimes in which droplets decay and {briefly recall} a related diffuse interface model which is of Cahn--Hilliard type. In Section 3{,} we analyze 
radially symmetric solutions with {a single} interface. 
For certain parameter values{,} we {establish the} existence of two stationary  spherical solutions. A stability analysis is performed which helps to identify parameter regimes in which 
splitting into {two} droplets can be expected. In Section 4{,} we analyze multi-layered solutions{,} first in the planar case and {then} in the radially symmetric case. Stability for multi-layered solutions has to be studied by using for each perturbation mode an intricate interaction between the layers.
Section 5 introduces a parametric finite element method for the chemically active Mullins--Sekerka problem and  presents several numerical computations including very complex splitting scenarios which are the {first} {to be} produced with a sharp interface approach. {Finally, two appendices collect auxiliary} computations involving radial solutions and modified Bessel functions which are needed for the stability analysis.

\section{The mathematical models}
\label{SEC:MOD}

We begin this section by introducing mathematical models that are the central focus of this paper. The model for active droplets, initially presented in \cite{zwickerostwald, Active_drops, bauermann2023formation}, employs a phase-field approach based on a variant of the Cahn--Hilliard equation, see \eqref{eq:sys21} below. 
{In \cite{OOL}, we first refined the previously introduced model by placing it within a rigorous mathematical framework and establishing its well-posedness. We then performed a formally matched asymptotic analysis to derive a Mullins--Sekerka-type free boundary problem from the phase-field formulation. This sharp-interface model, which can be stated as follows, constitutes the main focus of the present work.}

{%
Throughout, let $\Omega$ be a {\Lip\ continuous} bounded domain in $\mathbb{R}^d$, where $d \in \{2,3\}$, with boundary $\partial\Omega$, and let $T>0$ be a fixed but arbitrary time. Moreover, $\dn$ denotes the derivative in the direction of the unit outer normal $\nn$ to $\partial \Omega$.
}

\subsection{Mullins--Sekerka system with reactions}

We assume {that} $\Omega$ is subdivided into two disjoint open sets $\Omega^+(t)$ and $\Omega^-(t)$ that are separated by an evolving interface $\Sigma(t) := \pd \Omega^+(t) \cap \pd \Omega^-(t)$. A typical {configuration} is {given by} $\Omega^+$ {being} enclosed by $\Omega^-$ with $\Sigma \cap \pd \Omega = \varnothing$. 
The governing equations of the Mullins--Sekerka problem, which describe the evolution of active droplets {can be formulated as follows. Given
$\Sigma(0) = \Sigma_0$, for $t \in (0,T)$ {we seek the evolving interface} $\Sigma(t)$ partitioning
$\Omega$ as described above{, together with a chemical potential} $\mu(t,\cdot): \Omega \to \mathbb R$ such that} 
\begin{subequations}\label{SharpI}
\begin{alignat}{2}
	\label{sharpI:1:p}
	 - m_+ \Delta \mu &= \Sp - \rho_+ \mu  \quad &&\text{in $\Omega^+$,}
		\\
        \label{sharpI:1:m} - m_- \Delta \mu & =\Sm - \rho_- \mu  \quad && \text{in $\Omega^-$,}
	\\
	\label{sharpI:2}
	\mu &=  	{\alpha}	\kappa
	\qquad && \text{on $\Sigma$,}
	\\
	\label{sharpI:4}
	 - 2 {\cal V} & = \jump{m\nabla \mu} \cdot {\bnu}
	+ S_I
	\qquad && \text{on $\Sigma$,}
	\\
	\label{sharpI:5}
	 \dn \mu & =0
	\qquad && \text{on $\partial \Omega$,}
\end{alignat}
\end{subequations}
where we take
$S_+<0$, while $S_-$,
$\rho_\pm $, $m_{\pm}$ and {$\alpha$}
 are positive parameters,  and $S_I$ is the interfacial reaction constant.
In the above, the vector $\bnu$ denotes the unit normal to the interface $\Sigma$ pointing into $\Omega^+$, $\kappa = - \div_{\Sigma} \bnu$ represents the mean curvature of $\Sigma$ defined as the sum of the principal curvatures of $\Sigma$ {or, equivalently,} as the negative surface divergence of the unit normal $\bnu$. We {adopt} the convention that $\kappa$ is positive if $\Sigma$ is convex. The quantity ${\cal V}$ is the normal velocity of the interface in the direction of the normal $\bnu$, we refer to Figure~\ref{fig:setting} for an illustration of the geometric setting. We also consider settings where $\Sigma$ intersects the external boundary $\pd \Omega$, and for these cases we prescribe a $90^{\circ}$ {contact} angle condition.\\

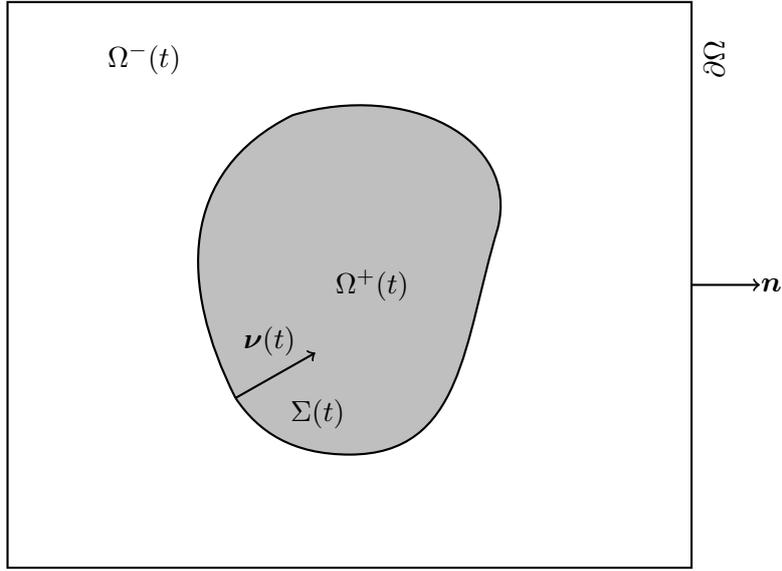
\begin{figure}[h]
\centering
\begin{tikzpicture}[scale=1.5]
\draw[thick] (0,0) rectangle (6,5);

\filldraw[color=black!100, fill=black!25, thick, line join=bevel]
  (2,1.5) .. controls (1.5,2.5) and (1.5,3.5) .. (2.5,4)
          .. controls (3.5,4.3) and (4.5,3.8) .. (4.3,3)
          .. controls (4,2) and (4,1) .. (3,1)
          .. controls (2.5,1) and (2.2,1.2) .. (2,1.5)
  node[pos=0.5, above right] {$\Sigma(t)$};

\node at (3.2,2.5) {$\Omega^{+}(t)$};
\node at (1.2,4.5) {$\Omega^{-}(t)$};
\draw[->, thick] (2.0,1.5) -- (2.7,1.9);
\node at (2.3,2.0) {$\boldsymbol{\nu}(t)$};
\draw[->, thick] (6,2.5) -- (6.6,2.5);
\node at (6.7,2.5) {$\boldsymbol{n}$};
\node[rotate=90] at (6.2,4.5) {$\partial\Omega$};
\end{tikzpicture}
\caption{The geometric setting involving the partition of $\Omega$ into two disjoint time dependent open sets $\Omega^+(t)$ and $\Omega^-(t)$ that are separated by an interface $\Sigma(t)$. {The unit normal $\bnu$ of $\Sigma(t)$ points into $\Omega^+(t)$.} 
}
\label{fig:setting}
\end{figure}

Additionally, for $x \in \Sigma(t)$ and a function $u$, we define its jump across the interface at $(t, x)$ as:
\begin{equation*}
	[u]^+_-(t,x):=\lim_{\substack{y\to x\\y\in\Omega^+(t)}} u(t,y) -
	\lim_{\substack{y\to x\\y\in\Omega^-(t)}} u(t,y)\,.
\end{equation*}
We will also frequently use the notation $\mu_\pm$ for $\mu$ when restricted to the sets $\Omega^\pm$.

It is important to note that this system is {closely} related to the well-known Mullins--Sekerka free boundary problem \cite{BDGP}. The key differences here are the presence of a constant source term, $S_I$, on the right-hand side of \eqref{sharpI:4}, as well as an affine linear term in the quasi-static diffusion equations \eqref{sharpI:1:p} and \eqref{sharpI:1:m}.
 
{Performing} a nondimensionalization of {\eqref{SharpI}}, similarly done as in \cite{OOL}{,} leads to the following dimensionless system: 
\begin{subequations}\label{Sharp:adim}
\begin{alignat}{3}
\label{sharpIIn:1:p}
	-  m^*\Delta \mu &=
		S^* -  \rho^* \mu \quad && \text{in $ \Omega^+$},
		\\
		\label{sharpIIn:1:m}
			-  \Delta  \mu &=  1 -
	\mu  \quad && \text{in $  \Omega^-$,}
	\\
	\label{sharpIIn:2}
	 \mu & = {\alpha^*}\kappa
	\qquad & &\text{on $ \Sigma$,}
	\\
	\label{sharpIIn:4}
	 - 2 { \cal  V} &= m^* \nabla \mu_+ \cdot { \bnu} - \nabla \mu_- \cdot { \bnu}
	+ S_I^*
	\qquad && \text{on $\Sigma$,}
	\\
	\label{sharpIIn:5}
	 \dn \mu &=0
	\qquad && \text{on $\partial \Omega$,}
\end{alignat}
\end{subequations}
where after choosing suitable units for length $\tilde{x} = \sqrt{m_-/\rho_-}$, {for} time $\tilde{t} = 1/S_-$ and {for the} chemical potential ${\tilde{\mu}_-} = S_-/\rho_-$, we have the dimensionless parameter
\begin{align}
	\label{cap:length}
	 \alpha^* =\frac{\alpha}{\tilde x} \frac 1{{\tilde{\mu}_-}}=
	 \frac{\alpha \rho_-}{S_-}\sqrt{\frac{\rho_-}{m_-}},
\end{align}
as the ratio of a modified capillary length $c_l := \frac{\alpha \rho_-}{S_-}$ and the length scale $\tilde{x}$. 
Moreover, the relative mobility $m^*$,  the relative reaction coefficients $S^*$ and $\rho^*$, and the nondimensional interface reaction term $S_I^*$ are given by 
\begin{equation*}
	m^*=\frac {m_+} {m_-},\quad S^*=\frac {S_+} {S_-},\quad
	\rho^*=\frac {\rho_+} {\rho_-}, \quad S_I^* = \frac{1}{S_-} \sqrt{\frac{\rho_-}{m_-}} S_I.
\end{equation*}

In the following {result}, we show that in {two spatial dimensions} and for sufficiently negative $S_I${,} the perimeter of $\Omega^+(t)$ decreases over time.
\begin{prop}[Dissipative inequality in two space dimensions]
Suppose that \begin{align}\label{diss:cond}
	2 \pi \alpha S_I + \frac{|\Omega|}{{4}}\Big ( \frac{S_+^2}{\rho_+} + \frac{S_-^2}{\rho_-} \Big ) \leq 0
\end{align}
and that {the spatial dimension $d=2$.}
Then any solution of \eqref{SharpI}{,} 
{with $\Sigma(t)$ being} a closed simple curve{,} satisfies
 \[
 \alpha \frac{d}{dt} |\Sigma(t)| \leq 0.
 \]
 
\end{prop}

\begin{proof}
{Multiplying  \eqref{sharpI:1:p} with $\mu_+$ and \eqref{sharpI:1:m} with $\mu_-$, integrating, performing  integration  by parts and adding the resulting equalities, we obtain}
\begin{equation}\label{energy:id}
\begin{aligned}
	& 2 \alpha\frac d{dt}  \big|\Sigma(t)\big|
	+ m_+ \int_{\Omega^+}  |\nabla \mu_+|^2
	+ m_- \int_{\Omega^-} |\nabla \mu_-|^2
	+ \rho_+ \int_{\Omega^+} | \mu_+|^2
	+ \rho_- \int_{\Omega^-}| \mu_-|^2
	\\ & \quad 
	= 
	\int_{\Omega^+} S_+ \mu_+
	+ \int_{\Omega^-} S_-  \mu_-
	+\alpha S_I \int_{\Sigma} \kappa.
	\end{aligned}
\end{equation}
{As $\Sigma$ is a closed simple curve,} we can apply the Gauss--Bonnet theorem to express the last term as
\[
\alpha S_I  \int_{\Sigma} \kappa =  2 \pi \alpha S_I.
\]
To control the right-hand side terms involving $\mu_+$ and $\mu_-$, we apply Young's inequality {to find}
\begin{align*}
	\int_{\Omega^+} S_+ \mu_+
	+ \int_{\Omega^-} S_-  \mu_-
	\leq 
	{\rho_+} \int_{\Omega^+} | \mu_+|^2
	+ {\rho_-} \int_{\Omega^-}| \mu_-|^2
	+ \frac 1{{4} \rho_+} S_+^2 |{\Omega^+}| 
	+ \frac 1{{4} \rho_-} S_-^2 |{\Omega^-}| .
\end{align*}
Combining the above, we obtain
\begin{align*}
	& 2 \alpha\frac d{dt}  \big|\Sigma(t)\big|
	{+ m_+ \int_{\Omega^+}  |\nabla \mu_+|^2
	+ m_- \int_{\Omega^-} |\nabla \mu_-|^2}
	\\ & \quad 
	\leq
	\frac 1{{4} \rho_+} S_+^2 |{\Omega^+}| 
	+ \frac 1{{4}  \rho_-} S_-^2 |{\Omega^-}|
	+ 2 \pi \alpha S_I \leq \frac{|\Omega|}{{4} } \Big ( \frac{S_+^2}{\rho_+} + \frac{S_-^2}{\rho_-} \Big ) + 2 \pi \alpha S_I.
\end{align*}
The condition \eqref{diss:cond} guarantees that the right-hand side is nonpositive.
\end{proof}

\begin{remark}
We note that condition \eqref{diss:cond} {is equivalent to requiring that}
\[
S_I \leq \frac{-|\Omega|}{{8} \alpha \pi}\Big ( \frac{S_+^2}{\rho_+} + \frac{S_-^2}{\rho_-} \Big ),
\]
which implies $S_I$ {must} be nonpositive. Furthermore, the above argument can be generalized to geometric settings where $\Omega^+ = \bigcup_{j =1}^{n} \Omega^+_{j}$ is a union of connected components such that $\Sigma_j :=\pd \Omega^+_{j}$ is a simple closed curve with associated curvature $\kappa_j$ and $\Sigma_j \cap \pd \Omega = \varnothing$ for $j = 1, \dots, n$. Then, by orienting the normals $\bnu_j$ to $\Sigma_j$ to point into $\Omega^+_{j}$ from the surrounding $\Omega^-${,} the identity \eqref{energy:id} holds for each connected component $\Omega^+_{j}$, and upon summing we have
\begin{align*}
&2 \alpha \frac{d}{dt} \sum_{j=1}^{n} |\Sigma_j(t)| + m_+ \sum_{j=1}^n \int_{\Omega^{+}_j} |\nabla \mu_+|^2 + m_- \int_{\Omega^{-}} |\nabla \mu_-|^2 + \rho_+ \sum_{j=1}^n \int_{\Omega^+_j} |\mu_+|^2 + \rho_- \int_{\Omega^-} |\mu_-|^2 \\
& \quad = \sum_{j=1}^n \int_{\Omega^+_j} S_+ \mu_+ + \int_{\Omega^-} S_- \mu_- + \alpha S_I \sum_{j=1}^n \int_{\Sigma_j} \kappa_j.
\end{align*}
Applying the Gauss--Bonnet theorem to the last term and writing $\Sigma = \bigcup_{j=1}^n \Sigma_j${,} we find that 
\begin{align*}
&2 \alpha \frac{d}{dt}  |\Sigma(t)| + m_+ \int_{\Omega^{+}} |\nabla \mu_+|^2 + m_- \int_{\Omega^{-}} |\nabla \mu_-|^2 + \rho_+  \int_{\Omega^+} |\mu_+|^2 + \rho_- \int_{\Omega^-} |\mu_-|^2 \\
& \quad =  \int_{\Omega^+} S_+ \mu_+ + \int_{\Omega^-} S_- \mu_- + 2 \pi n \alpha S_I.
\end{align*}
Then, the analogue condition to \eqref{diss:cond} that guarantees the nonincreasing of total interfacial perimeter is 
\[
2n \pi \alpha S_I + \frac{|\Omega|}{{4} }\Big ( \frac{S_+^2}{\rho_+} + \frac{S_-^2}{\rho_-} \Big ) \leq 0,
\]
where $n$ is the number of connected components of $\Omega^+$.
\end{remark}

\subsection{Related Cahn--Hilliard diffuse interface model}
For completeness, let us now briefly introduce the diffuse interface model that can be connected to the system {\eqref{SharpI}}\ in the limit of vanishing interfacial thickness. 
{In this model, the phase field $\varphi$ represents the normalized concentration difference between the two chemical components, scaled between $-1$ and $1$, while $\mu$ is the related chemical potential.}
The Cahn--Hilliard equation used to describe the dynamics of active droplets, formalized in \cite{OOL}, is given {as follows. Find $\varphi:[0,T) \times \Omega \to \mathbb R$ and $\mu:(0,T)\times\Omega \to \mathbb R$ such that}
\begin{subequations}\label{eq:sys21}
\begin{alignat}{2}
    \label{sys:1}
    \partial_t \varphi &= \div (m(\varphi) \nabla \mu) + S_\varepsilon(\varphi)
    \qquad && \text{in } (0,T) \times \Omega, \\
    \label{sys:2}
    \mu &= - \beta\varepsilon \Delta \varphi + \frac{\beta}{\varepsilon} \psi'(\varphi) 
    \qquad && \text{in } (0,T) \times \Omega, \\
    \label{sys:3}
    \dn \mu  &= \dn \varphi = 0
    \qquad && \text{on }  (0,T) \times \partial \Omega, \\
    \label{sys:4}
    \varphi(0) &=\varphi_0
    \qquad && \text{in } \Omega.
\end{alignat}
\end{subequations}
{Here}
$m(\varphi)$ is the concentration-dependent mobility where $m(\pm 1) = m_{\pm}$, $S_\eps(\varphi)$ is a source term, $\beta > 0$ is a parameter related to surface energy density, and $\varepsilon>0$ is a small length scale proportional to the thickness of the diffuse interface. The function $\psi$ is a suitable double-well potential, {and} $\varphi_0$ serves as the initial condition for $\varphi$. 
%
We assume that $\psi$ 
is the classical quartic potential:
\begin{equation*}
    \psi(r) = \frac{1}{4} (1 - r^2)^2, \quad r \in \mathbb{R}.
\end{equation*}
As for the source term $S_\varepsilon: \mathbb{R} \to \mathbb{R}$, we define for constants $S_+$, $S_-$, $K_+$, $K_-$, and $L$, 
\begin{align*}
    S_\eps(r) &= 
    \begin{cases}
        S_+ -\frac 1 \eps( K_+  (r-1)) & \text{ if } r \geq 1, \\[1ex]
        S_- + G_1(r)(S_+ - S_-) - \frac 1\eps \big(K_- G_2(r) + K_+ G_3(r) - L G_4(r{)}\big) & \text{ if } r \in (-1,1 ), \\[1ex]
        S_- - \frac 1 \eps\big( K_-   (r+1)) & \text{ if } r \leq -1,
    \end{cases}
\end{align*}
%
where for $r \in [-1,1]$,
\begin{align*}
	G_1(r) & =
	\tfrac 34 (r+1)^2- \tfrac 14(r+1)^3,
	\quad 
	G_2(r)  = 
	- \tfrac{1}{4} (1-r^2)(r-1),
	\\
	G_3 (r) & = - G_2(-r),
	\quad
	G_4(r) 
	 = \tfrac{1}{2} (1 - r^2)^2.
\end{align*}
We emphasize from the outset the presence of the $\frac 1\eps$ scaling in the source term associated with the interface thickness, in order to retain the relevant contributions in the asymptotic limit $\eps \to 0$.
Moreover, the constants $\alpha$, $\rho_{\pm}$ and $S_I$ appearing in {\eqref{SharpI}}\ are related to the system parameters of \eqref{eq:sys21} via the relations 
%
  \[
 \alpha=\frac{ \sqrt{2} \beta}3,
 \quad
 \rho_\pm= \frac {K_\pm}{2\beta}, \quad 
S_I = \frac{1}{\sqrt{2}} \Big (K_{+} - K_{-} + \frac{4}{3} L \Big ).
 \]

\section{Radially symmetric solutions}
\label{SEC:RAD}
{In this section, we investigate radially symmetric solutions to the sharp-interface problem \eqref{SharpI}.
}
\subsection{Problem setting}
We now consider the free boundary problem {{\eqref{SharpI}}} in the special domain 
\[
\Omega = B_R(0) = \{ x \in \erre^d \, : \, | x| \leq R \}
\]
which is the $d$-dimensional ball of radius $R>0$ centered at the origin, and seek radially symmetric solutions under the geometry
\begin{align*}
\Omega^+(t)= B_{q(t)}(0) ,\quad \Omega^-(t) = \Omega \setminus \ov{B_{q(t)}(0)}.
\end{align*}
Here, the {scalar} function $q(t)$ encodes the location of the time-dependent interface $\Sigma = \partial B_{q(t)}(0)$. At $p \in \Sigma$, the unit normal, normal velocity and mean curvature are given as
\begin{align*}
	\bnu (p) =- \frac p{|p|} =- \frac p{q(t)},
	\qquad
	{\cal V} = -\dot{q}(t)	,
	\qquad
	\kappa = \frac {d-1}{q(t)},
\end{align*}
with the dot denoting the time derivative.
In this radially symmetric setting, we make the ansatz
\[
\mu(t, x) = \hat{\mu}(t, |x|) = \hat{\mu}(t, r),
\quad r \in (0,R).
\]
In what follows, {for notational convenience, we henceforth {drop} the hat-notation} and use the symbol prime to denote differentiation with respect to $r$. {Under these assumptions}, the system {{\eqref{SharpI}}} reads as
\begin{subequations} \label{Linstab}
\begin{alignat}{3}
	\label{lin:stab:1}
	 m_+ \mu'' + m_+\frac {d-1}r \mu' &=
	-S_+
	+ \rho_+ \mu
	 \quad &&\text{in ${\{{0<r<q(t)}\}}$,}
	\\[1ex]	\label{lin:stab:1:bis}
	 m_-\mu'' + m_-\frac {d-1}r \mu' &=
	-S_-
	 + \rho_- \mu
	 \quad &&\text{in ${\{{q(t)<r<R}\}}$,}
	\\[1ex]
	\label{lin:stab:2}
	 \mu & = \alpha \frac{d-1}{q(t)}
	\quad && \text{on ${\{r=q(t)\}}$,}
	\\[1ex]
	\label{lin:stab:4}
	 2\dot{q} & = -\jump{m\mu'}
	+  S_I
	\quad && \text{on ${\{r=q(t)\}}$,}
	\\[1ex]
	\label{lin:stab:5}
	 \mu_{-}'(t, r) & =0
	\quad && \text{on ${\{r = R\}}$.}
\end{alignat}
\end{subequations}
We further complement the above system with the boundary condition at the origin
\begin{align}
	\label{sharp:in:cond}
	\mu_+{(t,r=0})<\infty.
\end{align}

\subsection{Analytical formula for solutions}
To comprehensively address the analysis in both {the} two and three-dimensional settings {in a unified manner}, for every $\ell \in \mathbb{N} \cup \{0\}$, we introduce the notation
\begin{align}\label{order:ell:Bess}
	\Il(r):=
	\begin{cases}
	I_\ell(r) \quad  & \text{if $d=2$},
	\\
	i_\ell(r) \quad  & \text{if $d=3$},
	\end{cases}
	\quad
	\Kl(r):=
	\begin{cases}
	K_\ell(r) \quad  & \text{if $d=2$},
	\\
	k_\ell(r) \quad  & \text{if $d=3$},
	\end{cases}
	\quad r \in \erre,
\end{align}
where $\{I_\ell,K_\ell\}$ and $\{i_\ell,k_\ell\}$ represent the modified Bessel functions of first and second kind to the order $\ell$ and the spherical modified Bessel functions of first and second kind to the order $\ell$, respectively. It is well-known that $\{I_\ell, K_\ell\}$ forms a set of two linearly independent solutions to the modified Bessel differential equation
\begin{align*}
	r^2 y''(r) + r y'(r) - r^2 y(r) = \ell^2 y(r),
	\quad r \in \erre,
\end{align*}
while $\{i_\ell, k_\ell\}$ forms a set of two linearly independent solutions to the modified spherical Bessel differential equation
\begin{align*}
	r^2 y''(r) + r y'(r) - r^2 y(r) = \ell (\ell+1) y(r),
	\quad r \in \erre.
\end{align*}
For a precise definition and further properties of these Bessel functions, we refer to, e.g., \cite{WatsonBessel}.  It is worth pointing out that $\III_\ell$, $\KKK_\ell$ and $\III'_\ell$ are always nonnegative, whereas $\KKK_\ell'$ is nonpositive. 
{For later use, we collect} some properties and identities involving the zeroth and first order modified Bessel functions:
\begin{subequations} \label{Bessel2d3d}
\begin{align}
	I_0(0)& =1,
	\quad
	 I_0'(r)=I_1(r), \quad \lim_{r \to 0^+} K_{0}(r)  = + \infty, \quad K_{0}'(r)  = - K_1(r),  \label{Bessel:2d} \\
	\label{Bessel:3d:1}
	i_0(r) &= \frac {\sinh(r)}{r},	\quad  i_1(r) =i_0'(r) = \frac {r \cosh(r) - \sinh(r)}{r^2},\\
	\label{Bessel:3d:2}
	k_0(r) & = \frac{e^{-r}}{r}, \quad k_1(r) =-k_0'(r) = \frac {e^{-r}(r+1)}{r^2},
\end{align}
\end{subequations}
to be considered for every $r \in \erre$, whenever it occurs.
Introducing the constants
\[
\lambda_{\pm} := \sqrt{\frac{\rho_{\pm}}{m_{\pm}}},
\]
in Appendix \ref{app:radial:soln} we derive the following analytical {expression} for the solution $\mu$ to {\eqref{Linstab}}\ \begin{align}\label{radial:mu:sol}
	\mu(t,r) & = \begin{dcases}\Big(
		\alpha \frac {d-1}{q(t)}  -\frac{S_+}{\rho_+}
		\Big)
		\frac{\Iz \Lpr}{\Iz (\Lp q(t))} + \frac{S_+}{\rho_+} &  \text{ if } {0 \leq r \leq q(t)}, \\[2ex]
		\Big(
		\alpha \frac {d-1}{q(t)} -\frac{S_-}{\rho_-}
		\Big)
		\frac{\Kz\Lmr + \frac {\KKK_1\LmR}{\III_1\LmR}\Iz\Lmr}{\Kz (\Lm q(t)) + \frac {\KKK_1\LmR}{\III_1\LmR}\Iz (\Lm q(t))}+ \frac{S_-}{\rho_-} & \text{ if } q(t) < r < R,
		\end{dcases}
\end{align}
where we note that \eqref{radial:mu:sol} resembles the forms suggested in (C1) and (C2) of \cite{bauermann2023formation}. The ordinary differential equation satisfied by the radius function $q(t)$ is expressed as
\begin{align}
\non        2 \dot{q}
	 & = -m_+ \Lp  \Big(
	\alpha \frac {d-1}q  -\frac{S_+}{\rho_+}
	\Big) \frac{\III_1 \Lpq}{\Iz \Lpq}
	\\ 
 & \quad
	+m_- \Lm  \Big(
	\alpha \frac {d-1}q  - \frac{S_-}{\rho_-}
	\Big) \frac{\III_1\Lmq \frac {\KKK_1\LmR}{\III_1\LmR} -\KKK_1\Lmq }{\Iz\Lmq  \frac {\KKK_1\LmR}{\III_1\LmR}+ \Kz \Lmq}
	+ S_I
	=: {\cal H} (q).
	\label{radial:ode}
\end{align}


\subsection{Radial solutions  in large and infinite domains}
When examining the expression \eqref{radial:mu:sol} for $\mu$, it becomes evident that in the scenario of increasingly large radius $R$, i.e., as $R \to +\infty$, the ratio $\frac {\KKK_1 \LmR}{\III_1 \LmR}$ tends to zero, indicating that the dominant terms in this limit are those related to $\Kz$.

We aim to demonstrate that with appropriate parameter tuning, the ODE \eqref{radial:ode} for $q(t)$ yields at least two distinct solutions corresponding to stationary states. This can be achieved by qualitatively analyzing the behavior of the right-hand side ${\cal H}(q)$. We formulate this in the following lemma.

\begin{lem}
{Let $R>0$ and 
	the parameters of the problem be given. Then there
exists a constants $d_{+,*} < 0$ which does not depend on $\dplus$ and $\dmin$, and a constant $d_{-,*} > 0$  which does not depend on $\dmin$} such that for all 
$\dplus  < d_{+,*}$ and $\dmin  > d_{-,*}$, the function ${\cal H}$ defined in \eqref{radial:ode} has at least two {distinct} roots in $(0,R)$. Furthermore,  {at least one} root is unstable with respect to radial {perturbations}, while {at least one other} root is stable with respect to radial perturbations.
{If ${\cal H}$ has exactly two roots, then the smaller root is unstable {and the larger one is stable}.}
\end{lem}

\begin{proof}
{We begin by recalling several asymptotic} properties of the modified Bessel functions, see, e.g., {Paragraphs 10.25, 10.30 and 10.52 in \cite{NIST:DLMF}},
\begin{equation}\label{Bessel:prop}
\begin{alignedat}{2}
	\lim_{s \to 0^+ }\frac{\III_1(s)}{\Iz(s)} & = 0,
	\qquad
	&&\lim_{s \to +\infty }\frac{\III_1(s)}{\Iz(s)} = 1,
	\\
	\lim_{s \to 0^+ }\frac{\KKK_1(s)}{\Kz(s)} & = +\infty,
	\qquad
	&& \lim_{s \to +\infty }\frac{\KKK_1(s)}{\Kz(s)} = 1,\qquad 
	\lim_{s \to +\infty}\frac {\KKK_1 (s)}{\III_1 (s)}=0.
\end{alignedat}
\end{equation}
Utilizing the aforementioned properties, one readily {infers} from the definition of ${\cal H}(q)$ that
 \begin{align*}
	\lim_{q\to 0^+} {\cal H}(q)=-\infty,
	\quad
	{\cal H}(R) =	-m_+ \Lp  \Big(
	\alpha \frac {d-1}R  - \frac{S_+}{\rho_+}
	\Big) \frac{\III_1 \LpR}{\Iz \LpR}+S_I.
\end{align*}
{Hence, on choosing $d_{+,*} < 0$ sufficiently
large in modulus, we can guarantee that ${\cal H}(R) < 0$ for
$\frac{S_+}{\rho_+} < d_{+,*}$.} It suffices to show that ${\cal H}$ attains positive values in $(0,R)$, from which we deduce that ${\cal H}$ has at least two roots in $(0,R)$. We can now use the fact that $ \frac {\KKK_1}{\III_1} $ and $ \frac {\KKK_0}{\III_0} $ are positive and {monotonically} decreasing
to observe that in  \eqref{radial:ode} the term
\[
\frac{\III_1\Lmq \frac {\KKK_1\LmR}{\III_1\LmR} -\KKK_1\Lmq }{\Iz\Lmq  \frac {\KKK_1\LmR}{\III_1\LmR}+ \Kz \Lmq} = \frac{\III_1\Lmq}{\III_0\Lmq} \left(\frac{\frac{\KKK_1\LmR}{\III_1\LmR} - \frac{\KKK_1\Lmq}{\III_1\Lmq}}{\frac{\KKK_1\LmR}{\III_1\LmR} + \frac{\KKK_0\Lmq}{\III_0\Lmq}}
\right)\]
is negative {for $q\in(0,R)$}. This implies that for $\dmin $ sufficiently large, ${\cal H}$ attains positive values inside the interval
$(0,R)$. 
{As ${\cal H}$ is analytic this guarantees that ${\cal H}$ has two zeros $q_1$ and $ q_2$ with $0 < q_1 < q_2 < R$ and with ${\cal H}$ changing sign from minus to plus at $q_1$, and vice versa at $q_2$.
Hence it follows from \eqref{radial:ode}}  that $q_1$ is unstable and $q_2$ is stable with respect to radial perturbations.
\end{proof}

We now consider the case of an infinite domain. From the fact that $\lim_{R \to \infty} \frac {\KKK_1 \LmR}{\III_1 \LmR} = 0$, the solution $\mu$ to {\eqref{Linstab}}\ on $(0,\infty)$ is given as
\begin{align}\label{radial:mu:sol:infty}
	\mu(t,r) & = \begin{dcases}\Big(
		\alpha \frac {d-1}{q(t)}  -\frac{S_+}{\rho_+}
		\Big)
		\frac{\Iz \Lpr}{\Iz(\Lp q(t))} + \frac{S_+}{\rho_+} &  \text{ if } {0\leq r \leq q(t)}, \\[2ex]
		\Big(
		\alpha \frac {d-1}{q(t)} -\frac{S_-}{\rho_-}
		\Big)
		\frac{\Kz\Lmr}{\Kz (\Lm q(t))}+ \frac{S_-}{\rho_-} & \text{ if } q(t) < r,
			\end{dcases}
\end{align}
while the ordinary differential equation for $q(t)$ now reads as
\begin{align}
\notag  2 \dot{q} & =
	-m_+ \Lp  \Big(
		\alpha \frac {d-1}q  - \frac{S_+}{\rho_+}
		\Big)
		 \frac{\III_1 \Lpq}{\Iz \Lpq}
		\\
		&\quad - m_-  \Lm  \Big(
		\alpha \frac {d-1}q  - \frac{S_-}{\rho_-}
		\Big) \frac{ \KKK_1\Lmq }{  \Kz \Lmq}
		+ S_I=:{\cal H}_\infty(q).
		\label{ode:combined:Rinf}
\end{align}
Let us remark that $\frac{1}{q} \frac{\III_1\Lpq}{\III_0\Lpq}$ remains nonnegative and bounded for $q \in (0,\infty)$, while $\frac{1}{q} \frac{\KKK_1\Lmq}{\KKK_0\Lmq}$ is positive over $(0,\infty)$ and satisfies 
\[
\lim_{q \to 0^+} \frac{1}{q} \frac{\KKK_1\Lmq}{\KKK_0\Lmq} =+ \infty, \quad \lim_{q \to +\infty} \frac{1}{q} \frac{\KKK_1\Lmq}{\KKK_0\Lmq} = 0.
\]
Hence, in the limit $q \to 0^+$, the second term in \eqref{ode:combined:Rinf} is predominant, leading to 
\begin{align*}
	\lim_{q \to 0^+} {\cal H}_\infty(q)=-\infty.
\end{align*}
On the other hand, as 
\[
\lim_{q \to +\infty} \frac{\III_1\Lpq}{\III_0\Lpq} = 1, \quad \lim_{q \to +\infty}\frac{1}{q} \frac{\III_1\Lpq}{\III_0\Lpq} = 0,
\]
we observe that
\begin{align*}
	\lim_{q \to +\infty} {\cal H}_\infty(q)
	= \underbrace{\frac{m_+  \Lp S_+}{\rho_+}}_{<0} + \underbrace{\frac{m_-  \Lm S_-}{\rho_-}}_{>0} + S_I.
\end{align*}
Consequently, the right-hand side can be made negative {by choosing $\dplus$ {large in modulus}}. On the other hand, one can choose the parameter $\dmin $ large, to achieve that ${\cal H}_\infty$ is positive for
certain bounded values of $q$, {recall \eqref{ode:combined:Rinf}}. As $ {\cal H}_\infty$ is continuous, we hence {guarantee} also in the case of infinite domains that  $ {\cal H}_\infty$ has at least two zeros. Similarly {to before, for a finite number of roots, the smallest} root is unstable, {while the largest sign-changing} root is stable with respect to radial perturbations.

\subsection{Linear stability of radial solutions}
\label{SEC:RAD:STAB}
Let us now examine the linear stability of a radially symmetric solution {about the} stationary state $\qstar$, which is {a} stable root of the function $\mathcal{H}(q)$ derived in \eqref{radial:ode}.
{To this end}, we consider a perturbed radius of the form $w = \qstar + \badeps \rad $, $\badeps \in (0,1)$. Here,  $\rad=\rad(t,\th,\phi)$,  where  $\th$ indicate the polar angle and $\phi$ the azimuthal angle, respectively. In the two dimensional setting,  $\rad=\rad(\th)$ depends on $\th$ only. Besides, we introduce $X(\th,\phi)$ as the standard  parametrization of the unit sphere using
the polar and azimuthal angles,
and indicate the perturbed domain by
\begin{align*}
	\Omega_{\rad,\badeps}^+ & := \{
	x \in \Omega\,\, : \,\,  x = |x| X(\th,\phi), \,
	|x|  < \qstar + \badeps \rad(t,\th,\phi)
	\},
	\\
	\Omega_{\rad,\badeps}^- & := \{
	x \in \Omega \,\, : \,\,  x = |x| X(\th,\phi), \,
	\qstar + \badeps \rad(t,\th,\phi) < |x| <R
	\}.
\end{align*}
We then express the solutions as $\mu_\pm(r,t,\th,\phi) = \mu_\pm^\star(t,r) + \badeps u_{\pm}(r,t,\th,\phi)${, where  $\mu_{\pm}^\star$ are the solutions to \eqref{SharpI} corresponding to the stationary state $\qstar$,}
and demand that those solve the free boundary problem {\eqref{SharpI}}\ on the perturbed domains:
\begin{subequations}
\begin{alignat}{2}
	\label{lin:1}
		- m_+ \Delta (\mustpl+ \badeps u_{+}) &= S_+ - \rho_+ (\mustpl+ \badeps u_{+}) \qquad &&\text{in $\Omega_{\rad,\badeps}^+$,}
	\\
		\label{lin:2}
		- m_- \Delta (\mu_-^\star+ \badeps u_{-}) &= S_- - \rho_- (\mu_-^\star+ \badeps u_{-}) \qquad &&\text{in $\Omega_{\rad,\badeps}^-$,}
	\\
	\label{lin:3}
	 \mu_\pm^\star+ \badeps u_{\pm} &=  \alpha \kappa
	\qquad && \text{on $\{r=w\}$,}
	\\
	\label{lin:5}
	 - 2 {\cal V} & = \jump{m \nabla (\mu^\star + \badeps u)} \cdot {\bnu}
	+ S_I
	\qquad && \text{on $\{r=w\}$,}
	\\
	\label{lin:6}
	 (\mu_-^\star + \badeps u_{-})_-'(R) & =0
	\qquad && \text{on $\{r=R\}$,}
\end{alignat}
\end{subequations}
where $\kappa$, $\mathcal{V}$ and $\bm{\nu}$ denote the mean curvature, normal velocity and unit normal of the perturbed interface $\{ r = w \}$, respectively. In the above{,} we introduce indices $\pm$ to represent the solution over $\Omega_{Y,\badeps}^{\pm}$, respectively. Linearizing the above equations about the original interface $\{r = \qstar\}$, while using the {expansions}
\begin{align*}
	\mu_\pm^\star(w) =
	\mu_\pm^\star(\qstar)
	+ (\mu_\pm^\star)'|_{r=\qstar} (w-\qstar) + \text{ h.o.t.}=
	\mu_\pm^\star(\qstar)
	+ \badeps (\mu_\pm^\star)'|_{r=\qstar} Y +\text{ h.o.t.},
\end{align*}
where h.o.t.~denotes higher order terms, we obtain the following system for $\rad$ and $u_{\pm}$:
\begin{subequations} \label{pertrad:general}
\begin{alignat}{2}
	\label{pertrad:general:1}
	&
	- m_+ \Delta  u_+=  - \rho_+ u_+ \qquad &&\text{in $\{{0 <r <\qstar}\}$,}
	\\
		\label{pertrad:general:2}
	&
	- m_- \Delta u_-= - \rho_- u_- \qquad &&\text{in $\{{ \qstar<r <R}\}$,}
	\\
	\label{pertrad:general:3}
	& (\mu^\star_\pm)' Y + u_{\pm}= -  \alpha 	\frac {d-1}{(\qstar)^2} \big( \rad + \frac 1{d-1} \Delta_{{\cal S}^{d-1}} \rad \big)
	\qquad && \text{on $\{r=\qstar\}$,}
	\\
	&  2 \dot{Y} =
	-m_+(\mu^\star_+)'' \rad
	+m_-(\mu^\star_-)'' \rad
	\label{pertrad:general:5}
	- m_+ (u_+)_r
	+ m_- (u_-)_r
	\qquad && \text{on $\{r=\qstar\}$,}
	\\
	\label{pertrad:general:6}
	& {(u_-)_r =0}
	\qquad &&
	\text{on $\{r=R\}$,}
\end{alignat}
\end{subequations}
where $\Delta_{{\cal S}^{d-1}} $ indicates the Laplace--Beltrami operator on the  $d-1$ dimensional unit sphere ${\cal S}^{d-1}$.
Here, we use  the linearization ${\tfrac {d-1}{(\qstar)^2}} \big( \rad + {\tfrac 1{d-1}} \Delta_{{\cal S}^{d-1}} \rad \big)$ of the mean curvature operator around a sphere, which has been computed for
example in Lemma 3.1 of \cite{ESCHER1998267} (that defines the mean curvature with an additional factor $\tfrac 1{d-1}$) or can be computed with the help of Lemma 39 in \cite{bgnreview}.
We further complement the above system with the additional boundary condition at the origin
\begin{align}
	\label{pertrad:general:initial}
	u_+(r=0)<\infty.
\end{align}
We want to solve the linear evolution equation on a surface given by \eqref{pertrad:general:5} possessing the form $ 2 \dot{Y} = {\cal G}(\rad)$ with ${\cal G}(\rad)$ being defined by solving \eqref{pertrad:general:1}--\eqref{pertrad:general:3}, \eqref{pertrad:general:6}, and \eqref{pertrad:general:initial}. We proceed with the ansatz
\[
\rad = \delta (t) Z(\th,\phi),
\]
where for $\ell \in \enne \cup \{ 0 \}$ and $ k \in \{-\ell , ..., \ell \}$, we set
\begin{alignat}{2}	
Z(\th,\phi)& =
	\cos(\ell \th),
	\; \text{ or }\;
	Z(\th,\phi) =
	\sin(\ell \th),
	 \quad &&\text{if $d=2$,} \label{perturb2d}
	\\
	Z(\th,\phi)& =
	Y_{\ell,k}(\th,\phi), \quad &&\text{if $d=3$.}	
	\label{perturb3d}
\end{alignat}
Here, the functions $Y_{\ell,k}(\th,\phi)$ denote  spherical harmonics, where $\ell$ and $\th$ indicate the polar wavenumber and angle, while $k$ and $\phi$ denote the azimuthal wavenumber and angle, respectively.
Inspired by the above ansatz, we consider $u_\pm$ taking the following form
\begin{align*}
	 u_\pm (r,t, \th,\phi)= \delta (t) Z(\th,\phi)U_\pm (r),
\end{align*}
where $U_\pm$ fulfill a similar system that we will present below. We recall that
the Laplace--Beltrami operator $\Delta_{{\cal S}^{d-1}}$ on the $(d-1)$-dimensional unit sphere $\mathcal{S}^{d-1}$ is defined as
\begin{align*}
	\Delta_{{\cal S}^{d-1}} = \begin{cases}
	\pd^2_\th, \quad &\text{if $d=2$},
	\\
	\pd^2_\th + \cot(\th) \pd_\th + \frac 1{\sin^2(\th)} \pd^2_\phi, \quad & \text{if $d=3$},
	\end{cases}
\end{align*}
and for a general function $f$ it holds that
\begin{align}
	\label{Delta:rad}	
	\Delta f = f'' + \frac {d-1}r f' + \frac 1{r^2} \Delta_{{\cal S}^{d-1}} f,
\end{align}
following the same convention that the radial derivative is denoted by a prime.
For the above choices of $Z(\th,\phi)$ we infer
\begin{align}	\label{lap}
	\Delta_{{\cal S}^{d-1}}Z(\th,\phi)
	=\zeta_{\ell,d}Z(\th,\phi) ,
	\quad \text{with} \quad
	\zeta_{\ell,d} = \begin{cases}
	-\ell^2, \quad &\text{in $d=2$},
	\\
	-\ell(\ell+1), \quad & \text{in $d=3$}.
	\end{cases}
\end{align}
Then, the system {\eqref{pertrad:general} together with
\eqref{pertrad:general:initial}} simplifies to
\begin{subequations}\label{linstab}
\begin{alignat}{2}
	\label{linstab:1}
	&
	- m_+ \Big( U_{+}'' + \frac {d-1}r U_{+}' + \frac {\zeta_{\ell,d}}{r^2}  U_{+} \Big)=  - \rho_+  U_{+} \quad &&\text{in $\{{0<r<\qstar}\}$,}
	\\
		\label{linstab:2}
	&
	- m_- \Big( U_{-}''  + \frac {d-1}r U_{-}' + \frac {\zeta_{\ell,d}}{r^2} \ U_{-} \Big)=  - \rho_-  U_{-} \quad &&\text{in $\{{\qstar<r <R}\}$,}
	\\
	\label{linstab:3}
	& (\mu^\star_+)'
	+ U_+ = (\mu^\star_-)' + U_- = -\alpha \frac {d-1}{r^2} \Big( 1 + \frac {\zeta_{\ell,d}}{d-1}\Big)
	\quad && \text{on $\{r=\qstar\}$,}
	\\
	\label{linstab:4}
	& 2 \dot{\delta}
	=
	-\big(m_+(\mu^\star_+)''
	- m_-(\mu^\star_-)'' \big)\delta
	-   \big(m_+U_+' -m_-U_-'\big)\delta
	\quad && \text{on $\{r=\qstar\}$,}
	\\
	\label{linstab:5}
	& U_-' =0
	\quad && \text{on $\{r=R\}$,} 
\end{alignat}
\end{subequations}
and
\begin{align}
 \label{linstab:6} & U_+(r = 0) < \infty.
\end{align}
As a consequence of equation \eqref{linstab:3}, we also have the following relation
\begin{align}
	\label{linstab:extra}
	& U_- - U_+ =(\mu_+^\star)'
	- (\mu_-^\star)'
	\quad  \text{on $\{r=\qstar\}$.}
\end{align}

\subsubsection{Solutions to the perturbed system on finite domains}
In Appendix \ref{app:radial:perturb} we detail the derivation of the following analytical {formulas} for $U_{+}$ and $U_{-}$. For $r \in (0,\qstar)$:
\begin{align*}
	U_+ (r) & =
	\left (- \alpha \frac {d-1}{(\qstar)^2} \Big( 1+ \frac {\zeta_{\ell,d}}{d-1}\Big)
	- \Lp   \Big( \alpha \frac {d-1}\qstar - \frac{S_+}{\rho_+} \Big)  \frac {\III_1\Lpqs}{\Iz \Lpqs}  \right ) \frac{\III_\ell(\Lp  r)}{\Il \Lpqs}
\end{align*}
and for $r \in (\qstar,R)$:
\begin{align*}
U_- (r)  & =
	\frac{\left ( \KKK_\ell(\Lm r) - \frac{\Kl' \LmR }{\Il' \LmR}  \III_\ell(\Lm r) \right )}{ \left (\Kl \Lmqs- \frac{\Kl' \LmR }{\Il' \LmR } \Il \Lmqs \right)} \\
	& \quad \times 	\left (- \alpha \frac {d-1}{(\qstar)^2} \Big( 1+ \frac {\zeta_{\ell,d}}{d-1}\Big)
	+ \Lm \Big(\alpha \frac {d-1}q-  \frac{S_-}{\rho_-} \Big) 
	\frac {\KKK_1\Lmqs -  \frac {\KKK_1 \LmR}{\III_1 \LmR} \III_1\Lmqs }{ \Kz \Lmqs +  \frac {\KKK_1 \LmR}{\III_1 \LmR}\Iz \Lmqs }
	 \right ).
\end{align*}
These expressions in turn {provide} the ordinary differential equation for the perturbation magnitude $\delta(t)$:
\begin{align}
\notag  \frac{2\dot{\delta}}{\delta}   & =	-m_+ \Lp^2  \Big( \alpha \frac {d-1}{\qstar} - \frac{S_+}{\rho_+} \Big)  \frac{1}{\Iz \Lpqs} \Big ( \III_1' \Lpqs - \III_1 \Lpqs \frac{\Il' \Lpqs}{\Il \Lpqs} \Big ) \\
\notag & \quad  + m_+ \Lp \alpha \frac {d-1}{(\qstar)^2} \Big( 1+ \frac {\zeta_{\ell,d}}{d-1}\Big)
\frac{\Il' \Lpqs}{\Il \Lpqs}\\ 
\label{radial:ode:perturb} & \quad - m_-\Lm^2 \Big (\alpha \frac {d-1}{\qstar}-  \frac{S_-}{\rho_-} \Big)
	\frac { \KKK_1' \Lmqs - \frac{\KKK_1 \LmR}{\III_1 \LmR} \III_1' \Lmqs}{ \Kz \Lmqs + \frac {\KKK_1 \LmR}{\III_1 \LmR} \Iz \Lmqs }\\ 
\notag & \quad	+ m_-\Lm  \frac{\Kl' \Lmqs - \frac{\Kl'\LmR}{\Il' \LmR} \Il' \Lmqs}{\Kl \Lmqs-  \frac{\Kl' \LmR }{\Il' \LmR }) \Il \Lmqs} \\
\notag & \qquad \times \left (- \alpha \frac {d-1}{(\qstar)^2} \Big( 1+ \frac {\zeta_{\ell,d}}{d-1}\Big)
	+ \Lm \Big(\alpha \frac {d-1}{\qstar}- \frac{S_-}{\rho_-} \Big)
	\frac {\KKK_1\Lmqs -  \frac {\KKK_1 \LmR}{\III_1 \LmR} \III_1\Lmqs}{ \Kz \Lmqs + \frac {\KKK_1 \LmR}{\III_1 \LmR} \Iz \Lmqs} \right ).
\end{align}
We can then express \eqref{radial:ode:perturb} as
\begin{align}
\frac{2\dot{\delta}}{\delta} &
	=\Big(\alpha \frac {d-1}{(\qstar)^2} \Big( 1+ \frac {\zeta_{\ell,d}}{d-1}\Big)\Big)F_1(\qstar)
	  + \Big(
	  \alpha \frac {d-1}{\qstar}  -\frac{S_+}{\rho_+}
	  \Big) F_2(\qstar)
   +\Big(
	  \alpha \frac {d-1}{\qstar}  -\frac{S_-}{\rho_-}
	  \Big)  F_3(\qstar), \label{strucperturb}
\end{align}	
with {the} natural definition of $F_1$, $F_2$ and $F_3$ {as} functions of $\qstar$ that are independent of $\alpha$, $S_+$ and $S_-$.

\begin{remark}
	We mention that the term $\frac {\zeta_{\ell,d}}{d-1}$ takes very large negative values for $\ell$ large and
	in such cases will  dominate in \eqref{strucperturb}. This {leads to the} decay of  perturbations {correspond{ing}} to mode $\ell$. The question now is under what parameter values
	will the right hand side be positive for certain $\ell$.
This would than lead
	to the situation that a perturbation $Z$ as in \eqref{perturb2d} and \eqref{perturb3d} will increase.
	This yields a
	deformation of the sphere with a number of oscillations  depending on $\ell$ that grows in time.
	Heuristically, taking small values of $\alpha$ {will} lead to more unstable situations.
	\end{remark}

\subsubsection{Solutions to the perturbed system on infinite domains}
Let us explore the behaviour in the limit $R \to \infty$. We recall the properties 
\begin{align*}
\lim_{R \to \infty} \frac{\KKK_1 \LmR}{\III_1 \LmR} = 0, \quad \lim_{R \to \infty} \frac{\KKK_{\ell}' \LmR}{\III_{\ell}' \LmR}= 0,
\end{align*}
so that in the limit $R \to \infty$, the ordinary differential equation \eqref{radial:ode:perturb} simplifies to 
\begin{equation}\label{radial:ode:infty:perturb}
\begin{aligned}
	\frac{2 \dot{\delta}}{\delta}
	&=
	m_+  \Lp    \alpha  \frac {d-1}{(\qstar)^2} \Big( 1+ \frac{\zeta_{\ell,d}}{d-1}\Big) \frac{\Il' \Lpqs}{\Il \Lpqs}
	- m_-  \Lm   \alpha  \frac {d-1}{(\qstar)^2}
	  \Big( 1+ \frac{\zeta_{\ell,d}}{d-1}\Big) \frac{\Kl' \Lmqs}{\Kl \Lmqs}
	\\ & \quad
	- m_+  \Lp^2  \Big(  \alpha \frac {d-1}{\qstar} - \frac{S_+}{\rho_+}\Big) \frac1 {\Iz \Lpqs} \Big( { \III_1'\Lpqs - \III_1 \Lpqs  \frac{\Il' \Lpqs}{\Il \Lpqs} }\Big)
	\\ & \quad
	- m_- \Lm^2  \Big(\alpha \frac {d-1}{\qstar}- \frac{S_-}{\rho_-} \Big)\frac1 {\Kz \Lmqs} \Big( { \KKK_1'\Lmqs - \KKK_1 \Lmqs  \frac{\Kl' \Lmqs}{\Kl \Lmqs} }\Big).
\end{aligned}
\end{equation}
When $\ell = 1$, we recall that $\frac{\zeta_{1,d}}{d-1} = -1$ for $d = 2, 3$, and so the above expression simplifies to
\begin{align*}
\frac{2 \dot{\delta}}{\delta}
	& =
	m_+  \Lp    \alpha  \frac {d-1}{(\qstar)^2} \Big( 1+ \frac{\zeta_{1,d}}{d-1}\Big) \frac{\III_1' \Lpqs}{\III_1 \Lpqs}
	- m_-  \Lm   \alpha  \frac {d-1}{(\qstar)^2}
	  \Big( 1+ \frac{\zeta_{1,d}}{d-1}\Big) \frac{\KKK_1' \Lmqs}{\KKK_1 \Lmqs} \\
	  & = \alpha \frac{d-1}{(\qstar)^2} \Big ( 1 + \frac{\zeta_{1,d}}{d-1} \Big ) \Big [ m_+ \lambda_+ \frac{\III_1' \Lpqs}{\III_1 \Lpqs} - m_- \lambda_- \frac{\KKK_1' \Lmqs}{\KKK_1 \Lmqs} \Big ] = 0.
\end{align*}
This is due to the fact that these perturbations corresponding to $\ell = 1$ relate to translations of the sphere which{,} in an infinite domain{,} do not lead to any instability as the translated sphere is stationary as well. For $\ell \geq 2$ we see that
\[
\alpha \frac{d-1}{(\qstar)^2} \Big ( \underbrace{1 + \frac{\zeta_{\ell,d}}{d-1}}_{<0} \Big ) \Big [ \underbrace{m_+ \lambda_+ \frac{\III_\ell' \Lpqs}{\III_\ell \Lpqs} - m_- \lambda_- \frac{\KKK_\ell' \Lmqs}{\KKK_\ell \Lmqs} }_{>0} \Big ] < 0,
\] 
while recalling $S_+ < 0$ and using the property $\III_\ell \III_1' - \III_1 \III_\ell' < 0$ from Corollary \ref{cor:Bessel}{,} we have
\[
- m_+  \Lp^2  \Big(  \underbrace{\alpha \frac {d-1}{\qstar} - \frac{S_+}{\rho_+}}_{>0} \Big) \frac1 {\Iz \Lpqs} \Big( \underbrace{ \III_1'\Lpqs - \III_1 \Lpqs  \frac{\Il' \Lpqs}{\Il \Lpqs} }_{<0} \Big) > 0.
\]
Furthermore, for $\alpha$ sufficiently small and $S_-$ sufficient large, using the property $\KKK_\ell \KKK_1' - \KKK_1 \KKK_{\ell}' > 0$ from Corollary \ref{cor:Bessel}, it holds that
\[
- m_- \Lm^2  \Big( \underbrace{\alpha \frac {d-1}{\qstar}- \frac{S_-}{\rho_-}}_{<0} \Big)\frac1 {\Kz \Lmqs} \Big( \underbrace{ \KKK_1'\Lmqs - \KKK_1 \Lmqs  \frac{\Kl' \Lmqs}{\Kl \Lmqs} }_{>0} \Big) > 0.
\]
Hence, for $S_+ < 0$, sufficiently small $\alpha > 0$ and sufficiently large $S_- > 0$, the right-hand side of \eqref{radial:ode:infty:perturb} can be made positive, leading to the enhancement of perturbations.

\section{Multilayered type solutions}
We now turn to discussing the multilayered  situation, first considering the flat case and then the radial case.

\subsection{Flat multilayered {solutions}}
For $\mathcal{L}, \widetilde{\mathcal{L}} > 0$, we consider the special domain
\[
\Omega = (0, {\cal L}) \times (0, \Lpanarv)^{d-1},
\]
and seek planar solutions under the geometry 
\begin{align*}
	\Omega_-^{\rm in}(t)& =(0,q_1(t)) \times (0, \Lpanarv)^{d-1}, \\
	\Omega_+(t) & =(q_1(t),q_2(t)) \times (0, \Lpanarv)^{d-1}, 
	\\ 
	\Omega_-^{\rm ext}(t)& =(q_2(t),\Lpanarh) \times (0, \Lpanarv)^{d-1},
\end{align*}
with $q_1(t) < q_2(t)$ encoding the location of the moving interfaces with normal velocities and normals given by ${\cal V}_1 = \dot{q}_1$, ${\cal V}_2 =  - \dot{q}_2$, $\bm{\nu}_1 = (1, \mathbf{0})^{\top}$ and $\bm{\nu}_2 = (-1, \mathbf{0})^{\top}$, respectively. In particular, the unit normals point towards $\Omega_+$ from $\Omega_-^{\rm in}$ and $\Omega_-^{\rm ext}$, see Figure \ref{fig:planar} for a sketch.

\begin{figure}[h]
\centering
\begin{tikzpicture}[scale=1.2]
  \draw[thick] (0,0) rectangle (8,4);

  \draw[thick] (2.5,0) -- (2.5,4);
  \draw[thick] (5.5,0) -- (5.5,4);
  
  \filldraw[fill=black!25, thick, line join=bevel]
 	(2.5,0) rectangle (5.5,4);

  \node at (1.25,2) {$\Omega^{\textrm{in}}_-(t)$};
  \node at (4,2) {$\Omega_+(t)$};
  \node at (7,2) {$\Omega^{\textrm{ext}}_-(t)$};

  \node[rotate=90] at (8.3,3.5) {$\partial \Omega$};

  \draw[->] (2.5,1.5) -- (3.3,1.5) node[midway, below] {$\boldsymbol{\nu}_1(t)$};
  \draw[->] (5.5,1.5) -- (4.8,1.5) node[midway, below] {$\boldsymbol{\nu}_2(t)$};
  \draw[->] (8,1.5) -- (8.8,1.5) node[midway, below] {$\boldsymbol{n}$};

  \node at (2.5,-0.5) {$q_1(t)$};
  \node at (5.5,-0.5) {$q_2(t)$};
\end{tikzpicture}
\caption{The geometric setting for the flat multilayered solutions. The unit normals point into {$\Omega_+(t)$}.}
\label{fig:planar}
\end{figure}
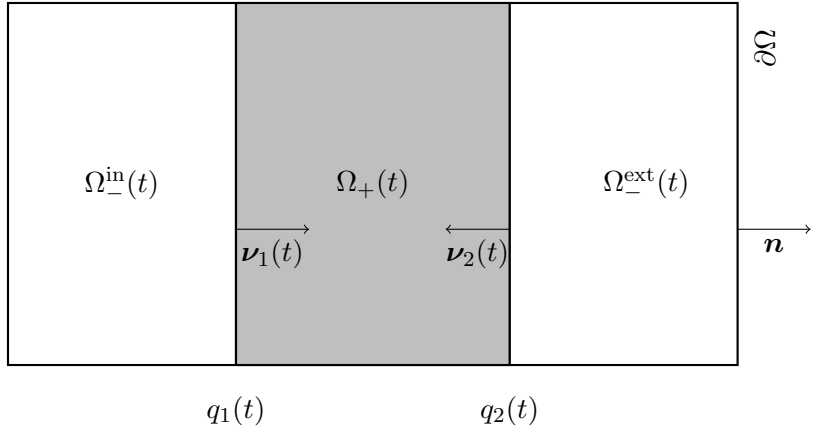

For $x \in \RRR^d$, we write $x = (z, \widehat{x})$ with $z \in \RRR$ and $\widehat{x} \in \RRR^{d-1}$. We look for planar solutions $\mu_-^{\rm in}$, $\mu_+$ and $\mu_-^{\rm ext}$, that are just function of time $t$ and $z$, solving 
\begin{subequations}\label{Flatmultil}
\begin{alignat}{2}
	-m_- (\mumin)'' &=S_-  -\rho_-\mumin \qquad && \text{for } z\in (0,q_1(t)),
	\\
	-m_+ \mu''_+ &=S_+  -\rho_+ \mu_+ \qquad && \text{for } z\in (q_1(t),q_2(t)) ,
	\\ 
	-m_- (\mumex)'' &=S_-  -\rho_- \mumex \qquad && \text{for } z\in (q_2(t),\Lpanarh),
	\\
	\mumin &  =0=\mu_+ &&\text{for } \{z=q_1(t) \},
	\\
	\mu_+ & =0 =\mumex &&\text{for } \{z=q_2(t) \},
	\\
	-2 {\dot q_1} &=  \jump{m \mu'} +S_I \qquad && \text{for } \{z=q_1(t) \},
	\\
	2 {\dot q_2} &= - \jump{m \mu'} +S_I \qquad && \text{for } \{z=q_2(t) \}, 
	\end{alignat}
\end{subequations}
	along with the Neumann boundary conditions
	\[
		(\mumin)^\prime (t,0)  =0, \quad 
	(\mumex)^\prime (t,\Lpanarh) =0.
	\]
The sign differences in the ODEs for $q_1 $ and $q_2$ arise from the differing orientations of the normals, and we recall that the mean curvature $\kappa$ of a planar interface is zero.

\subsubsection{Analytical formula for solutions}
Following similar calculations as in \cite{OOL} we arrive at the solution formula
\begin{subequations} \label{mum:thin:lay}
\begin{alignat}{2}
	\label{mumin:thin:lay}
\mumin(t,z) &= \dmin  \left(1 - \frac{\cosh(\Lm z\big)}{\cosh(\Lm q_1(t)\big)} \right),
\quad && z \in (0, q_1(t)),
\\
	\label{mumpl:thin:lay}
\mu_+(t,z) &= \dplus  \left(1 - \frac{
	\cosh\left( \Lp \left( \tfrac{q_1(t)+q_2(t) }{2} - z\right)
	\right)
}{
	\cosh\left( \Lp \tfrac{q_2(t) - q_1(t)}{2} \right)
} \right),
\quad && z \in (q_1(t), q_2(t)),
\\
	\label{mumext:thin:lay}
	\mumex(t,z) &= \dmin  \left(1 - \frac{\cosh(\Lm (\Lpanarh - z))}{\cosh(\Lm (\Lpanarh- q_2(t)))} \right),
\quad && z \in (q_2(t), \Lpanarh),
\end{alignat}
\end{subequations}
along with the following evolution equations for the interface locations $q_1$ and $q_2$:
\begin{subequations} \label{q12dot:planar}
\begin{align} 
\label{q1dot:planar}  	 \dot{q}_1 &= \frac{1}{2}(- m_+ \mu_+'(q_1) 
+m_- (\mumin)'(q_1)  - S_I)
\\
\notag &= \frac{1}{2} (-\dplus m_+  \Lp \tanh\left( \Lp \tfrac{q_2 - q_1}{2} \right)
- \dmin   m_- \Lm \tanh\big( \Lm q_1 \big)
- S_I ) =: {\cal H}^{\rm p}_1\qud,
\\
\label{q2dot:planar}  	 \dot{q}_2 &= \frac{1}{2}(- m_+ \mu_+'(q_2) + m_- (\mumex)'(q_2)  + S_I)
\\
\notag &= \frac{1}{2}(\dplus m_+  \Lp \tanh\left( \Lp \tfrac{q_2 - q_1}{2} \right)+ \dmin m_-  \Lm \tanh\big( \Lm (\Lpanarh - q_2) \big)
+ S_I)  =: {\cal H}^{\rm p}_2\qud.
\end{align}
\end{subequations}

{We will now show the existence of a stationary solution of the ODE system \eqref{q12dot:planar} and discuss its stability with respect to {translational} perturbations.}

\begin{prop}
Let $S_I = 0$. The {triangular set } $\mathcal{T}:= \{(q_1, q_2) \in \RRR^2 \, : \, 0 \leq q_2 \leq \mathcal{L}, \, 0 \leq q_1 \leq q_2 \}$ is an invariant {set} for the ODE system {\eqref{q12dot:planar}}, i.e., if $(q_1(0), q_2(0)) \in \mathcal{T}$, then any solution $(q_1(t), q_2(t))_{t > 0}$ to {\eqref{q12dot:planar}} emanating from the initial condition $(q_1(0), q_2(0))$ satisfies $(q_1(t), q_2(t)) \in \mathcal{T}$ for all $t > 0$.
\end{prop}

\begin{proof}
We consider the decomposition of the boundary $\pd \mathcal{T}$ into $\pd \mathcal{T}_1 := \{ (0, q_2) \, : \, 0 \leq q_2 \leq \mathcal{L}\}$, $\pd \mathcal{T}_2 := \{ (q_1, \mathcal{L}) \, : \, 0 \leq q_1 \leq \mathcal{L}\}$ and $\pd \mathcal{T}_3 := \{ (q_1, q_1) \, : \, 0 \leq q_1 \leq \mathcal{L}\}$.  We say that $\bm{n}(\bm{z})$ is an outer normal to $\mathcal{T}$ at $\bm{z} \in \pd \mathcal{T}$ if the the open ball $B$ with center $\bm{z} + \bm{n}(\bm{z})$ and radius $|\bm{n}(\bm{z})|$ has empty intersection with $\mathcal{T}$. We then verify the tangent condition
\begin{align}\label{tangent:cond}
\bm{n}((q_1,q_2)) \cdot \begin{pmatrix} \mathcal{H}_1^{\rm p}(q_1, q_2) \\ \mathcal{H}_2^{\rm p}(q_1, q_2) \end{pmatrix} \leq 0 \quad \forall (q_1, q_2) \in \pd \mathcal{T},
\end{align}
where $\bm{n}((q_1, q_2))$ is an outer normal to $\mathcal{T}$ at $(q_1, q_2)$, so that by classical results for ordinary differential systems (see, e.g., {\cite[Chapter 10, page 118]{Walter}}) we infer that $\mathcal{T}$ is invariant under the dynamics of the ODE system {\eqref{q12dot:planar}}.

On $\pd \mathcal{T}_1$ the outer normal for $\{(0, q_2) \, : \, 0 < q_2 < \mathcal{L}\}$ is $(-1, 0)^{\top}$, and so
\begin{align*}
\begin{pmatrix} -1 \\ 0 \end{pmatrix} \cdot \begin{pmatrix} \mathcal{H}_1^{\rm p}(q_1, q_2) \\ \mathcal{H}_2^{\rm p}(q_1, q_2) \end{pmatrix}  = -
\mathcal{H}_1^{\rm p}(q_1, q_2) &= -\frac{1}{2}(- \dplus  m_+ \lambda_+ \tanh (\lambda_+ \tfrac{q_2}{2})) \leq  0.
\end{align*}
On $\pd \mathcal{T}_2$ the outer normal for $\{(q_1, \mathcal{L}) \, : \, 0 < q_1 < \mathcal{L}\}$ is $(0,1)^{\top}$, and so 
\begin{align*}
\begin{pmatrix} 0 \\ 1 \end{pmatrix} \cdot \begin{pmatrix} \mathcal{H}_1^{\rm p}(q_1, q_2) \\ \mathcal{H}_2^{\rm p}(q_1, q_2) \end{pmatrix}  = 
\mathcal{H}_2^{\rm p}(q_1, q_2) &= \frac{1}{2}(\dplus  m_+\lambda_+ \tanh (\lambda_+ \tfrac{\mathcal{L} - q_1}{2})) \leq 0.
\end{align*}
On $\pd \mathcal{T}_3$ the outer normal for $\{(q_1, q_1) \, : \, 0 < q_1 < \mathcal{L}\}$ is $(1, -1)^{\top}$, and so 
\begin{align*}
& \begin{pmatrix} 1 \\ -1 \end{pmatrix} \cdot \begin{pmatrix} \mathcal{H}_1^{\rm p}(q_1, q_2) \\ \mathcal{H}_2^{\rm p}(q_1, q_2) \end{pmatrix}   = \mathcal{H}_1^{\rm p}(q_1, q_2) - \mathcal{H}_2^{\rm p}(q_1, q_2) \\
& \quad = \frac{1}{2}(-\dmin  m_- \lambda_- \tanh(\lambda_- q_1)) - \frac{1}{2}(\dmin  m_- \lambda_- \tanh (\lambda_-(\mathcal{L} - q_1)) \leq 0.
\end{align*}
On the corners $\{(0,0)\}$, $\{(0, \mathcal{L})\}$ and $\{(\mathcal{L}, \mathcal{L})\}$ of $\mathcal{T}$ there are more than one choice of outer normal. Take for instance $\{(0,0)\}$, where the set of outer normals at $\{(0,0)\}$ can be expressed as $\{(\cos \theta, \sin \theta) \,: \, \theta \in (\pi, \frac{5\pi}{4})\} $. Then, we compute at $(q_1, q_2) = (0,0)$
\begin{align*}
\begin{pmatrix} \cos \theta \\ \sin \theta \end{pmatrix} \cdot \begin{pmatrix} \mathcal{H}_1^{\rm p}(q_1, q_2) \\ \mathcal{H}_2^{\rm p}(q_1, q_2) \end{pmatrix}  = \frac{1}{2} (\sin \theta ) (\dmin  m_- \lambda_- \tanh(\lambda_- \mathcal{L})) \leq 0,
\end{align*}
as $\sin \theta \leq 0$ for $\theta \in (\pi, \frac{5\pi}{4})$. Similarly, at $(q_1, q_2) = (0, \mathcal{L})$, the set of outer normals can be expressed as $\{(\cos \theta, \sin \theta) \, : \, \theta \in (\frac{\pi}{2}, \pi)\}$, and so 
\begin{align*} \begin{pmatrix} \cos \theta \\ \sin \theta \end{pmatrix} \cdot \begin{pmatrix} \mathcal{H}_1^{\rm p}(q_1, q_2) \\ \mathcal{H}_2^{\rm p}(q_1, q_2) \end{pmatrix} = \frac{1}{2} ( \sin \theta -\cos \theta) (\dplus  m_+ \lambda_+ \tanh(\lambda_+ \tfrac{\mathcal{L}}{2})) \leq 0.
\end{align*}
Lastly, at $(q_1, q_2) =( \mathcal{L}, \mathcal{L})$, the set of outer normals can be expressed as $\{(\cos \theta, \sin \theta) \, : \, \theta \in (\frac{5\pi}{4}, \frac{5 \pi}{2})\}$, and so 
\begin{align*} \begin{pmatrix} \cos \theta \\ \sin \theta \end{pmatrix} \cdot \begin{pmatrix} \mathcal{H}_1^{\rm p}(q_1, q_2) \\ \mathcal{H}_2^{\rm p}(q_1, q_2) \end{pmatrix} = \frac{1}{2} ( \cos \theta) (-\dmin  m_- \lambda_- \tanh(\lambda_- \mathcal{L}) \leq 0.
\end{align*}
This verifies the tangent condition \eqref{tangent:cond}, and hence $\mathcal{T}$ is an invariant {set} under the dynamics of the ODE system {\eqref{q12dot:planar}}.
\end{proof}

{We next show that the planar ODE system \eqref{q12dot:planar} admits an underlying gradient flow  structure.}

\begin{prop}
The ODE system {\eqref{q12dot:planar}} is a gradient flow of the function 
\begin{align*}
\mathcal{E}(q_1,q_2) & :=  \frac{\dmin  m_-}{2} \ln \cosh ( \lambda_- q_1) +  \frac{\dmin  m_-}{2} \ln \cosh (\lambda_- (\mathcal{L} - q_2)) \\
& \quad - \dplus  m_+ \ln \cosh (\lambda_+ \frac{q_2 - q_1}{2}) + \frac{1}{2}S_I (q_1 - q_2),
\end{align*}
which is a strictly convex Lyapunov function.
\end{prop}

\begin{proof}
A direct calculation shows that 
\begin{align*}
\frac{\pd \mathcal{E}}{\pd q_1} & = \frac{1}{2}(\dmin  m_- \lambda_- \tanh(\lambda_- q_1) + \dplus  m_+ \lambda_+ \tanh(\lambda_+ \tfrac{q_2 - q_1}{2}) + S_I) = - \mathcal{H}_1^{\rm p}(q_1, q_2), \\
\frac{\pd \mathcal{E}}{\pd q_2} & = -\frac{1}{2}(\dmin  m_- \lambda_- \tanh(\lambda_-(\mathcal{L} - q_2)) + \dplus  m_+ \lambda_+ \tanh(\lambda_+ \tfrac{q_2 - q_1}{2}) + S_I) = - \mathcal{H}_2^{\rm p}(q_1,q_2).
\end{align*}
We compute the entries of the Hessian matrix $D^2 \mathcal{E}$ of $\mathcal{E}$:
\begin{align*}
(D^2 \mathcal{E})_{11} & =
\frac{1}{2} \dmin  m_- \lambda_-^2 \mathrm{sech}^2(\lambda_- q_1) - \frac{1}{4} \dplus  m_+ \lambda_+^2 \mathrm{sech}^2(\lambda_+ \tfrac{q_2 - q_1}{2}), \\
(D^2 \mathcal{E})_{12} & = (D^2 \mathcal{E})_{21} = \frac{1}{4} \dplus  m_+ \lambda_+^2 \mathrm{sech}^2(\lambda_+ \tfrac{q_2 -q_1}{2}), \\
(D^2 \mathcal{E})_{22} & = \frac{1}{2} \dmin  m_- \lambda_-^2 \mathrm{sech}^2(\lambda_-(\mathcal{L} - q_2)) - \frac{1}{4} \dplus  m_+ \lambda_+^2 \mathrm{sech}^2(\lambda_+ \tfrac{q_2 - q_1}{2}).
\end{align*}
{Since $S_+$ is negative, while all the other parameters are positive,} we have that $\mathrm{tr}(D^2 \mathcal{E}) > 0 $ and $\det(D^2 \mathcal{E}) > 0$, and hence the two eigenvalues of $D^2 \mathcal{E}$ are positive. Thus, $\mathcal{E}$ is strictly convex.
\end{proof}

We say that $(\qstaru,\qstard)$ is a symmetric stationary solution if $\mathcal{H}_1^{\rm p}(\qstaru, \qstard) = 0$, $\mathcal{H}_2^{\rm p}(\qstaru, \qstard)$ and $\qstaru = \mathcal{L} - \qstard$. A symmetric stationary solution corresponds to the setting where $\Omega_-^{\rm in}$ and $\Omega_-^{\rm ext}$ are of the same size.
\begin{prop}
{
Let $S_I = 0$.  
The ODE system \eqref{q12dot:planar} admits a unique symmetric stationary solution 
$(\qstaru, \qstard)$, with 
$\qstaru = \mathcal{L} - \qstard \in \bigl(0, \tfrac{\mathcal{L}}{2}\bigr)$, 
which is the global minimum of the Lyapunov function $\mathcal{E}$ on $\mathcal{T}$.  
Moreover, every solution with initial data in $\mathcal{T}$, that is $(q_1(0), q_2(0)) \in \mathcal{T}$, converges to $(\qstaru, \qstard)$ as $t$ tends  to infinity.
}
\end{prop}

\begin{proof}
We consider a symmetric stationary solution ansatz $\qstaru = \mathcal{L} - \qstard$, which allows us to reduce to a one-dimensional problem of solving $\mathcal{H}_1^{\rm p}(\qstaru, \mathcal{L} - \qstaru) = 0$. Observe that
\[
\mathcal{H}_1^{\rm p}(0, \mathcal{L}) = -\frac{1}{2} \dplus  m_+ \lambda_+ \tanh(\lambda_+ \tfrac{\mathcal{L}}{2}) > 0,
\]
while
\[
\mathcal{H}_1^{\rm p}(\tfrac{\mathcal{L}}{2},\tfrac{\mathcal{L}}{2}) = - \frac{1}{2} \dmin  m_- \lambda_- \tanh(\lambda_- \tfrac{\mathcal{L}}{2}) < 0,
\]
and so by continuity of the function $q \mapsto \mathcal{H}_1^{\rm p}(q, \mathcal{L} - q)$, there exists a root $\qstaru \in (0,\tfrac{\mathcal{L}}{2})$ to the equation $\mathcal{H}_1^{\rm p}(\qstaru, \mathcal{L} - \qstaru) = 0$. Such a root fulfills
\begin{align}\label{planar:symm:root}
- \dplus  m_+ \lambda_+ \tanh(\lambda_+( \tfrac{\mathcal{L}}{2} - \qstaru)) = \dmin  m_- \lambda_- \tanh(\lambda_- \qstaru).
\end{align}
Setting $\qstard = \mathcal{L} - \qstaru$ in \eqref{planar:symm:root} also provides $\mathcal{H}_{2}^{\rm p}(\qstaru, \qstard) = 0$, and so $(\qstaru, \qstard)$ is an equilibrium of {\eqref{q12dot:planar}}.  As $\mathcal{E}$ is a strictly convex Lyapunov function, it has only one minimum, which we can identify as $(\qstaru, \qstard)$.  Then, invoking LaSalle's {principle} \cite[Chapter 6]{Teschl}, as $\mathcal{T}$ is a compact {set} in $\RRR^2$, we infer that any solution $(q_1(t), q_2(t))$ to {\eqref{q12dot:planar}} emanating from $(q_1(0), q_2(0)) \in \mathcal{T}$ converges to $(\qstaru, \qstard)$ as $t \to \infty$.

\end{proof}

In the simplest scenario, where we set the parameters in \eqref{planar:symm:root} such that the prefactors $\dplus  m_+ \lambda_+ = -1$ and $\dmin  m_- \lambda_- = 1$, {and} using the strict monotonicity of the $\tanh$ function, it follows that 
\[
q_1^{\star} = \frac{\mathcal{L}}{2 (1 + \frac{\lambda_-}{\lambda_+})}, \quad q_2^{\star} = \frac{\mathcal{L}(1 + 2 \frac{\lambda_-}{\lambda_+})}{2(1 + \frac{\lambda_-}{\lambda_+})}.
\]
From the above formula, we observe that in the limit $\frac{\lambda_-}{\lambda_+} \to \infty$:
\[
\lim_{\frac{\lambda_-}{\lambda_+} \to \infty} \qstaru = 0, \quad \lim_{\frac{\lambda_-}{\lambda_+} \to \infty} \qstard = \mathcal{L},
\]
so that the $\Omega_+$ phase grows to occupy the entire domain. On the other hand, in the limit $\frac{\lambda_-}{\lambda_+} \to 0$:
\[
\lim_{\frac{\lambda_-}{\lambda_+} \to 0} \qstaru = \frac{\mathcal{L}}{2}, \quad \lim_{\frac{\lambda_-}{\lambda_+} \to 0} \qstard = \frac{\mathcal{L}}{2},
\]
so that the $\Omega_-$ phases grow to occupy the entire domain.

In the case where $S_I \neq 0$, we can prescribe conditions to guarantee the existence of a symmetric stationary solution $(\qstaru,\qstard)$.  From the above arguments, we require
\begin{align*}
\mathcal{H}_1^{\rm p}(0, \mathcal{L}) & = \frac{1}{2}(- \dplus  m_+ \lambda_+ \tanh(\lambda_+ \tfrac{\mathcal{L}}{2}) - S_I )> 0,\\
\mathcal{H}_1^{\rm p}(\tfrac{\mathcal{L}}{2},\tfrac{\mathcal{L}}{2}) & = \frac{1}{2}(- \dmin  m_- \lambda_- \tanh(\lambda_- \tfrac{\mathcal{L}}{2}) - S_I) < 0.
\end{align*}
Upon rearranging, we find that a sufficient condition to obtain the existence of a symmetric stationary solution to {\eqref{q12dot:planar}} is
\begin{align*}
-\dmin  m_- \lambda_- \tanh( \lambda_- \tfrac{\mathcal{L}}{2}) < S_I < - \dplus  m_+ \lambda_+ \tanh(\lambda_+ \tfrac{\mathcal{L}}{2}).
\end{align*}

\subsubsection{Linear stability of {flat multilayered} solutions}
Next, we discuss the stability of the {flat multilayered} solutions obtained above around the stationary fronts identified by $\qstaru$ and $\qstard$, which satisfy ${\cal H}^{\rm p}_1(\qstaru, \qstard) = 0$ and ${\cal H}^{\rm p}_2(\qstaru, \qstard) = 0$. We consider perturbed interfaces of the form $w_i := q^\star_i + \badeps \rad_i$, $i=1,2$, where $0 < \badeps < 1$ and $\rad_i = \rad_i(t, \hat x)$ are suitable perturbations. We define perturbed domains as
\begin{align*}
	\Omega_{\rad_1,\badeps}^{-,{\rm in}} & := \{
	x \in \Omega\,\, : \,\,
	0 < {z}  <w_1(t, \hat x)
	\}, \\ 
	\Omega_{\rad_{1,2},\badeps}^+ & := \{
	x \in \Omega \,\, : \,\,
	 w_1(t, \hat x)<{z}  < w_2(t, \hat x)
	\},
	\\ 
	\Omega_{\rad_2,\badeps}^{-,{\rm ext}} & := \{
	x \in \Omega\,\, : \,\,
	w_2(t, \hat x) < z < \mathcal{L}
	\},
\end{align*}
with the  ansatz
\begin{align*}
		\mumin (t,x) & = \mu_-^{\star,{\rm in}}(z) + \badeps \uin (t,x),
		\\
		\mustpl (t,x) & = \mustpl(z) + \badeps \upl (t,x),
		\\
		\mumex (t,x) & = \mu_-^{\star,{\rm ext}} (z) + \badeps \uext (t,x),
\end{align*}
{where  $\mumin, \mustpl $ and $\mumex$ are the solutions to \eqref{Flatmultil} corresponding to the stationary states $\qstaru$ and $\qstard$, respectively,}
and demand that these solve the following free boundary problem on the perturbed domains:
\begin{subequations}\label{linstab:planar:membrane}
\begin{alignat}{3}
	- m_- \Delta (\mu_-^{\star,{\rm in}} + \badeps \uin) &= S_- - \rho_- (\mu_-^{\star,{\rm in}} + \badeps \uin) \quad & & \text{ in } \Omega_{\rad_1,\badeps}^{-,{\rm in}}, \\
	- m_+ \Delta (\mustpl + \badeps u_+) &= S_+ - \rho_+ (\mustpl + \badeps u_+) \quad && \text{ in } \Omega_{\rad_{1,2},\badeps}^+ , \\
- m_- \Delta (\mu_-^{\star,{\rm ext}} + \badeps \uext) &= S_- - \rho_- (\mu_-^{\star,{\rm ext}}+ \badeps \uext) \quad & & \text{ in }\Omega_{\rad_2,\badeps}^{-,{\rm ext}}, \\
	 (\mu_-^{\star,{\rm in}} + \badeps \uin)  &= \alpha \kappa_1 =(\mustpl + \badeps \upl) \quad && \text{ on } \{{z} = w_1\}, \\
	  (\mustpl + \badeps \upl) &= \alpha  \kappa_2 =(\mu_-^{\star,{\rm ext}} + \badeps \uext)  \quad && \text{ on } \{{z} = w_2\}, \\
	 - 2 {\cal V}_1 & = [m \nabla (\mu^{\star} + \badeps u)]_{-}^+ \cdot \bm{\nu}_1 + S_I \quad && \text{ on } \{{z} = w_1\}, \\
	 - 2 {\cal V}_2 &= [m \nabla (\mu^{\star} + \badeps u)]_{-}^+ \cdot \bm{\nu}_2 + S_I \quad && \text{ on } \{{z} = w_2\},
\end{alignat}
\end{subequations}
along with the Neumann boundary conditions
\[
 (\mu_-^{\star,{\rm in}} + \badeps \uin)'(t, 0)= 0, \quad 
  (\mu_-^{\star,{\rm ext}} + \badeps \uext)'(t, {\cal L})  = 0,
\]
where we recall the prime notation indicates partial derivatives with respect to $z$. In the above, $\kappa_i$ and $\bm{\nu}_i$ denote the mean curvature and unit normal of the interface {$\{z = w_i\}$}.

Linearizing the above equations around the stationary solution $(\mu_-^{\star,{\rm in}}, \mu_+^{\star}, \mu_-^{\star,{\rm ext}},  \qstaru, \qstard)$ leads to the following system for $(\uin, u_+, \uext)$:
\begin{subequations}\label{perturb:planar:memb}
\begin{alignat}{3}
\label{pertplan:memb:1}  - m_- \Delta \uin & = - \rho_- \uin \quad && \text{ in } \{0<{z} < \qstaru\}, \\
\label{pertplan:memb:2}  - m_+ \Delta u_+ & = - \rho_+  u_+ \quad && \text{ in } \{\qstaru<{z} < \qstard\}, \\
\label{pertplan:memb:3}  - m_- \Delta \uext & = - \rho_- \uext\quad && \text{ in } \{
{  \qstard<{z}<{\cal L}}\}, \\
\label{pertplan:memb:4} \alpha \Delta_{\hat x} \rad_1& =  ( \mu^{\star, {\rm in}}_-)' \vert_{{z} = \qstaru} \rad_1 + \uin  = ( \mustpl)' \vert_{{z} = \qstaru} \rad_1 + \upl \quad && \text{ on } \{ {z} = \qstaru\}, \\
\label{pertplan:memb:5} - \alpha \Delta_{\hat x} \rad_2  &= ( \mustpl)' \vert_{{z} = \qstard} \rad_2 + \upl 
= 
( \mu^{\star, {\rm ext}}_-)' \vert_{{z} = \qstard} \rad_2 + \uext 
\quad && \text{ on } \{ {z} = \qstard\}, \\
\label{pertplan:memb:6}  2 {\dot Y_1} &= - m_+ ((\mustpl)'' \vert_{{z} = \qstaru} \rad _1+ u_+') && \\
\notag & \quad + m_- (\mu_-^{\star,{\rm in}})'' \vert_{{z} = \qstaru} \rad_1 +(\uin)')  \quad && \text{ on } \{ {z} = \qstaru \}, \\
\label{pertplan:memb:7}  2 {\dot Y_2} &= - m_+ ((\mustpl)'' \vert_{{z} = \qstard} \rad _2+ u_+')
&& \\
\notag & \quad +  m_- ((\mu_-^{\star,{\rm ext}})'' \vert_{{z} =\qstard} \rad_2 +(\uext)')  \quad && \text{ on } \{ {z} = \qstard \}, \\
\label{pertplan:memb:8}  (\uin)'(t, 0) &= 0 , \quad  (\uext)'(t, {\cal L}) = 0,  && \\
\label{pertplan:memb:9} \dn Y_1 & = 0, \quad \dn Y_2 = 0 \quad && \text{ on } \pd (0, \widetilde{\cal L})^{d-1}. 
\end{alignat}
\end{subequations}
We make the following ansatz
\begin{align*}
Y_1(t,\widehat{x}) = y_1 \delta(t) W(\widehat{x}), \quad Y_2(t,\widehat{x}) =y_2  \delta(t) W(\widehat{x}), 
\end{align*}
for some $y_1, y_2 \in \RRR$ to be determined, and set
\begin{align*}
\uin(t,x) = \delta(t)\vin(z) W(\widehat{x}), \quad u_+(t,x) = \delta(t)v_+(z)W(\widehat{x}) , \quad \uext(t,x)  = \delta(t) \vext(z) W(\widehat{x}), 
\end{align*}
where we fix $W$ as a particular eigenfunction of the $\Delta_{\widehat{x}}$-operator with Neumann boundary conditions such that for $\hat{\bm {\ell}} = (\ell_2, \dots, \ell_d) \in (\enne \cup \{ 0 \})^{d-1}$,
\begin{equation*}
\begin{cases}
\displaystyle \Delta_{\hat x} W  =  \frac{\zeta_{\hat{\bm\ell}, d}}{(\widetilde{\cal L})^2} W & \text{ in } (0, \widetilde{\cal L})^{d-1},\\
\dn W  = 0 & \text{ on } \partial (0, \widetilde{\cal L})^{d-1},
\end{cases} 
\end{equation*}
where $\frac{\zeta_{\hat{\bm\ell}, d}}{(\widetilde{\cal L})^2}$ serves as the corresponding eigenvalue. A possible choice is
\[
W(x_2, \dots, x_d) = \cos \Big ( \frac{\pi}{\widetilde{\cal L}} \ell_2 x_2 \Big ) \times \cdots \times \cos \Big ( \frac{\pi}{\widetilde{\cal L}} \ell_d x_d \Big ),
\]
so that $\zeta_{\hat{\bm \ell}, d} = - \pi^2 \big((\ell_2)^2 + \cdots + (\ell_d)^2\big) = - \pi^2 |\bm\ell|^2$. 

With these {ansatzes}, and setting 
\[
\Gamma_{\pm}^{\widehat{\bm{\ell}}} := \sqrt{\lambda_{\pm}^2 - \frac{\zeta_{\widehat{\bm{\ell}},d}}{(\widetilde{\cal L})^2}} = \sqrt{\frac{\rho_{\pm}}{m_{\pm}}- \frac{\zeta_{\widehat{\bm{\ell}},d}}{(\widetilde{\cal L})^2}},
\]
we then solve the following system of equations
\begin{subequations}
\begin{alignat}{3}
\label{planar:shell:perturbed:1} (\vin)'' & = (\Gamma_{-}^{\widehat{\bm{\ell}}})^2 \vin \quad && \text{ in } (0, \qstaru), \\
\label{planar:shell:perturbed:2} (v_+)'' & = (\Gamma_{+}^{\widehat{\bm{\ell}}})^2 v_+, \quad && \text{ in } (\qstaru,\qstard), \\ 
\label{planar:shell:perturbed:3} (\vext)'' & = (\Gamma_{-}^{\widehat{\bm{\ell}}})^2 \vext \quad && \text{ in } (\qstard, \mathcal{L}), \\
\label{planar:shell:perturbed:4}  y_1 \alpha \frac{\zeta_{\hat{\bm\ell}, d}}{(\widetilde{\cal L})^2} & =  y_1 ( \mu^{\star, {\rm in}}_-)' + \vin  = y_1 ( \mustpl)' + v_+, \quad && \text{ on } \{ {z} = \qstaru\}, \\
\label{planar:shell:perturbed:5}  -y_2 \alpha \frac{\zeta_{\hat{\bm\ell}, d}}{(\widetilde{\cal L})^2}   &= y_2 ( \mustpl)'   + v_+= y_2 ( \mu^{\star, {\rm ext}}_-)'+ v_-^{\rm ext}, \quad && \text{ on } \{ {z} = \qstard\}, \\
\label{planar:shell:perturbed:6} 2 {\dot Y_1} & = - m_+ ((\mustpl)''  \rad _1+ u_+')  + m_- ((\mu_-^{\star,{\rm in}})''  \rad_1 +(\uin)')  \quad && \text{ on } \{ {z} = \qstaru \}, \\
\label{planar:shell:perturbed:7} 2 {\dot Y_2} & = - m_+ ((\mustpl)'' \rad _2+ u_+') + m_- ((\mu_-^{\star,{\rm ext}})'' \rad_2 +(\uext)')   \quad && \text{ on } \{ {z} = \qstard \}, \\
\label{planar:shell:perturbed:8}  (\vin)'(t, 0) & = 0, \quad (\vext)'(t, {\cal L}) = 0. && 
\end{alignat}
\end{subequations}
From \eqref{planar:shell:perturbed:1}, \eqref{planar:shell:perturbed:2}, \eqref{planar:shell:perturbed:3} and \eqref{planar:shell:perturbed:8}, we infer that
\begin{subequations}
\begin{alignat}{3}
\label{planar:vin} \vin(z) & = a_-^{\rm in} \cosh \big ( \Gamma_-^{\widehat{\bm{\ell}}} z \big ) \quad && z \in (0, \qstaru), \\
\label{planar:vp} v_+(z) & = a_+ \sinh \big ( \Gamma_-^{\widehat{\bm{\ell}}} (z - \qstaru) \big ) + b_+ \sinh \big ( \Gamma_-^{\widehat{\bm{\ell}}} (z - \qstard) \big )\quad &&  z \in (\qstaru,\qstard), \\
\label{planar:vext} \vext(z) & = a_-^{\rm ext} \cosh \big ( \Gamma_-^{\widehat{\bm{\ell}}} (\mathcal{L} -  z) \big ) \quad && z \in (\qstard, \mathcal{L}), 
\end{alignat} 
\end{subequations}
for unknown functions $a_-^{\rm in}$, $a_+$, $b_+$ and $a_-^{\rm ext}$ to be determined. Recalling {\eqref{mum:thin:lay}} we see that  
\begin{equation*}
\begin{alignedat}{3}
(\mustin)'(\qstaru) &= -\dmin  \Lm \tanh(\Lm \qstaru), 
\quad
&&(\mustin)''(\qstaru) = -\dmin  \Lm^2,
\\
(\mustpl)'(\qstaru) &= \dplus  \Lp \tanh\left( \Lp  \tfrac{\qstard - \qstaru}{2}  \right), 
\quad 
&&(\mustpl)''(\qstaru) = -\dplus  \Lp^2, 
\\
(\mustpl)'(\qstard) &= -\dplus  \Lp \tanh\left( \Lp  \tfrac{\qstard - \qstaru}{2} \right), 
\quad 
&&(\mustpl)''(\qstard) = -\dplus  \Lp^2 ,
\\
(\mustext)'(\qstard) &= \dmin  \Lm \tanh\left( \Lm (\Lpanarh - \qstard) \right), 
\quad
&&(\mustext)''(\qstard) = -\dmin  \Lm^2,
\end{alignedat}
\end{equation*}
and so from \eqref{planar:shell:perturbed:4} and \eqref{planar:shell:perturbed:5} we have the conditions
\begin{align*}
\vin(\qstaru) & = \Big ( \dmin  \Lm \tanh(\Lm \qstaru) + \alpha \frac{\zeta_{\widehat{\bm{\ell}},d}}{(\widetilde{\cal L})^2} \Big ) y_1, \\
v_+(\qstaru) & = \Big (- \dplus  \Lp \tanh \Big ( \Lp \tfrac{\qstard-\qstaru}{2} \Big ) + \alpha \frac{\zeta_{\widehat{\bm{\ell}},d}}{(\widetilde{\cal L})^2} \Big ) y_1, \\
\vext(\qstard) & = \Big (- \dmin  \Lm \tanh (\Lm (\mathcal{L} - \qstard)) - \alpha \frac{\zeta_{\widehat{\bm{\ell}},d}}{(\widetilde{\cal L})^2} \Big ) y_2, \\
v_+(\qstard) & = \Big (\dplus  \Lp \tanh \Big ( \Lp \tfrac{\qstard-\qstaru}{2} \Big )  - \alpha \frac{\zeta_{\widehat{\bm{\ell}},d}}{(\widetilde{\cal L})^2} \Big ) y_2,
\end{align*}
which also allows us to identify
\begin{align*}
a_-^{\rm in} &= \frac{ \Big ( \dmin  \Lm \tanh(\Lm \qstaru) + \alpha \frac{\zeta_{\widehat{\bm{\ell}},d}}{(\widetilde{\cal L})^2}  \Big ) y_1}{\cosh (\Gamma^{\widehat{\bm{\ell}}}_- \qstaru)}, \quad a_-^{\rm ext} = \frac{ \Big (- \dmin  \Lm \tanh (\Lm (\mathcal{L} - \qstard)) - \alpha \frac{\zeta_{\widehat{\bm{\ell}},d}}{(\widetilde{\cal L})^2} \Big ) y_2}{\cosh(\Gamma^{\widehat{\bm{\ell}}}_-(\mathcal{L} - \qstard))}, \\
a_+ & = \frac{\Big (\dplus  \Lp \tanh \Big ( \Lp \tfrac{\qstard-\qstaru}{2} \Big )  - \alpha \frac{\zeta_{\widehat{\bm{\ell}},d}}{(\widetilde{\cal L})^2} \Big ) y_2}{\sinh(\Gamma^{\widehat{\bm{\ell}}}_+(\qstard-\qstaru))}, \quad b_+ =  \frac{ \Big ( \dplus  \Lp \tanh \Big ( \Lp \tfrac{\qstard-\qstaru}{2} \Big ) - \alpha \frac{\zeta_{\widehat{\bm{\ell}},d}}{(\widetilde{\cal L})^2} \Big ) y_1}{\sinh(\Gamma^{\widehat{\bm{\ell}}}_+(\qstard-\qstaru))},
\end{align*}
where we have also used that $\sinh$ is an odd function. Then, for the evolution equations \eqref{planar:shell:perturbed:6}--\eqref{planar:shell:perturbed:7}, we infer that
\begin{align*}
2 \dot{Y}_1  = 2 y_1 \dot{\delta} W & =- m_+ \Big [ -\dplus  \lambda_+^2 y_1 + \Big ( \Gamma^{\widehat{\bm{\ell}}}_+ \Big ( a_+ + b_+ \cosh (\Gamma^{\widehat{\bm{\ell}}}_+(\qstard-\qstaru) \Big ) \Big ) \Big ] \delta W \\
& \quad + m_- \Big [ -\dmin  \lambda_-^2 y_1 + a_-^{\rm in} \Gamma^{\widehat{\bm{\ell}}}_- \sinh(\Gamma^{\widehat{\bm{\ell}}}_- \qstaru) \Big ] \delta W, \\
2 \dot{Y}_2  = 2 y_2 \dot{\delta} W & =- m_+ \Big [ -\dplus  \lambda_+^2 y_2 + \Big ( \Gamma^{\widehat{\bm{\ell}}}_+ \Big ( a_+ \cosh \big (\Gamma^{\widehat{\bm{\ell}}}_+ (\qstard-\qstaru) \big)+ b_+ \Big ) \Big ) \Big ] \delta W \\
& \quad + m_- \Big [ -\dmin  \lambda_-^2 y_2 - a_{-}^{\rm ext} \Gamma^{\widehat{\bm{\ell}}}_- \sinh \Big ( \Gamma^{\widehat{\bm{\ell}}}_- (\mathcal{L} - \qstard) \Big ) \Big ] \delta W,
\end{align*}
which can be expressed with the help of a matrix structure as
\begin{align}\label{perturbation_matrix:multiplanar}
2 \dot{\delta} W \begin{pmatrix} y_1 \\ y_2 \end{pmatrix} = \delta W \begin{pmatrix} A_{11} & A_{12} \\ A_{21} & A_{22} \end{pmatrix} \begin{pmatrix} y_1 \\ y_2 \end{pmatrix} = \delta W \mathbf{A} \begin{pmatrix} y_1 \\ y_2 \end{pmatrix},
\end{align}
where the entries of the symmetric matrix $\mathbf{A}$ are
\begin{align*}
A_{11} & = - m_+ \left [ - \dplus  \lambda_+^2 + \frac{\Big ( \dplus  \lambda_+ \tanh(\lambda_+ \frac{\qstard-\qstaru}{2}  ) - \alpha \frac{\zeta_{\widehat{\bm{\ell}},d}}{(\widetilde{\mathcal{L}})^2 }\Big ) \Gamma^{\widehat{\bm{\ell}}}_+ \cosh  ( \Gamma^{\widehat{\bm{\ell}}}_+ (\qstard-\qstaru)  )}{\sinh ( \Gamma^{\widehat{\bm{\ell}}}_+ (\qstard-\qstaru)  )} \right ] \\
& \quad + m_- \left [ -\dmin  \lambda_-^2 + \frac{\Big ( \dmin  \lambda_- \tanh(\lambda_- \qstaru) + \alpha \frac{\zeta_{\widehat{\bm{\ell}},d}}{(\widetilde{\mathcal{L}})^2} \Big ) \Gamma^{\widehat{\bm{\ell}}}_- \sinh  ( \Gamma^{\widehat{\bm{\ell}}}_- \qstaru  )}{\cosh  ( \Gamma^{\widehat{\bm{\ell}}}_- \qstaru  )} \right ] , \\
A_{12} & = A_{21} = - m_+ \Gamma^{\widehat{\bm{\ell}}}_+ \frac{\Big ( \dplus  \lambda_+ \tanh ( \lambda_+ \frac{\qstard-\qstaru}{2}) - \alpha \frac{\zeta_{\widehat{\bm{\ell}},d}}{(\widetilde{\mathcal{L}})^2} \Big )}{\sinh  ( \Gamma^{\widehat{\bm{\ell}}}_+ (\qstard - \qstaru)  )}, \\
A_{22} & = -m_+ \left [ - \dplus  \lambda_+^2 +  \frac{\Big ( \dplus  \lambda_+ \tanh(\lambda_+ \frac{\qstard-\qstaru}{2}  ) - \alpha \frac{\zeta_{\widehat{\bm{\ell}},d}}{(\widetilde{\mathcal{L}})^2 }\Big ) \Gamma^{\widehat{\bm{\ell}}}_+ \cosh  ( \Gamma^{\widehat{\bm{\ell}}}_+ (\qstard-\qstaru)  )}{\sinh ( \Gamma^{\widehat{\bm{\ell}}}_+ (\qstard-\qstaru)  )} \right ] \\
& \quad + m_- \left [ - \dmin \lambda_-^2 +  \frac{\Big (\dmin  \lambda_- \tanh(\lambda_-(\mathcal{L}- \qstard)) + \alpha \frac{\zeta_{\widehat{\bm{\ell}},d}}{(\widetilde{\mathcal{L}})^2} \Big ) \Gamma^{\widehat{\bm{\ell}}}_- \sinh  ( \Gamma^{\widehat{\bm{\ell}}}_- (\mathcal{L} - \qstard)  )}{\cosh  ( \Gamma^{\widehat{\bm{\ell}}}_- (\mathcal{L} - \qstard)  )}  \right ].
\end{align*}
We note that for a symmetric stationary solution, i.e., $\qstaru = \mathcal{L} - \qstard$, it even holds that $A_{11} = A_{22}$. Consequently, the matrix $\mathbf{A}$ has eigenvalues $A_{11} + A_{12}$ and $A_{11} - A_{12}$. Hence, instability can be observed if either one of the eigenvalues is positive.

\subsection{Radial multilayered solutions}
\label{SEC:RAD:MULTI}

We now proceed with the analysis of radial multilayered solutions for {\eqref{SharpI}}. In this direction, we consider again the setting $\Omega = B_R(0)$ and seek for radially symmetric solutions under the geometry
\begin{align*}
\Omega^{\rm int}_-(t) = 
B_{q_1(t)}(0),
\quad 
\Omega_+(t) = B_{q_2(t)}(0) \setminus \ov{B_{q_1(t)}(0)},
\quad \Omega^{\rm ext}_-(t) = B_R(0)\setminus \ov{B_{q_2(t)}(0)},
\end{align*}
where $q_1=q_1(t)$ and $q_2=q_2(t)$ encode the location of the moving interfaces $\Sigma^1 := \partial B_{q_1(t)}(0)$ and $\Sigma^2 := \partial B_{q_2(t)}(0)$ for which we assume the condition $q_1 <q_2$ holds for all times. For $p_i \in \Sigma^i$, with $i=1,2$, the unit normal, normal velocity and mean curvature are given as
\begin{subequations} \label{q123:prop:shell}
\begin{alignat}{3}
	\label{q1:prop:shell}
	\bnu_1 (p_1) & =\frac {p_1}{|p_1|} = \frac {p_1}{q_1(t)},
	\qquad &
	{\cal V}_1 &= \dot{q_1}(t),
	\qquad &
	\kappa_1 &= -\frac {d-1}{q_1(t)},
	\\
	\label{q2:prop:shell}
	\bnu_2 (p_2) & =- \frac {p_2}{|p_2|} =- \frac {p_2}{q_2(t)},
	\qquad &
	{\cal V}_2 &= -\dot{q_2}(t),
	\qquad &
	\kappa_2 & = \frac {d-1}{q_2(t)},
\end{alignat}
\end{subequations}
respectively{,} see Figure~\ref{fig:radialshell} for a schematic representation.

\begin{figure}[h]
\centering
\begin{tikzpicture}[scale=1.1]
  \draw[thick] (0,0) circle (1.5);
  \draw[thick] (0,0) circle (3);
  \draw[thick] (0,0) circle (4.5);
  
    \filldraw[fill=black!25, thick, line join=bevel]
 	(0,0) circle (3);
 	    \filldraw[fill=black!0, thick, line join=bevel]
 	(0,0) circle (1.5);

  \node at (0,0) {$\Omega^{\textrm{in}}_-(t)$};
  \node at (-2.3,0) {$\Omega_+(t)$};
  \node at (-3.7,0) {$\Omega^{\textrm{ext}}_-(t)$};

  \node at (0,-1.7) {$q_1(t)$};
  \node at (0,-3.2) {$q_2(t)$};

  \draw[->] (0,1.5) -- (0,2.2) node[midway, right] {$\boldsymbol{\nu}_1(t)$};
  \draw[->] (3,0) -- (2,0) node[midway, below] {$\boldsymbol{\nu}_2(t)$};
  \draw[->] (4.5,0) -- (5.2,0) node[midway, below] {$\boldsymbol{n}$};

  \node at (4.2,3.5) {$\partial B_{R}(0)$};
\end{tikzpicture}
\caption{The geometric setting for the radial multilayered solutions.}
\label{fig:radialshell}
\end{figure}
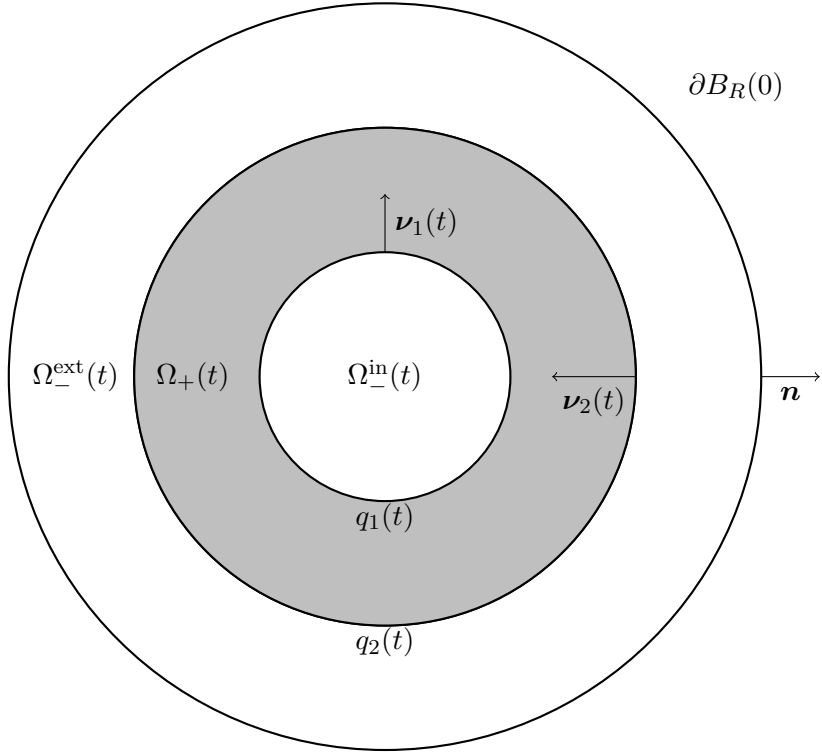
In this radially symmetric setting, we derive from {\eqref{SharpI}}\ the system
\begin{subequations}\label{rad:shell}
\begin{alignat}{2}
	\label{rad:shell:1}
	 m_-(\mumin)'' +   m_-\frac {d-1}r (\mumin)' & =
	-{S_-} +  \rho_- \mumin
	 \quad &&\text{in $\{{0<r<q_1(t)}\}$,}
	\\
	\label{rad:shell:2}
	 m_+\mu''_+ +   m_+\frac {d-1}r \mu_+' & =
	- {S_+}  +  \rho_+ \mu_+
	 \quad &&\text{in $\{q_1(t)<r<q_2(t)\}$,}
	\\
		\label{rad:shell:3}
	  m_-(\mumex)''_- +   m_-\frac {d-1}r (\mumex)' & =
	- {S_-}  +  \rho_- \mumex
	 \quad &&\text{in $\{{q_2(t)<r<R}\}$,}
	\\
	\label{rad:shell:4}
	 \mumin & = \mu_+ = -\alpha\frac{d-1}{q_1(t)}
	\quad && \text{on $\{r=q_1(t)\}$,}
	\\
	\label{rad:shell:5}
	 \mu_+ & = \mumex = \alpha\frac{d-1}{q_2(t)}
	\quad && \text{on $\{r=q_2(t)\}$,}
	\\
	\label{rad:shell:8}
	 2\dot{q_1} & = -\jump{m\mu'} - S_I
	\quad && \text{on $\{r=q_1(t)\}$,}
	\\
	\label{rad:shell:9}
	 2\dot{q_2} & = -\jump{m\mu'} + S_I
	\quad && \text{on $\{r=q_2(t)\}$,}
	\\
	\label{rad:shell:10}
	(\mumex)'(t, r) & =0
	\quad && \text{on $\{r = R\}$,}
\end{alignat}
\end{subequations}
where we recall {that} the prime notation indicates derivative with respect to $r$.
We further complement the above system with the {condition} at the origin
\begin{align}
	\label{rad:shell:11}
	\mumin(t,0)<\infty.
\end{align}

\subsubsection{Analytical {formulas} for {the} solutions}
Recalling the constants $\lambda_{\pm} = \sqrt{\frac{\rho_{\pm}}{m_{\pm}}}$, we introduce the coefficient functions
\begin{align*}
	\bmin (q_1) & = \Big(-\alpha\frac{d-1}{q_1} - \frac{S_-}{\rho_-} \Big) \frac 1{\Iz(\Lm  q_1)}, \\
		b_{d,+}\qud & 
	 = \frac{  \frac 1{\Kz(\Lp q_1)} \Big(-\alpha\frac{d-1}{q_1} - \frac{S_+}{\rho_+} \Big) - \frac1{\Kz(\Lp q_2)}  \Big(\alpha\frac{d-1}{q_2} - \frac{S_+}{\rho_+} \Big) }{\frac{\Iz (\Lp q_1)}{ \Kz (\Lp q_1)} 
- \frac {\Iz (\Lp q_2)}{\Kz (\Lp q_2)}},
	\\
	c_{d,+}\qud & =
	\frac{  \frac 1{\Iz(\Lp q_2)} \Big(\alpha\frac{d-1}{q_2} - \frac{S_+}{\rho_+} \Big) - \frac1{\Iz(\Lp q_1)}  \Big(-\alpha\frac{d-1}{q_1} - \frac{S_+}{\rho_+} \Big) }{\frac{\Kz (\Lp q_2)}{ \Iz (\Lp q_2)}
- \frac {\Kz (\Lp q_1)}{\Iz (\Lp q_1)}}, \\
	\bmex (q_2) & = \Big( \frac{\Iz (\Lm q_2)}{ \Kz (\Lm q_2)}
+ \frac {\III_1 (\Lm R)}{\KKK_1 (\Lm R)}\Big)^{-1} \Big( \alpha\frac{d-1}{q_2} 
	- \frac{S_-}{\rho_-} \Big) \frac 1 {\Kz (\Lm q_2)} ,
	\\
	\cmex (q_2) & =  \Big( \frac{\KKK_1 (\Lm R)}{ \III_1 (\Lm R)}
+ \frac {\Kz (\Lm q_2)}{\Iz (\Lm q_2)}\Big)^{-1}
	\Big( \alpha\frac{d-1}{q_2} 
	- \frac{S_-}{\rho_-} \Big)\frac 1{\Iz (\Lm q_2)},
\end{align*}
and in Appendix \ref{app:radialshell} we detail the derivation of the following analytical {formulas} for the solution $(\mumin, \mu_+, \mumex)$ to \eqref{rad:shell}
\begin{subequations} \label{def:shell:mu}
\begin{alignat}{3}
	\label{def:shell:muin}
	\mumin (t,r) & =
	\bmin(q_1(t)) \Iz(\Lm r)+ \frac{S_-}{\rho_-}, && \quad \text{ for } r \in (0, q_1(t)),
	\\ \non
	\mu_+ (t,r) & = b_{d,+}(q_1(t),q_2(t)) \Iz(\Lp r) &&
	\\
	\label{def:shell:mupl} & \quad +c_{d,+}(q_1(t),q_2(t)) \Kz(\Lp r)  	
 + \frac{S_+}{\rho_+}, && \quad \text{ for } r \in (q_1(t), q_2(t)),
	\\ \label{def:shell:muext}
	\mumex (t,r) & = \bmex (q_2(t)) \Iz(\Lm r)  
	+	\cmex (q_2(t)) \Kz(\Lm r) + \frac{S_-}{\rho_-}, && \quad \text{ for } r \in (q_2(t),R),
\end{alignat}
\end{subequations}
where we note the resemblance of \eqref{def:shell:mupl} to (C4) of \cite{bauermann2023formation}.  Together with {the} expressions
\begin{align*}
(\mumin)'(r) & =  \Lm \bmin(q_1) \III_1 (\Lm r), \\
\mu_+'(r) & =\Lp b_{d,+}\qud \III_1 (\Lp r) 
	-\Lp c_{d,+} \qud \KKK_1(\Lp r) , \\
(\mumex)'(r) & =\Lm  \bmex (q_2) \III_1 (\Lm r)  
	- \Lm \cmex (q_2) \KKK_1(\Lm r),
\end{align*}
and upon substituting into \eqref{rad:shell:8} and \eqref{rad:shell:9}, {this} yields the ordinary differential equations for the evolving inner radius $q_1(t)$ and outer radius $q_2(t)$:
\begin{subequations}\label{radial:shell:ode}
\begin{alignat}{2}
\notag	2\dot{q_1}  
\notag	 & = - m_+\mu_+'(q_1)
	+m_-(\mumin)'(q_1)
	- S_I\\
\notag	& =  - m_+ \Lp (b_{d,+}\qud \III_1(\Lp q_1) - c_{d,+} \qud \KKK_1 (\Lp q_1)) \\
\notag	& \qquad + m_- \Lm \bmin(q_1) \III_1 (\Lm q_1) - S_1, \\
&=: {\cal H}_{\mathrm{in}}\qud, \label{radial:shell:ode1}
	\\
\notag	  2\dot{q_2} 
\notag	  & =	 - m_+\mu_+'(q_2)
	+m_-(\mumex)'(q_2)
	+ S_I \\
\notag 	& =  - m_+ \Lp ( b_{d,+}\qud \III_1(\Lp q_2) -  c_{d,+} \qud \KKK_1 (\Lp q_2)) \\
\notag	&\quad + m_-\Lm ( \bmex (q_2) \III_1 (\Lm q_2)  
	-  \cmex (q_2) \KKK_1(\Lm q_2))  + S_I \\
&  
	=: {\cal H}_{\mathrm{out}}\qud. \label{radial:shell:ode2}
\end{alignat}
\end{subequations}
Numerically we can identify certain parameter regimes where ${\cal H}_{\mathrm{in}} \qud$ and ${\cal H}_{\mathrm{out}} \qud$ admit roots $(q_1^\star, q_2^\star)$ compatible with the multilayered structure, i.e., $q_1^\star < q_2^\star$.

\subsubsection{Linear stability of radial multilayered solutions}
Let us now examine {perturbations} of a radial multilayered solution around the stationary radii $\qstaru$ and $\qstard$ that are roots of ${\cal H}_{\mathrm{in}}$ and ${\cal H}_{\mathrm{out}}$. We proceed as before and consider perturbed radii of the form 
\[
w_i = q^\star_i + \badeps \rad_i, \quad i=1,2, \quad \badeps \in (0,1),
\]
with a similar structure for $\rad_i$ as for $\rad$ in Section \ref{SEC:RAD:STAB}.
Using again $X(\th,\phi)$ to indicate the standard  parametrization of the $d$ dimensional unit sphere using
the polar and azimuthal angles, we define the perturbed domains by
\begin{align*}
	\Omega_{\rad_1,\badeps}^{-,{\rm in}} & := \{
	x \in \Omega \,\, : \,\,  x = |x| X(\th,\phi), \,
	0<x  < \qstaru + \badeps \rad_1(t,\th,\phi)
	\},
	\\
	\Omega_{\rad_{1},\rad_{2},\badeps}^+ & := \{
	x \in \Omega\,\, : \,\,  x = |x| X(\th,\phi), \,
	\qstaru + \badeps \rad_1(t,\th,\phi)<x  < \qstard + \badeps \rad_2(t,\th,\phi)
	\},
	\\
	\Omega_{\rad_2,\badeps}^{-,{\rm ext}} & := \{
	x \in \Omega \,\, : \,\,  x = |x| X(\th,\phi), \,
	 \qstard + \badeps \rad_2(t,\th,\phi)<x<R
	\}.
\end{align*}
Denoting by $(\mustin, \mustpl, \mustext)$ the solution to the free boundary problem corresponding to the stationary radii $\qudst$, we then express the solutions as 
\begin{align*}
\mu_-^{\rm in}(r,t,\th,\phi) & = \mustin(t,r) + \badeps u_{-}^{\rm in}(r,t,\th,\phi), \\
\mu_+(r,t,\th,\phi) & = \mustpl(t,r) + \badeps u_{+}(r,t,\th,\phi), \\
\mu_-^{\rm ext}(r,t,\th,\phi) & = \mu_-^{\star, {\rm ext}}(t,r) + \badeps u_{-}^{\rm ext}(r,t,\th,\phi),
\end{align*}
and demand that these solve the free boundary problem {\eqref{SharpI}} on the perturbed domains:
\begin{subequations}
\begin{alignat}{2}
		\label{lin:shell:1:lin}
		- m_- \Delta (\mustin + \badeps u_{-}^{{\rm in}}) & = S_- - \rho_- (\mustin+ \badeps u_{-}^{\rm in}) \qquad &&\text{in $\Omega_{\rad_1,\badeps}^{-,{\rm in}}$,}
	\\	\label{lin:shell:2:lin}
		- m_+ \Delta (\mustpl+ \badeps u_{+}) &= S_+ - \rho_+ (\mustpl+ \badeps u_{+}) \qquad &&\text{in $\Omega_{\rad_1,\rad_2,\badeps}^+$,}
	\\
		\label{lin:shell:3:lin}
		- m_- \Delta (\mustext+ \badeps u_{-}^{{\rm ext}}) &= S_- - \rho_- (\mu_-^\star+ \badeps u_{-}^{{\rm ext}}) \qquad &&\text{in $\Omega_{\rad,\badeps}^{-,{\rm ext}}$,}
	\\
	\label{lin:shell:4:lin}
	 \mustin + \badeps u_-^{\rm in} & =  \alpha \kappa_1= \mustpl + \badeps u_+ 
	\qquad && \text{on $\{r=w_1\}$,}
	\\
		\label{lin:shell:5:lin}
	 \mustpl + \badeps u_+ &= \alpha \kappa_2= \mustext + \badeps u_-^{\rm ext} 
	\qquad && \text{on $\{r=w_2\}$,}
	\\
	\label{lin:shell:6:lin}
	 - 2 {\cal V}_1 &= \jump{m \nabla (\mu + \badeps u)} \cdot {\bnu}_1
	+ S_I
	\qquad && \text{on $\{r=w_1\}$,}
	\\
	\label{lin:shell:7:lin}
	 - 2 {\cal V}_2 & = \jump{m \nabla (\mu + \badeps u)} \cdot {\bnu}_2
	+ S_I
	\qquad && \text{on $\{r=w_2\}$,}
	\\
	\label{lin:shell:8:lin}
	(\mustext + \badeps u_{-}^{{\rm ext}})'(R) & =0
	\qquad && \text{on $\{r=R\}${.}}
\end{alignat}
\end{subequations}
Linearizing the above equations about the stationary interfaces $\{r = q^\star_i\}$, $i=1,2$, while using the expansions
\begin{align*}
	\mustin(w_1) &   =
	\mustin(\qstaru)
	+ (\mustin)'|_{r=\qstaru} (w_1-\qstaru) + \text{ h.o.t.} \\
	&=
	\mustin(\qstaru)
	+ \badeps (\mustin)'|_{r=\qstaru} Y_1 +\text{ h.o.t.},
	\\
	\mu_+^{\star}(w_1) &   =
	\mu_+^{\star}(\qstaru)
	+ (\mustpl)'|_{r=\qstaru} (w_1-\qstaru) + \text{ h.o.t.}\\
	&=
	\mu_+^{\star}(\qstaru)
	+ \badeps (\mu^{\star}_+)'|_{r=\qstaru} Y_1 +\text{ h.o.t.},
	\\
	\mu_+^{\star}(w_2) &   =
	\mu_+^{\star}(\qstard)
	+ (\mustpl)'|_{r=\qstard} (w_2-\qstard) + \text{ h.o.t.}\\
	&=
	\mu_+^{\star}(\qstard)
	+ \badeps (\mu^{\star}_+)'|_{r=\qstard} Y_2 +\text{ h.o.t.},
	\\
		\mu_-^{\star, {\rm ext}}(w_2) &   =
	\mu_-^{\star, {\rm ext}}(\qstard)
	+ (\mu^{\star, {\rm ext}}_-)'|_{r=\qstard} (w_2-\qstard) + \text{ h.o.t.}\\
	&=
	\mu_-^{\star, {\rm ext}}(\qstard)
	+ \badeps (\mu^{\star, {\rm ext}}_-)'|_{r=\qstard} Y_2 +\text{ h.o.t.},
\end{align*}
we obtain the following system for $\rad_1$, $\rad_2$, $u^{\rm in}_-$, $u_{+}$ and $u^{\rm ext}_-$:
\begin{equation*}
\begin{alignedat}{2}
	- m_- \Delta  \uin &=  - \rho_- \uin \qquad &&\text{in $(0, \qstaru)$,}
	\\
	- m_+ \Delta  \upl &=  - \rho_+ \upl \qquad &&\text{in $(\qstaru,\qstard)$,}
	\\
	- m_- \Delta \uext  &= - \rho_- \uext  \qquad &&\text{in $(\qstard, R)$,}
	\\
 	 (\mustpl)' Y_1  + \upl  &= (\mustin)' Y_1  + \uin = {\alpha} 	\frac {d-1}{(\qstaru)^2} \big( \rad_1 + \frac 1{d-1} \Delta_{{\cal S}^{d-1}} \rad_1 \big) \quad && \text{on $\{r = \qstaru\}$,} 
	\\
 	 (\mustpl)' Y_2  + \upl
	 &= (\mustext)' Y_2  + \uext =  - {\alpha} 	\frac {d-1}{(\qstard)^2} \big( \rad_2 + \frac 1{d-1} \Delta_{{\cal S}^{d-1}} \rad_2 \big)\quad && \text{on $\{r = \qstard\}$,} \\
	  2 \dot{Y_1}  &= -
	m_+(\mustpl)'' \rad_1
	+m_-(\mustin)'' \rad_1
	- m_+ (\upl)' + m_- (\uin)'
	\qquad && \text{on $\{r=\qstaru\}$,}
	\\
	 2 \dot{Y_2}  &=
	-m_+(\mustpl)'' \rad_2
	+m_-(\mustext)'' \rad_2
	- m_+ (\upl)'
	+ m_- (\uext)'
	\qquad && \text{on $\{r=\qstard\}$,}
	\\
	 (\uext)'  & =0
	\qquad &&
	\text{on $\{r=R\}$.}
\end{alignedat}
\end{equation*}
Here, $\Delta_{{\cal S}^{d-1}} $ denotes the Laplace--Beltrami operator on the $(d-1)$-dimensional unit sphere ${\cal S}^{d-1}$ and  we employed the same linearization as before for the mean curvature operator around a sphere. Moreover, we impose the additional {condition} at the origin
\begin{align}
	\label{pertrad:shell:initial}
	\uin (r=0)<\infty.
\end{align}
We now proceed with the ansatz
\begin{align}
	\label{eq:y12}
	\rad_i (t,\th,\phi)  =   y_{i}  \delta(t)Z(\th,\phi),\quad  i=1,2,
\end{align}
for some $y_1, y_2 \in \RRR$ and $Z(\th,\phi)$ is chosen as in \eqref{perturb2d} or \eqref{perturb3d} for $\ell \in \enne \cup \{ 0 \}$ and $ k \in \{-\ell , ..., \ell \}$. With the above ansatz we require $\uin, \upl,$ and $\uext$  to assume the form
\begin{align*}
	  \uin (r,t, \th,\phi) & =   \Uin (r)  \delta(t)Z(\th,\phi),\\
	 \upl (r,t, \th,\phi)   & =     \Upl(r) \delta(t)Z(\th,\phi),
	 \\
	 \uext (r,t, \th,\phi)& =     \Uext(r) \delta(t)Z(\th,\phi).
\end{align*}
Upon recalling \eqref{Delta:rad} and \eqref{lap}, we focus on solving the following system
\begin{subequations}\label{Linstabshell}
\begin{alignat}{2}
	\label{linstab:shell:1}
	&
	- m_- \Big( (\Uin)'' + \frac {d-1}r (\Uin)' + \frac {\zeta_{\ell,d}}{r^2}  \Uin \Big)=  - \rho_-  \Uin  &&\text{ in $(0, \qstaru)$,}
	\\
	\label{linstab:shell:2}
	&
	- m_+ \Big( (\Upl)''  + \frac {d-1}r (\Upl)'+ \frac {\zeta_{\ell,d}}{r^2}  \Upl \Big)=  - \rho_+  \Upl  &&\text{ in $(\qstaru, \qstard)$,}
	\\
	\label{linstab:shell:3}
	&
	- m_- \Big( (\Uext)''  + \frac {d-1}r (\Uext)' + \frac {\zeta_{\ell,d}}{r^2} \ \Uext \Big)=  - \rho_-  \Uext  &&\text{ in $(\qstard, R)$,}
	\\
	\label{linstab:shell:4}
	&   (\mustpl)' y_1
	+ \Upl
 = \alpha \frac {d-1}{r^2} \Big( 1 + \frac {\zeta_{\ell,d}}{d-1}\Big) y_1
	=  (\mustin)' y_1
	+ \Uin   && \text{ on $\{r=\qstaru\}$,}
	\\
	\label{linstab:shell:5}
	&  (\mustpl)' y_2
	+ \Upl   = -\alpha \frac {d-1}{r^2} \Big( 1 + \frac {\zeta_{\ell,d}}{d-1}\Big) y_2
	=  (\mustext)' y_2
	+ \Uext 
	 && \text{ on $\{r=\qstard\}$,}
	\\
	\label{linstab:shell:6}
	& \frac{ 2\dot{\delta}}{\delta} y_1
	=
	-\big(m_+(\mustpl)''
	- m_-(\mustin)'' \big) y_1
	-   \big(m_+(\Upl)' - m_-(\Uin)'{\big)}
	 && \text{ on $\{r=\qstaru\}$,}
	\\
	\label{linstab:shell:7}
	& \frac{2\dot{\delta}}{\delta} y_2
	=
	-\big(m_+(\mu^\star_+)''
	- m_-(\mustext)'' \big) y_2
	-   \big(m_+(\Upl)' -m_-(\Uext)'{\big)}
	 && \text{ on $\{r=\qstard\}$,}
	\\
	\label{linstab:shell:8}
	& (\Uext)' =0
	 && \text{ on $\{r=R\}$,} 
\end{alignat}
\end{subequations}
and 
\begin{align}
 \label{linstab:shell:9} & \Uin(r = 0) < \infty.
\end{align}
As a consequence of \eqref{linstab:shell:4} and \eqref{linstab:shell:5}, we also have
\begin{alignat}{2}
	\label{linstab:shell:extra:1}
	\Upl  -  \Uin &= (\mustin)'  y_1
	- (\mustpl)' y_1
	\quad  &&\text{on $\{r=\qstaru\}$,}
	\\
		\label{linstab:shell:extra:2}
	\Upl  -  \Uext &=(\mustext)'  y_2
	- (\mustpl)' y_2
	\quad  &&\text{on $\{r=\qstard\}$.}
\end{alignat}

\subsubsection{Radial multilayered solutions to the perturbed system on finite domains}
Let us now move to solving {the} system {\eqref{Linstabshell}} for finite domains, i.e., $R< \infty$. In Appendix {\ref{app:radialshell:perturb}} we detail the derivation of the following analytical {formulas} for $\Uin, \Upl, \Uext$. For $r \in (0,\qstaru)$:
\begin{align*}
\Uin(r) = \Big(\alpha \frac {d-1}{(\qstaru)^2} \Big( 1 + \frac {\zeta_{\ell,d}}{d-1}\Big)
	- \Lm  b_{d,-}^{{\rm in}}(\qstaru) \III_1(\Lm \qstaru)\Big) \frac {y_1} { \Iell(\Lm  \qstaru)} \Il(\Lm r){,}
\end{align*}
while for $r \in (\qstaru, \qstard)$:
\begin{align*}
\Upl(r) & = \frac  { y_1 F_1 \qudst\Kell(\Lp\qstard) -    y_2 F_2 \qudst\Kell(\Lp\qstaru)}{\Iell(\Lp \qstaru)\Kell(\Lp \qstard)-\Iell(\Lp \qstard)\Kell(\Lp \qstaru)} \Il(\Lp r) \\
& \quad + \frac  {\Iell(\Lp\qstaru) y_2 F_2\qudst -   \Iell (\Lp \qstard) y_1 F_1 \qudst}{\Iell(\Lp \qstaru)\Kell(\Lp \qstard)-\Iell(\Lp \qstard)\Kell(\Lp \qstaru)} \Kl(\Lp r),
\end{align*}
where
\begin{align*}
 F_1\qudst & = 
\Big(\alpha \frac {d-1}{(\qstaru)^2} \Big( 1 + \frac {\zeta_{\ell,d}}{d-1}\Big)
	-\Lp  b_{d,+} \qudst\III_1 (\Lp \qstaru)
	+ \Lp  c_{d,+} \qudst\KKK_1(\Lp\qstaru)	
	\Big),
	\\
F_2\qudst & =
	\Big(-\alpha \frac {d-1}{(\qstard)^2} \Big( 1 + \frac {\zeta_{\ell,d}}{d-1}\Big)
	-\Lp  b_{d,+} \qudst\III_1 (\Lp \qstard)
	+ \Lp  c_{d,+} \qudst\KKK_1(\Lp\qstard)\Big) ,
\end{align*}
and for $r \in (\qstard, R)$:
\begin{align*}
\Uext(r) & =  y_2 \Big(- \alpha \frac {d-1}{(\qstard)^2} \Big( 1 + \frac {\zeta_{\ell,d}}{d-1}\Big)
	- \Lm  b_{d,-}^{{\rm ext}} (\qstard)\III_1(\Lm \qstard)
	+ \Lm  c_{d,-}^{{\rm ext}} (\qstard)\KKK_1(\Lm \qstard) 
	\Big) \\
	& \quad \times \left [ \frac{\Iell(\Lm r)}{\Iell(\Lm  \qstard)  - \frac{\Il' \LmR}{\Kl' \LmR} \Kell(\Lm  \qstard)} + \frac{\Kell(\Lm r)}{ \Kell(\Lm  \qstard)  -  \frac{\Kl' \LmR}{\Il' \LmR} \Iell(\Lm  \qstard)} \right ].
\end{align*}
Recalling the analytical {formulas \eqref{def:shell:mu}} for $\mustin, \mustpl, \mustext$, with the above expressions for $\Uin, \Upl, \Uext$ we can express the ordinary differential system \eqref{linstab:shell:6} and \eqref{linstab:shell:7} as
\begin{align}\label{shell:perturbation matrix}
	& 2 \dot{\delta} 
	\begin{pmatrix}
	y_{1 }\\
	y_{2 }
	\end{pmatrix}
	=
	\delta \begin{pmatrix}
	A_{11}  & A_{12} \\
	A_{21} & A_{22} 
	\end{pmatrix}
	\begin{pmatrix}
	y_{1 }\\
	y_{2 }
	\end{pmatrix}
	 = \delta \mathbf{A} \begin{pmatrix}
	y_{1 }\\
	y_{2 }
	\end{pmatrix},
\end{align}
where the entries of the matrix $\mathbf{A}$ can be found in \eqref{app:radialshell:matrix}. For a fixed $\ell$, if either one of the two eigenvalues of $\mathbf{A}$ is positive, then instability of the radial multilayered structure can be observed. In contrast to the flat multilayered setting, the above matrix is not symmetric. We can numerically identify parameter regions where $\mathbf{A}$ exhibits at least one positive eigenvalue for a specific value of $\ell$.

\renewcommand{\thefootnote}{\arabic{footnote}}
\section{Numerical Computations}
\newcommand{\nabs}{\nabla_{\!s}}
\newcommand{\id}{{\rm id}}
\renewcommand{\vec}{\boldsymbol}

In this section{,} we present some numerical simulations for the 
moving free boundary problem \eqref{SharpI}. To this end, we will first
introduce an appropriate weak formulation, which is then discretized with the
help of piecewise linear finite elements. Here we will employ an unfitted
description of the moving interface. That is, the finite element triangulations
for the bulk $\Omega$ and for the approximation of the interface 
$\Sigma(t)$ are completely independent.
Finally, we will present several numerical computations in 2d and 3d.

The governing equations \eqref{SharpI} are very close to the 
Mullins--Sekerka problems studied in \cite{dendritic,crystal}. Here we will
adapt these approaches to the novel aspects encountered in \eqref{SharpI}:
the kinetic terms $\rho_\pm \mu$ in \eqref{sharpI:1:p}, \eqref{sharpI:1:m},
the phase-dependent diffusion coefficients $m_\pm$ and forcing terms $S_\pm$ in 
\eqref{sharpI:1:p}, \eqref{sharpI:1:m} and
the forcing term $S_I$ in \eqref{sharpI:4}.

We begin by introducing a natural variational formulation. Multiplying
\eqref{sharpI:1:p}, \eqref{sharpI:1:m} with smooth test functions, integrating
over $\Omega^\pm$, performing integration by parts and noting
\eqref{sharpI:2}, \eqref{sharpI:4} yields:
Given $\Sigma(0)$, with $\partial\Sigma(0) \subset\partial\Omega$,
for $t > 0$ find $\Sigma(t)$, with $\partial\Sigma(t) \subset\partial\Omega$,
and separating $\Omega$ into $\Omega^+(t)$ and $\Omega^-(t)$, as well as
$\mu \in H^1(\Omega)$ and $\kappa \in L^2(\Sigma(t))$ such that
\begin{subequations}
\label{eq:2}
\begin{align}
& \left(m \nabla\mu, \nabla\varphi\right) 
+ \left(\rho\mu, \varphi\right)
- 2\left\langle \mathcal{V},\varphi\right\rangle_{\Sigma(t)} 
=\left(S, \varphi\right)+S_I \left\langle 1,\varphi\right\rangle_{\Sigma(t)}
\quad\forall \varphi\in H^1(\Omega),\label{eq:2aa} \\ &
 \left\langle \mu - \alpha\kappa , \chi \right\rangle_{\Sigma(t)} = 0 
\quad\forall \chi \in L^2(\Sigma(t))
,\label{eq:2bb} \\
& \left\langle \kappa\bnu,\vec\eta \right\rangle_{\Sigma(t)} 
+ \left\langle \nabs\vec\id,\nabs\vec\eta 
\right\rangle_{\Sigma(t)} = 0 
\quad\forall \vec\eta \in [H^1(\Sigma(t))]^d \text{ with }
\vec\eta \cdot \vec n = 0 \text{ on } \partial\Sigma(t), \label{eq:2cc}
\end{align}
\end{subequations}
where we have also recalled a weak formulation for the curvature $\kappa$,
see, e.g., \cite[Remark~22]{bgnreview}. In addition, we let $(\cdot,\cdot)$ and
$\langle\cdot,\cdot\rangle_{\Sigma(t)}$ denote the $L^2$--inner products 
in $\Omega$ and on $\Sigma(t)$, respectively. 
Moreover, we have introduced the notation
$m{(t,\cdot)} = m_+ \charfcn{\Omega^+(t)} + m_- \charfcn{\Omega^-(t)}$, 
and similarly for $\rho{(t,\cdot)}$ and $S{(t,\cdot)}$, where
$\charfcn{\Omega^\pm(t)}$ denotes the characteristic function on 
$\Omega^\pm(t)$.
Finally, $\vec\id$ denotes the
identity function in $\mathbb R^d$.

\newcommand{\sigmaO}{o}
\newcommand{\SmD}{\mathcal S^n}
\newcommand{\Whm}{V(\Sigma^n)}
\newcommand{\Vhm}{\underline{V}(\Sigma^n)}
\newcommand{\Vhmd}{\underline{V}_\partial(\Sigma^n)}
\newcommand{\bB}{\mathbb{B}}
Let $\Delta t > 0$ be a chosen time step size, and let $t_n = n\Delta t$,
$n \geq 0$.
Let $\Omega$ be a polyhedral domain. For $n\geq0$, let ${\cal T}^n$ 
be a regular partitioning of $\Omega$ into disjoint open simplices, so that 
$\overline{\Omega}=\cup_{\sigmaO\in{\cal T}^n}\overline{\sigmaO}$. 
Associated with ${\cal T}^n$ is the finite element space
$\SmD := \{\chi \in C^0(\overline{\Omega}) : \chi\!\mid_{\sigmaO} 
\text{ is affine } \ \forall \sigmaO \in {\cal T}^n\}$.

Let $\Sigma^{n}$ be a polyhedral hypersurface, see
\cite[Definition~41]{bgnreview}, approximating the
surface $\Sigma(t_n)$, $n\geq0$.
In particular, let $\Sigma^n=\bigcup_{j=1}^{J^n_\Sigma} 
\overline{\sigma^n_j}$,
where $\{\sigma^n_j\}_{j=1}^{J^n_\Sigma}$ is a family of mutually disjoint open 
$(d-1)$-simplices, $J^n_\Sigma \in \enne$. 
Then let
\begin{equation*}
\Vhm := \{\vec\chi \in [C^0(\Sigma^n)]^d:\vec\chi\!\mid_{\sigma^n_j}
\mbox{ is affine}\ \forall\ j=1,\ldots, J^n_\Sigma\} 
=: [\Whm]^d,
\end{equation*}
where $\Whm \subset H^1(\Sigma^n)$ is the space of scalar continuous
piecewise linear functions on $\Sigma^n$, and define
\[
\Vhmd := \{ \vec\chi \in \Vhm : \vec\chi \cdot \vec n = 0 \text{ on }
\partial\Sigma^n \}.
\]
For later purposes, we also introduce 
$\pi_{\Sigma^n}: C^0(\Sigma^n)\to \Whm$, the standard Lagrange
interpolation operator.

As in the continuous case, we let $\langle\cdot,\cdot\rangle_{\Sigma^n}$ denote
the $L^2$--inner product on $\Sigma^n$. In addition,
for scalar and vector valued functions $v,w$ that are piecewise continuous, 
with possible jumps across the edges of $\{\sigma_j^n\}_{j=1}^{J^n_\Sigma}$,
we introduce the mass lumped inner product
$\langle\cdot,\cdot\rangle_{\Sigma^n}^h$ as
\begin{equation*}
\langle v, w \rangle^h_{\Sigma^n} :=
\tfrac1d \sum_{j=1}^{J^n_\Sigma} |\sigma^n_j|\sum_{k=1}^{d} 
(v\cdot w)((\vec{q}^n_{j_k})^-),
\end{equation*}
where $\{\vec{q}^n_{j_k}\}_{k=1}^{d}$ are the vertices of $\sigma^n_j$,
and where we define $v((\vec{q}^n_{j_k})^-):=
\underset{\sigma^n_j\ni \vec{p}\to \vec{q}^n_{j_k}}{\lim}\, v(\vec{p})$.

Given $\Sigma^n$, we 
let $\Omega^n_+$ denote the interior of $\Sigma^n$ and let
$\Omega^n_-$ denote the exterior of $\Sigma^n$, so that
$\Sigma^n = \partial \Omega^n_+ = \overline\Omega^n_+ \cap 
\overline\Omega^n_-$. 
We then partition the elements of the bulk mesh 
$\mathcal{T}^n$ into interior, exterior and interfacial elements as follows.
Let
\begin{align}
\mathcal{T}^n_- & := \{ o \in \mathcal{T}^n : o \subset
\Omega^n_- \} , \qquad 
\mathcal{T}^n_+ := \{ o \in \mathcal{T}^n : o \subset
\Omega^n_+ \} , \nonumber \\
\mathcal{T}^n_{\Sigma^n} & := \{ o \in \mathcal{T}^n : o \cap
\Sigma^n \not = \emptyset \} . \label{eq:partT}
\end{align}
Clearly $\mathcal{T}^n = \mathcal{T}^n_- \cup \mathcal{T}^n_+ \cup
\mathcal{T}^n_{\Sigma^n}$ is a disjoint partition, which in practice
can easily be found{,} e.g.{,} with the Algorithm~4.1 in \cite{crystal}.
In addition, we let $\vec{\nu}^n$ denote the piecewise constant unit normal 
to $\Sigma^n$, pointing into $\Omega^n_+$. We also
introduce the vertex normal vector $\vec\omega^{n}\in \Vhm$ as
the mass-lumped $L^2$--projection of $\vec\nu^n$ onto $\Vhm$, that is
\begin{equation} \label{eq:nuhomegah}
\langle \vec\omega^n, \vec\eta \rangle^h_{\Sigma^n} = 
\langle \vec\nu^n, \vec\eta \rangle_{\Sigma^n}
\quad\forall\vec\eta\in \Vhm.
\end{equation}
On recalling \eqref{eq:partT}, we introduce the piecewise constant functions
$\rho_h^n$, $m_h^n$ and $S_h^n$ via
\begin{equation} \label{eq:rhoma}
\rho_h^n\!\mid_{o} = \begin{cases}
\rho_- & o \in \mathcal{T}^n_-, \\
\rho_+ & o \in \mathcal{T}^n_+, \\
\tfrac12(\rho_- + \rho_+) & o \in \mathcal{T}^n_{\Sigma^n},
\end{cases}
\ \text{ and}\quad
m_h^n\!\mid_{o} = \begin{cases}
m_- & o \in \mathcal{T}^n_-, \\
m_+ & o \in \mathcal{T}^n_+, \\
\tfrac12(m_- + m_+) & o \in \mathcal{T}^n_{\Sigma^n},
\end{cases}
\end{equation}
and analogously for $S_h^n$.

In addition to $\Sigma^n$, we will denote by $\mu_h^n$, and $\kappa_h^n$ 
the discrete approximations to $\mu(t_n)$ and $\kappa(t_n)$,
respectively.
We will parameterize the new discrete surface 
$\Sigma^{n+1}$ over $\Sigma^n$, with the help of a parameterization
$\vec{X}^{n+1} \in \Vhm$.

Our finite element scheme to approximate \eqref{eq:2} can then be formulated as
follows:
Given $\Sigma^0$, for $n\geq0$, find $(\mu_h^{n+1},\vec X^{n+1},\kappa_h^{n+1})
\in \SmD \times \Vhm \times \Whm$ such that 
$\vec X^{m+1} - \vec\id\!\mid_{\Sigma^n} \in \Vhmd$ and
\begin{subequations} \label{eq:MSHG}
\begin{align}
& \left(m_h^n\nabla\mu_h^{n+1}, \nabla\varphi\right) 
+\left(\rho_h^n \mu_h^{n+1}, \varphi \right)^h
- 2 \left\langle \pi_{\Sigma^n}\left[
\frac{\vec X^{n+1}-\vec\id}{\Delta t} \cdot \vec\omega^n \right], \varphi
\right\rangle_{\Sigma^n} 
\nonumber \\ & \hspace{6cm}
= \left(S_h^{n}, \varphi\right)^h
+S_I\left\langle 1, \varphi \right\rangle_{\Sigma^n}
\quad\forall \varphi \in \SmD, \label{eq:MSHGa}\\
& \left\langle \mu_h^{n+1}, \chi \right\rangle_{\Sigma^n} 
- \alpha \left\langle \kappa_h^{n+1}, 
\chi\right\rangle_{\Sigma^n}^h =0  \quad\forall \chi \in \Whm
, \label{eq:MSHGb} \\
& \left\langle \kappa_h^{n+1}\vec\nu^n, \vec\eta \right\rangle_{\Sigma^n}^h + 
\left\langle \nabs\vec X^{n+1}, \nabs\vec\eta \right\rangle_{\Sigma^n}
= 0 \quad\forall \vec\eta \in \Vhmd,
 \label{eq:MSHGc} 
\end{align}
\end{subequations}
and set $\Sigma^{n+1} = \vec X^{n+1}(\Sigma^n)$. The linear system
\eqref{eq:MSHG} is closely related to \cite[(119)]{bgnreview}. In fact, under
a mild assumption on the compatibility between the bulk mesh $\mathcal T^n$
and the surface mesh $\Sigma^n$, it is 
straightforward to prove existence and uniqueness of a solution to
\eqref{eq:MSHG}, see the proof of Theorem~109 in \cite{bgnreview} for details.

We implemented the scheme \eqref{eq:MSHG} with the help of the finite element 
toolbox ALBERTA, see \cite{Alberta}.
To increase computational efficiency, we employ adaptive bulk meshes
that feature small elements in $\mathcal T^n_{\Sigma^n}$ and coarser elements
in $\mathcal T^n_\pm$, recall \eqref{eq:partT}. In addition, due to the growth
in time of the discrete interfaces $\Sigma^n$, we adaptively refine 
$\Sigma^n$, $n\geq1$, in regions where the elements have grown beyond a certain
threshold compared to the largest element in $\Sigma^0$. We refer to
\cite{dendritic} for the precise details. For some simulations in
3d topological changes need to be applied to $\Sigma^{n+1}$, to avoid
self-intersection and to allow for pinch-offs, for example. For these
topological changes we employ the algorithm from \cite{BenninghoffG17},
see also \cite{Benninghoff15}. In addition, for simulations where topological
changes are necessary, we also improve the mesh quality by applying a mesh
smoothing step as described in \cite[(2.19), (2.18)(iii)]{willmore} 
after each time step.

For the solution of the linear systems of equations arising from 
\eqref{eq:MSHG} we use the direct factorization package UMFPACK, 
see \cite{Davis04}, in 2d, and employ a Schur complement approach that is
solved with a preconditioned conjugate gradient method in 3d, see
\cite{dendritic} for details. Some further information on the implementation
and solution of systems like \eqref{eq:MSHG} can also be found in
Remarks~107 and 111 in \cite{bgnreview}.

\subsection{Numerical simulations for flat interfaces}

We begin with a numerical simulation for a known exact solution to
\eqref{SharpI}, that is, a moving flat front in a rectangular domain in $2d$.
In particular, we let $\Omega = (0,8)^2$ and choose the 
parameters $\alpha=0.1$, $m_\pm=1$, $\rho_\pm = 1$, $S_\mp=\pm1$, and $S_I=0$.
As the initial interface $\Sigma(0)$ we choose a flat vertical line at
positions $q(0) = 0.3$ and $q(0) = 7.7$, respectively, with its normal pointing
to the left.
A comparison between $q(t)$ (which is the solution to the ordinary differential equation (5.8) of \cite{Active_drops}) and the positions $q_h(t_n)$ of the discrete
interfaces $\Sigma^n$, where
\[
q_h(t_n) = \max\{x : (x,y)^\top \in \Sigma^n \},
\]
is shown in Figure~\ref{fig:flat1}. We observe an excellent agreement.
\begin{figure}
\center
\includegraphics[angle=-90,width=0.45\textwidth]{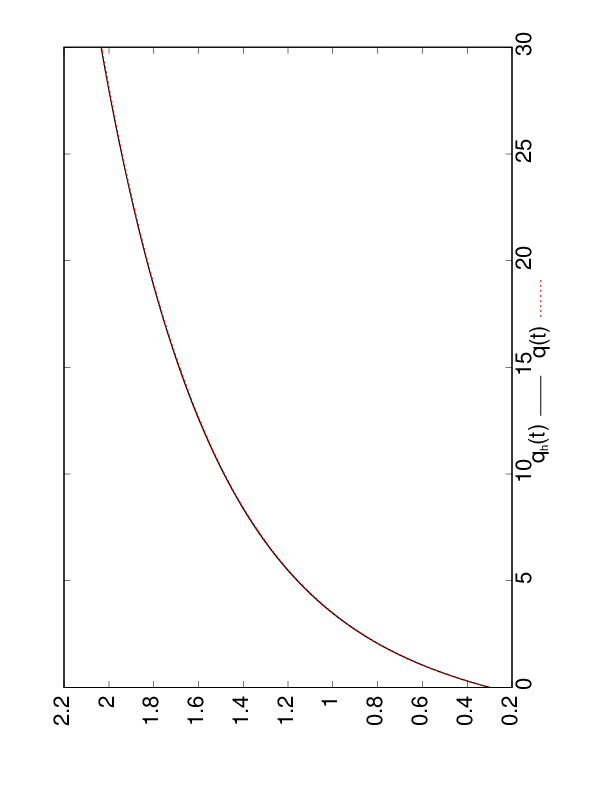}
\includegraphics[angle=-90,width=0.45\textwidth]{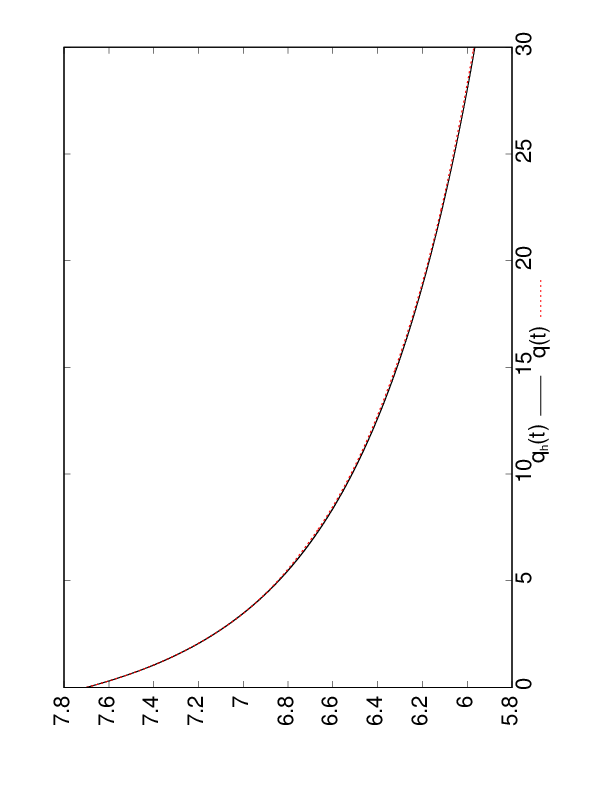}
\caption{($\Omega = (0,8)^2$)
Moving front for $\alpha=0.1$, $m_\pm=1$, $\rho_\pm = 1$, 
$S_\mp=\pm1$, $S_I=0$ with $q(0)=0.3$ (left) and $q(0)=7.7$ (right).
We plot $q(t)$ and $q_h(t)$ over time.
}
\label{fig:flat1}
\end{figure}%

Next we investigate the stability of an initially flat interface. To this end, 
we choose $\Omega = (0,8)^2$ and let
$\alpha=0.2$, $m_\pm=0.1$,
$\rho_\pm=0.1$, $S_\mp=\pm1$, $S_I=0$.
For the initial interface $\Sigma(0)$ we choose a vertical line
at position ${q(0) = 4}$ with an added perturbation. In fact, we let
\[
\Sigma(0) = \Big{\{} \Big{(} {4+}\sum_{\ell=1}^{20} \sigma_\ell \cos \Big (\frac{\ell\pi y}8 \Big ), y \Big)^\top : y
\in [0,8] \Big{\}},
\]
with $|\sigma_\ell| < 0.01$, $\ell=1,\ldots,20$, {as} some random coefficients.
The evolution shown in Figure~\ref{fig:flatunstable2}
indicates that the mode $\ell=10$ is growing the fastest and eventually
dominates. This observation is supported by Table~\ref{tbl:flatunstable2} where we list the values of the perturbation amplification factor (5.13) in \cite{Active_drops}, playing an analogous role to the right-hand side of \eqref{strucperturb}, corresponding to perturbation modes $\ell = 1, \dots, 20$ for the parameter setting of Figure~\ref{fig:flatunstable2}. We see that mode $\ell = 10$ exhibits the largest amplification of the perturbations.
\begin{table}[h]
\centering
\begin{tabular}{|c|c|c|c|c|c|}
\hline
$\ell$ & $1$ & $2$ & $3$ & $4$ & $5$ \\
\hline
&  0.1398   & 0.5098   & 1.0027  &  1.5379  &  2.0642   \\
\hline\hline
$ \ell$ & $6$ & $7$ & $8$ & $9$ & $10$ \\
\hline
&  2.5474  &  2.9622  &  3.2878  &  3.5059   & \textbf{3.5995} \\
\hline\hline
$\ell$ & 11 & 12 & 13 & 14 & 15  \\
\hline
&3.5525  &  3.3491  &  2.9742  &  2.4125   & 1.6491   \\
\hline\hline
$\ell$ & 16 & 17 & 18 & 19 & 20 \\
\hline
&  0.6691 & $ -0.5422$ &  $ -1.9995$  & $ -3.7177$ &   $-5.7112$ \\
\hline
\end{tabular}
\caption{Values of the amplification factor (5.13) in \cite{Active_drops} corresponding to perturbation modes $\ell = 1, \dots, 20$ for the setting of Figure~\ref{fig:flatunstable2}. The largest value highlighted in bold corresponds to mode $\ell = 10$.}
\label{tbl:flatunstable2}
\end{table}

\begin{figure}
\center
\includegraphics[angle=-90,width=0.45\textwidth]{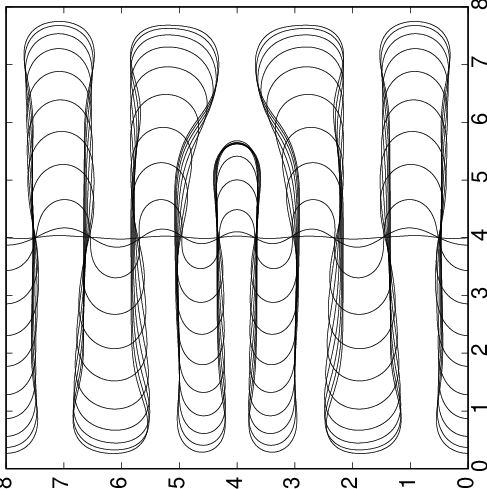} 
\includegraphics[angle=-90,width=0.45\textwidth]{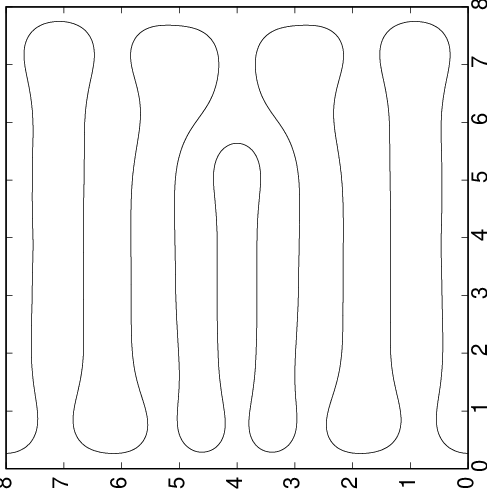} 
\caption{($\Omega = (0,8)^2$) $\alpha=0.2$, $m_\pm=0.1$,
$S_\mp=\pm1$, $\rho_\pm=0.1$, $S_I=0$.
Evolution for a perturbed flat interface at position ${q(0) = 4}$.
We show the solution at times $t=0,1,\ldots,10$, and separately at time $t=10$.}
\label{fig:flatunstable2}
\end{figure}%

Next we consider the evolution of multilayered phases. Hence $\Sigma(0)$ is
made up of two connected components. For a simulation inside the domain
$\Omega = (0,8)^2$ we choose $\Omega^+(0) = (3.7, 4.3) \times (0, 8)$
and $\Omega^+(0) = (0.9, 6.1) \times (0, 8)$, respectively, with
$\Sigma(0) = \Omega\cap\partial\Omega^+$. We then track the positions of the
two discrete interfaces,
\[
q_h(t_n) = \big ( \min \{x : (x,y)^\top \in \Sigma^n \},
\max\{x : (x,y)^\top \in \Sigma^n \} \big ),
\]
and compare them with the exact solution obtained from solving the ODE system {\eqref{q12dot:planar}}
in Figure~\ref{fig:strip2a0011}.
Once again, we notice an excellent agreement between discrete and continuous
solution.
\begin{figure}
\center
\includegraphics[angle=-90,width=0.45\textwidth]{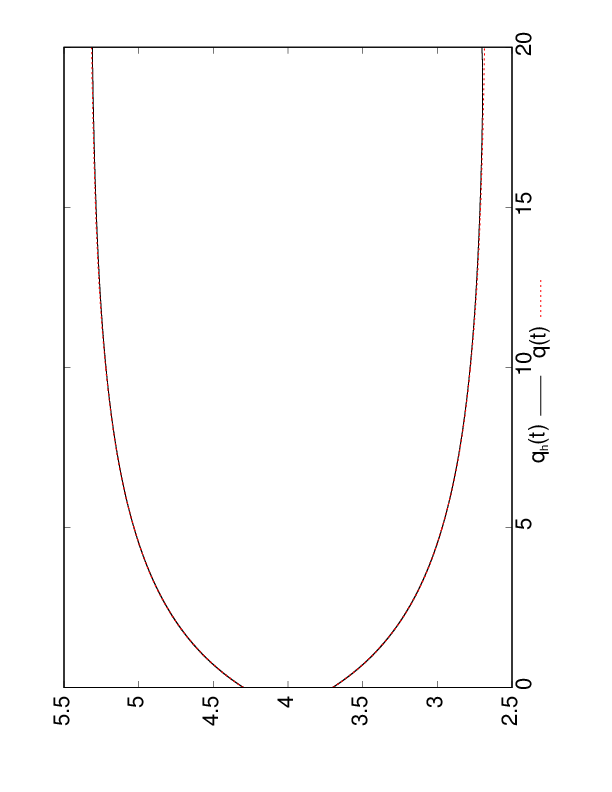} \qquad
\includegraphics[angle=-90,width=0.45\textwidth]{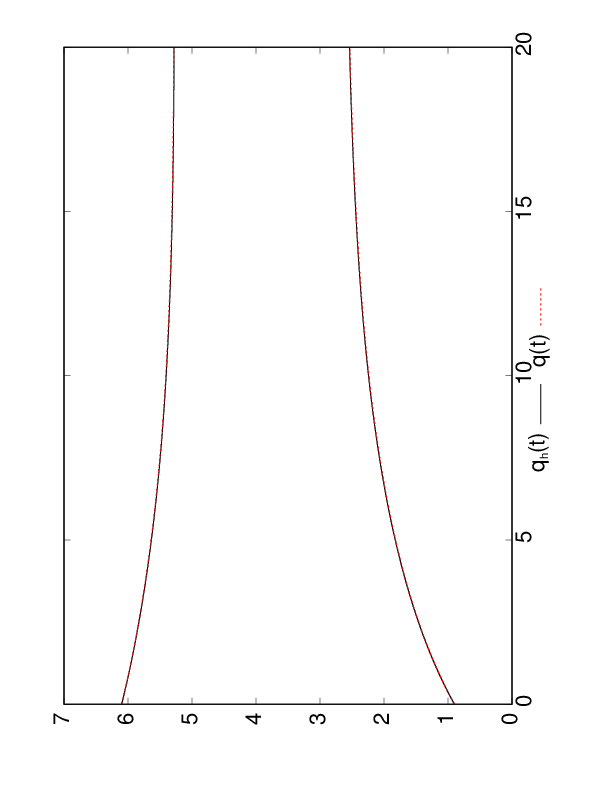}
\caption{($\Omega = (0,8)^2$) $\alpha=0.011$, $m_-=4$, $m_+=1$, $\rho_\pm=1$,
$S_-=0.5$, $S_+=-1$, $S_I=0$
with $q(0)=(3.7, 4.3)$ (left) and $q(0)=(0.9, 6.1)$ (right).
We plot $q(t)$ and $q_h(t)$ over time.
}
\label{fig:strip2a0011}
\end{figure}%

We are interested in the stability behaviour of such multilayered setups. In
the domain $\Omega = (0,8)^2$ we choose two different types of perturbations
for an initial layer $\Omega^+(0) = {(\frac83,\frac{16}3)}
\times(0,8)$
when the physical parameters are $\alpha=0.011$, $m_-=4$, $m_+=1$, 
$\rho_\pm=1$, $S_-=0.5$, $S_+=-1$, $S_I=0$. As initial perturbations we choose
either
\begin{subequations}
\begin{alignat}{2} 
\label{eq:perta} \Sigma(0) & = 
\Big \{ \Big ( \frac 83 + 0.01 \cos \Big(\frac{8\pi y}8 \Big), y \Big)^\top : y \in [0,8] \Big\} \\
\notag & \qquad \cup \Big\{ \Big( \frac{16}3 + 0.01 \cos\Big(\frac{8\pi y}8 \Big), y \Big)^\top : y \in [0,8] \Big\}, \\
\label{eq:pertb}  \text{ or } \ {\Sigma(0)} & =  \Big\{ \Big( \frac83 + \sum_{\ell=1}^{20} \sigma_{-,\ell} \cos \Big(\frac{\ell\pi y}8 \Big ), y \Big)^\top : y \in [0,8] \Big\} \\
\notag & \qquad \cup \Big\{ \Big(\frac{{16}}3 + \sum_{\ell=1}^{20} \sigma_{+,\ell} \cos \Big(\frac{\ell\pi y}8 \Big), y \Big)^\top : y \in [0,8] \Big \},
\end{alignat}
\end{subequations}
with $|\sigma_{\pm,\ell}| < 0.01$, $\ell=1,\ldots,20$, some random coefficients.
Hence in the first case, we use the same deterministic perturbation made up of
the mode $\ell=8$ only, while in \eqref{eq:pertb} we use two different, and
essentially random, perturbations on the two interfaces.
The numerical simulations shown in Figure~\ref{fig:strip2a0011pertnew}
appear to confirm that $\ell=8$ is the most unstable mode. This observation is supported by the eigenvalues $\lambda_1 := A_{11} + A_{12}$ and $\lambda_2 := A_{11} - A_{12}$ of the matrix $\mathbf{A}$ appearing in \eqref{perturbation_matrix:multiplanar} reported in  Table~\ref{tbl:multi_planar_eigens}, where we note that the eigenvalues corresponding to the mode $\ell = 8$ are the largest positive pair.
\begin{table}[h]
\centering
\begin{tabular}{|c|c|c|c|c|c|c|c|c|c|c|}
\hline
$\ell$ & $1$ & $2$ & $3$ & $4$ & $5$ & $6$ & $7$ & $8$ \\
\hline
$\lambda_1$ & 0.0588 &  0.4482  &  0.8976  &  1.3466  &  1.7531  &  2.0825 &   2.3058  &  \textbf{2.3977}   \\
$\lambda_2$&  $-0.1546$  &  0.3001  &  0.8118  &  1.3028  &  1.7326    &2.0735  &  2.3020  &  \textbf{2.3962} \\
\hline\hline
$\ell$ & $9$ & $10$ & $11$ & $12$ & $13$ & $14$ & $15$ & $16$ \\
\hline
$\lambda_1$ &  2.3352  &  2.0967 & 1.6610  &  1.0074  &  0.1157   & $-1.0347$ & $-2.4639$  & $-4.1920$ \\
$\lambda_2 $ & 2.3346  &  2.0964  &  1.6609   & 1.0074 &  0.1157  & $-1.0347$ & $-2.4639$ &  $-4.1920$ \\
\hline
\end{tabular}
\caption{Eigenvalues of the matrix $\mathbf{A}$ in \eqref{perturbation_matrix:multiplanar} for the setting of Figure~\ref{fig:strip2a0011pertnew} corresponding to perturbation modes $\ell = 1, \dots, 16$. The largest positive pair highlighted in bold corresponds to mode $\ell = 8$.}
\label{tbl:multi_planar_eigens}
\end{table}
\begin{figure}
\center
\includegraphics[angle=-90,width=0.4\textwidth]{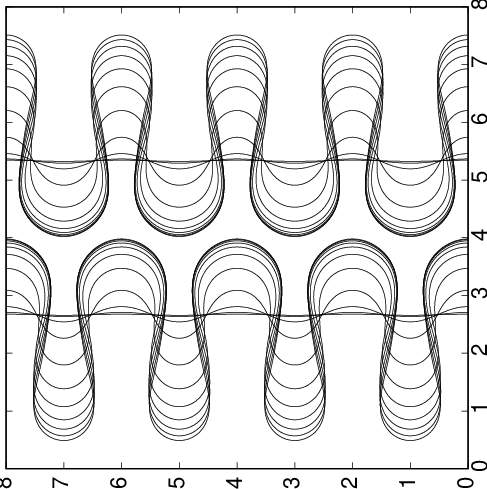} 
\includegraphics[angle=-90,width=0.4\textwidth]{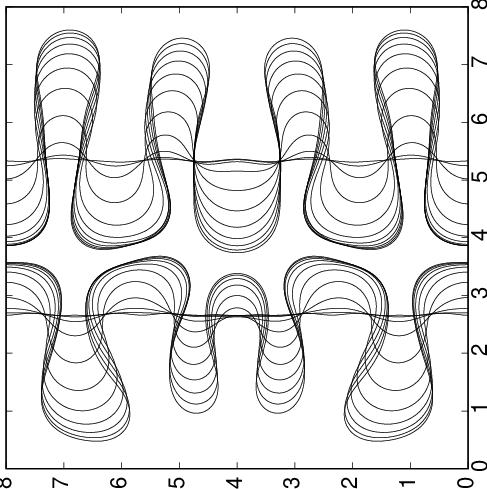} 
\caption{($\Omega = (0,8)^2$) $\alpha=0.011$, $m_-=4$, $m_+=1$, $\rho_\pm=1$,
$S_-=0.5$, $S_+=-1$, $S_I=0$.
We show the solution at times $t=0,1,\ldots,10$, for the 
initial perturbations \eqref{eq:perta} (left) and \eqref{eq:pertb} (right).
}
\label{fig:strip2a0011pertnew}
\end{figure}%

Inside the same computational domain, but now for the physical parameters
$\alpha=0.1$, $m_-=0.1$, $m_+=0.2$, $\rho_-=0.1$, $\rho_+=0.5$,
$S_-=1$, $S_+=-4$, $S_I=1$, we consider an experiment for an initial layer
$\Omega^+(0) = (0.3,0.6)\times(0,8)$ that is perturbed on the left hand
side. In particular, for the initial interface we let
\begin{equation} \label{eq:pertc}
\Sigma(0) =
\Big \{ \Big ( 0.3 + \sum_{\ell=1}^{20} \sigma_{\ell} \cos \Big (\frac{\ell\pi y}8 \Big ), y \Big )^\top : y \in [0,8] \Big \}
\cup
\{0.6\} \times (0,8)
\end{equation}
with $|\sigma_{\ell}| < 0.01$, $\ell=1,\ldots,20$, some random coefficients.
In Figure~\ref{fig:stripunstable2} we can observe that the perturbation on the
left interface grows, seemingly fastest for the mode $\ell=12$, which then
starts to also affect the right interface. Interestingly, for 
the right interface the mode $\ell=14$ appears to be growing the fastest. For
long times, one of the seven fingers is restricted in growth, meaning that
eventually only six fingers continue to grow.
\begin{figure}
\center
\includegraphics[angle=-90,width=0.45\textwidth]{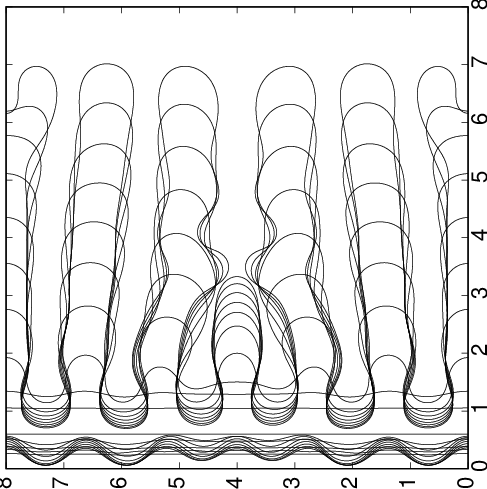} 
\includegraphics[angle=-90,width=0.45\textwidth]{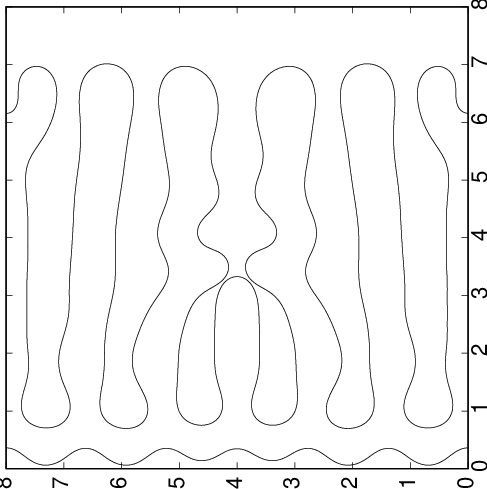} 
\caption{($\Omega = (0,8)^2$) $\alpha=0.1$, $m_-=0.1$, $m_+=0.2$,
$\rho_-=0.1$, $\rho_+=0.5$,
$S_-=1$, $S_+=-4$, $S_I=1$.
We show the solution at times $t=0,1,\ldots,10$, and separately at time $t=10$,
for the initial perturbation \eqref{eq:pertc}.}
\label{fig:stripunstable2}
\end{figure}%

\subsection{Numerical simulations for radially symmetric interfaces}

We begin with a single circular interface inside the domains
$\Omega = \bB^d_4(0) := \{ \vec z \in \mathbb R^d : |\vec z| < 4$,
$d \in \{2,3\}$. For the initial radius we choose $r(0) = 0.25$, while the physical
parameters are chosen as 
$\alpha=1$, $m_-=0.1$, $m_+=1$, $\rho_-=0.1$, $\rho_+=1$, 
$S_-=4$, $S_+=-1$, $S_I=1$.
A comparison between $r(t)$ and the radii $r_h(t_n)$ of the discrete
interfaces $\Sigma^n$, where
\[
r_h(t_n) = \max\{|\vec z| : \vec z \in \Sigma^n \},
\]
is shown in Figure~\ref{fig:2f}. We observe that in both 2d and 3d our
numerical algorithm matches the evolution of the true solution obtained from solving the ODE \eqref{radial:ode} very well.
\begin{figure}
\center
\includegraphics[angle=-90,width=0.45\textwidth]{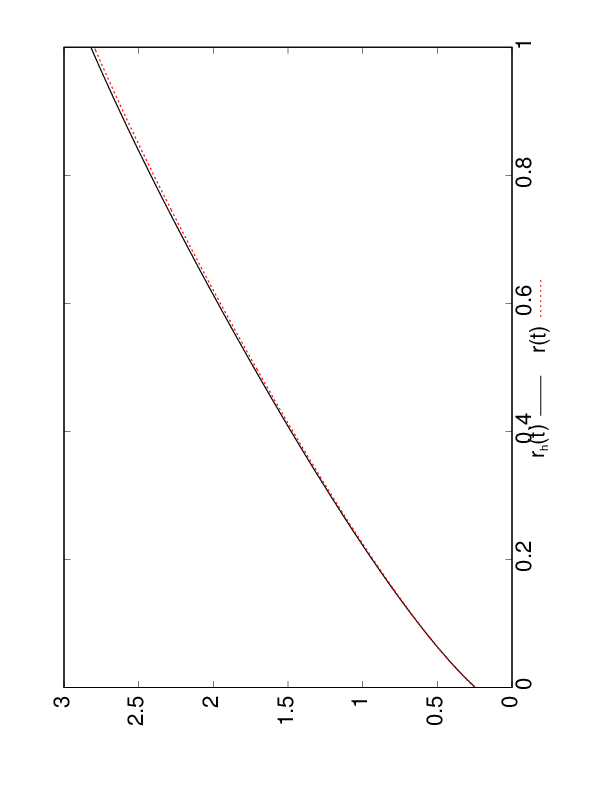}
\includegraphics[angle=-90,width=0.45\textwidth]{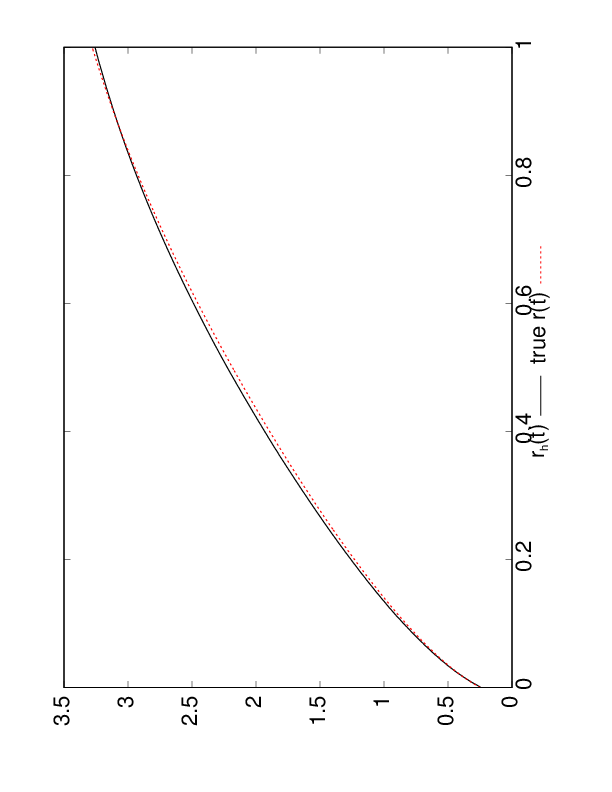}
\caption{($\Omega = \bB^d_4(0)$) 
$\alpha=1$, $m_-=0.1$, $m_+=1$, $\rho_-=0.1$, $\rho_+=1$, 
$S_-=4$, $S_+=-1$, $S_I=1$ with $r(0) = 0.25$ for $d=2$ (left)
and $d=3$ (right).
We plot $r(t)$ and $r_h(t)$ over time.
}
\label{fig:2f}
\end{figure}%

Within $\Omega = \bB^2_4(0)$ we also consider a multilayered setup for the
parameters $\alpha=1$, $m_-=0.1$, $m_+=1$,
$\rho_-=0.1$, $\rho_+=1$, $S_-=4$, $S_+=-1$, $S_I=1$. In particular,
we let $\Omega^+(0) = \bB^2_2(0) \setminus \overline{\bB^2_1(0)}$ so that
$\Sigma(0) = \partial\bB^2_1(0) \cup \bB^2_2(0)$.
We then track the radii of the two discrete interfaces,
\[
r_h(t_n) = (\min\{|\vec z| : \vec z \in \Sigma^n \},
\max\{|\vec z| : \vec z \in \Sigma^n \}),
\]
and compare them with the exact solution obtained from solving the ODE system \eqref{radial:shell:ode} in Figure~\ref{fig:3}.
\begin{figure}
\center
\includegraphics[angle=-90,width=0.45\textwidth]{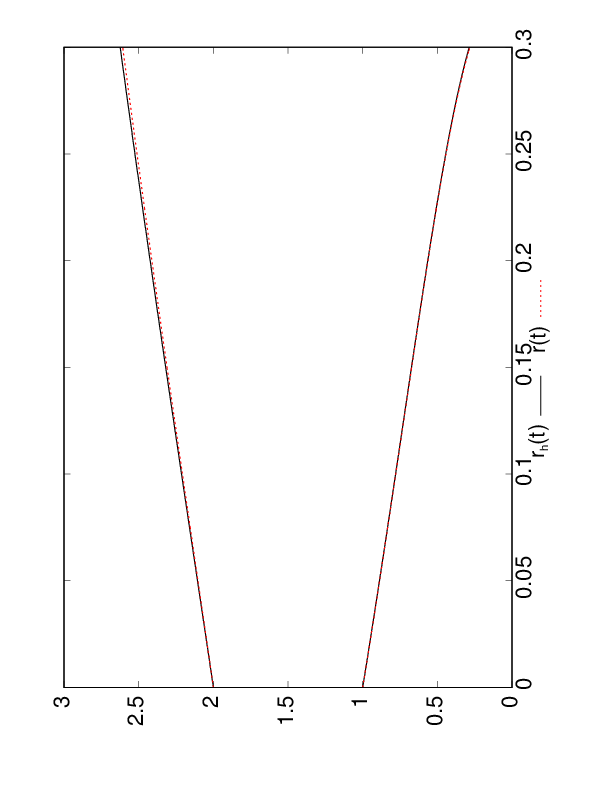}
\caption{($\Omega = \bB^2_4(0)$) $\alpha=1$, $m_-=0.1$, $m_+=1$,
$\rho_-=0.1$, $\rho_+=1$, $S_-=4$, $S_+=-1$, $S_I=1$
with $r(0)=(1, 2)$.
We plot $r(t)$ and $r_h(t)$ over time.
}
\label{fig:3}
\end{figure}%

Next we would like to investigate how a radially symmetric layer phase evolves
in the presence of perturbations. Inside the computational domain
$\Omega = \bB^2_5(0)$ we let $\alpha=0.1$, $m_\pm=1$, $\rho_\pm=0.4$,
$S_\mp=\pm1$, $S_I=0$. For the initial interface we choose
\begin{align*}
\Sigma(0) & = {\Big \{ } {\Big(}2.23 + 0.01 \cos{\Big(}4\theta - \frac\pi6{\Big)\Big)}
\binom{\cos\theta}{\sin\theta} : \theta \in [0,2\pi) {\Big\}} \\ & \quad
\cup
{\Big\{} {\Big(}4.16 + 0.01 \cos{\Big(}4\theta - \frac\pi6{\Big)\Big)}
\binom{\cos\theta}{\sin\theta} : \theta \in [0,2\pi) {\Big\}},
\end{align*}
{i.e., two circles with radii $2.23$ and $4.16$ with the same perturbation of
size $0.01$ for the single mode $\ell=4$.
The results are shown in Figure~\ref{fig:fig7noise2}.
We remark that in the absence of non-radial perturbations, numerical solutions of the ODE system \eqref{radial:shell:ode} yield in the large time limit an inner and outer equilibrium radii of approximately 2.23 and 4.16, respectively. In Table~\ref{tbl:multiradialshells_d2}, we report on the two eigenvalues (denoted as $\lambda_1$ and $\lambda_2$) of the perturbation matrix $\mathbf{A}$ in \eqref{shell:perturbation matrix} corresponding to perturbation modes $\ell=1, \dots, 10$, where we see that mode $\ell = 4$ exhibits the largest positive eigenvalue, which supports the above numerical observation on the growth of the mode $\ell = 4$ perturbation.
We notice that the
inner perturbation grows much more profoundly than the outer one.
This is motivated by the fact that the eigenvector corresponding to the positive eigenvalue of the unstable $4$ mode is $(0.9924, 0.1228)$. Referring back to equation \eqref{eq:y12}, we see that $y_1 = 0.9924$ (corresponding to the inner radius) is roughly  eight times larger than $y_2 = 0.1228$ (corresponding to the outer radius) which supports the observation that the inner perturbation is amplified more drastically.}
\begin{table}[h]
\centering
\begin{tabular}{|c|c|c|c|c|c|}
\hline
$\ell$ & $1$ & $2$ & $3$ & $4$ & 5 \\
\hline
$\lambda_1$ & $-0.1579$  &  0.1013  &  0.3544  &  \textbf{0.3995}  &  0.0751  \\
$\lambda_2$&   $-0.8508$  &  $-0.6310$ & $-0.4461$  &  $-0.2891$   & $-0.1663$ \\
\hline\hline
$\ell$ & 6 & 7 & 8 & 9 & 10 \\
\hline
$\lambda_1$ &  $-0.7916$  &  $-2.2748 $ & $-4.5112$ &  $-7.6099$  &$-11.6804$ \\
$\lambda_2$ &$ -0.0485$  &   0.0084  &  0.0182 & $-0.0296$ & $-0.1439$ \\
\hline
\end{tabular}
\caption{Eigenvalues of the matrix $\mathbf{A}$ in \eqref{shell:perturbation matrix} for the setting of Figure~\ref{fig:fig7noise2} corresponding to perturbation modes $\ell = 1, \dots, 10$. The largest positive eigenvalue highlighted in bold corresponds to mode $\ell = 4$.}
\label{tbl:multiradialshells_d2}
\end{table}
\begin{figure}
\center
\mbox{
\includegraphics[angle=-90,width=0.25\textwidth]{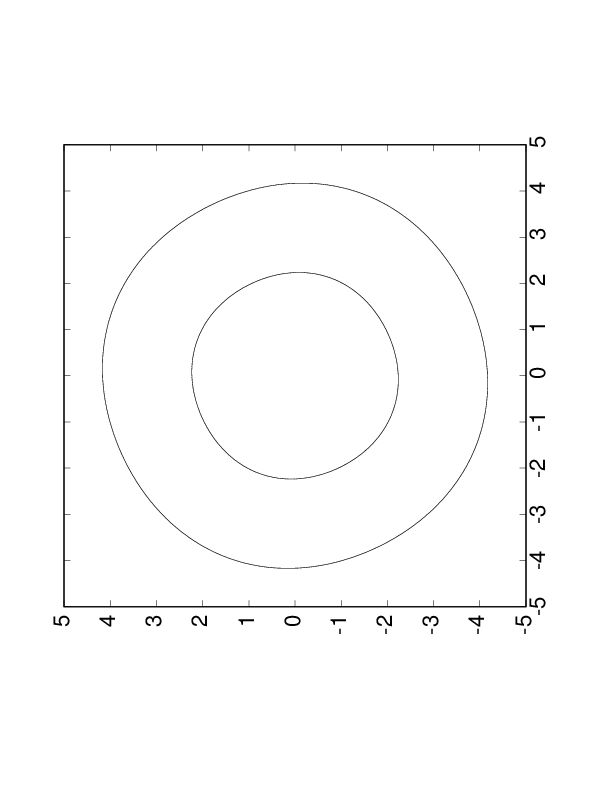} 
\includegraphics[angle=-90,width=0.25\textwidth]{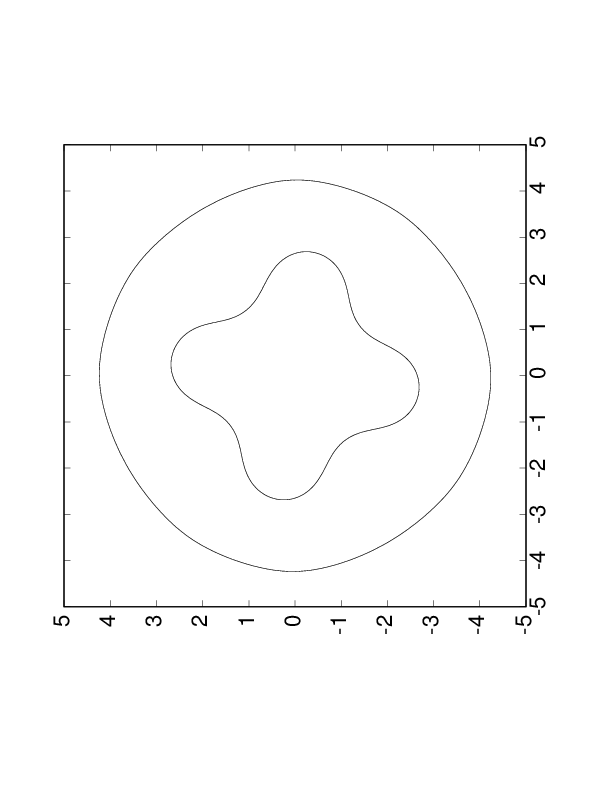}
\includegraphics[angle=-90,width=0.25\textwidth]{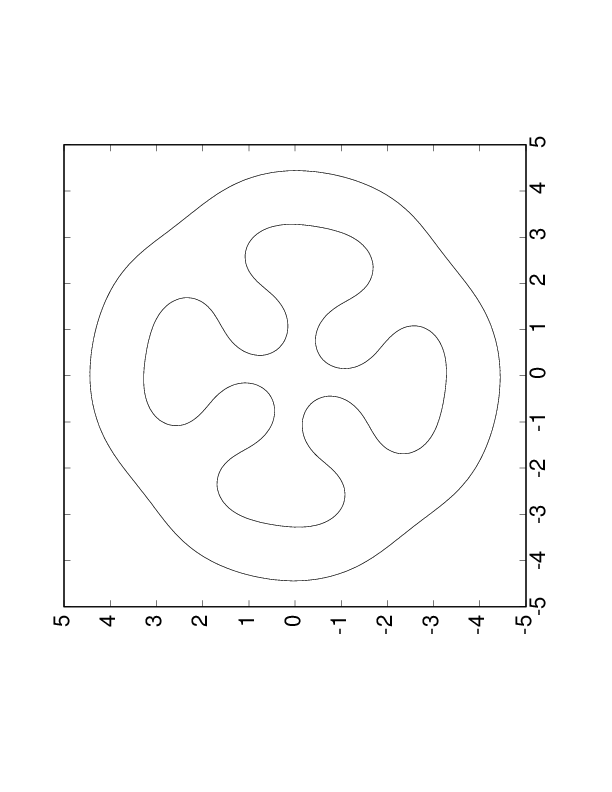}
\includegraphics[angle=-90,width=0.25\textwidth]{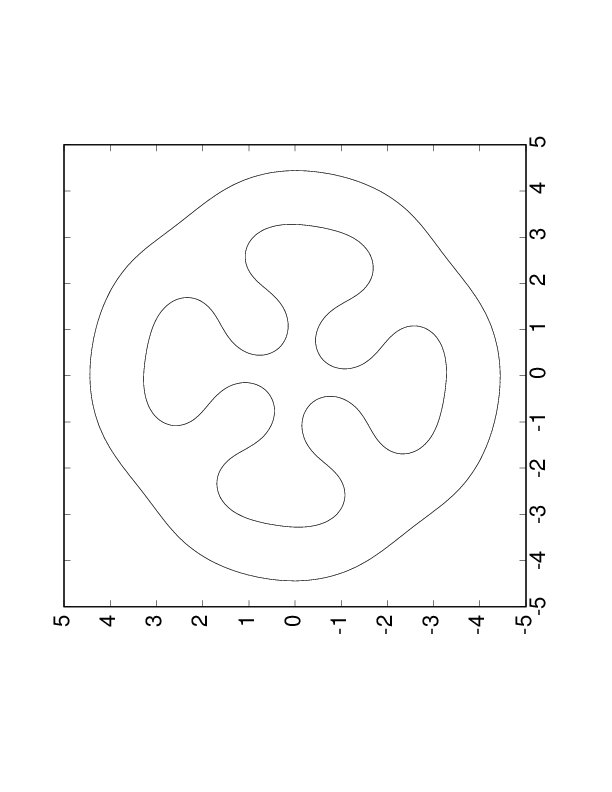}
}
\caption{($\Omega = \bB^2_5(0)$) $\alpha=0.1$, $m_\pm=1$, $\rho_\pm=0.4$,
$S_\mp=\pm1$, $S_I=0$
with $r(0)=(2.23, 4.16)$.
We show the solution at times $t=0,10,50,200$.}
\label{fig:fig7noise2}
\end{figure}%

We also consider two 3d simulations for a phase in the form of a shell. An
experiment analogous to Figure~\ref{fig:3}, but now in the domain
$\Omega = \bB^3_5(0)$ can be seen in Figure~\ref{fig:3dthin0}. Here the
physical parameters are set to $\alpha=0.07$, $m_\pm=1$, $\rho_\pm=2$,
$S_-=1$, $S_+=-4$, $S_I=0.42$, while the initial phase is chosen as
$\Omega^+(0) = \bB^3_{3.1}(0) \setminus \overline{\bB^3_{2.1}(0)}$. We notice
that our algorithm \eqref{eq:MSHG} approximates the known true solution obtained from solving the ODE system \eqref{radial:shell:ode2} very
well.
\begin{figure}
\center
\includegraphics[angle=-90,width=0.45\textwidth]{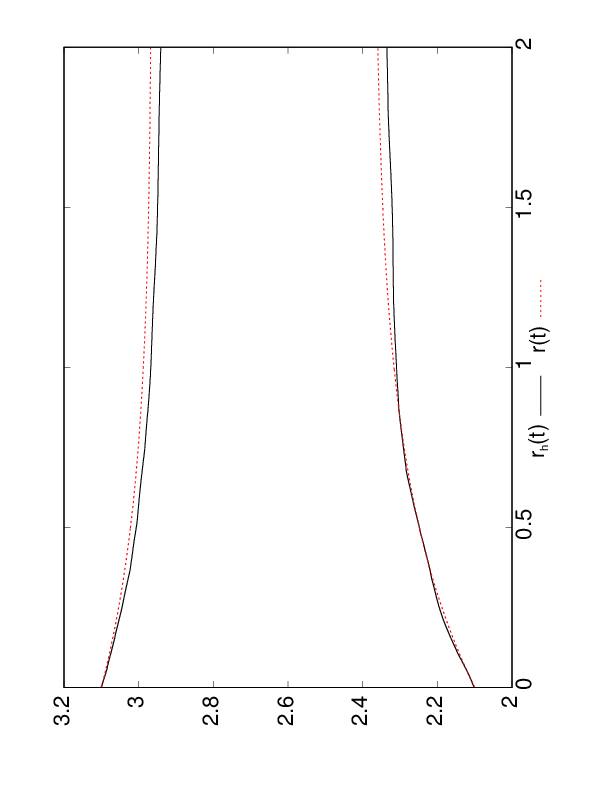}
\caption{($\Omega = \bB^3_5(0)$) $\alpha=0.07$, $m_\pm=1$, $\rho_\pm=2$,
$S_-=1$, $S_+=-4$, $S_I=0.42$ with $r(0)=(2.1, 3.1)$.
We plot $r(t)$ and $r_h(t)$ over time.
}
\label{fig:3dthin0}
\end{figure}%
The same experiment but within the cubic domain $\Omega = (-5,5)^3$ is shown in
Figure~\ref{fig:3dthin0sq}.
\begin{figure}
\center
\includegraphics[angle=-0,width=0.45\textwidth]{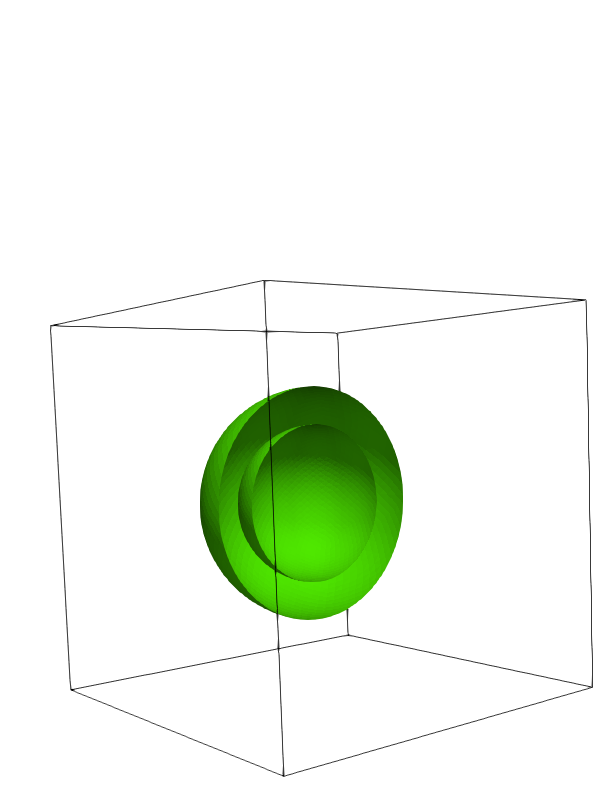}
\includegraphics[angle=-0,width=0.45\textwidth]{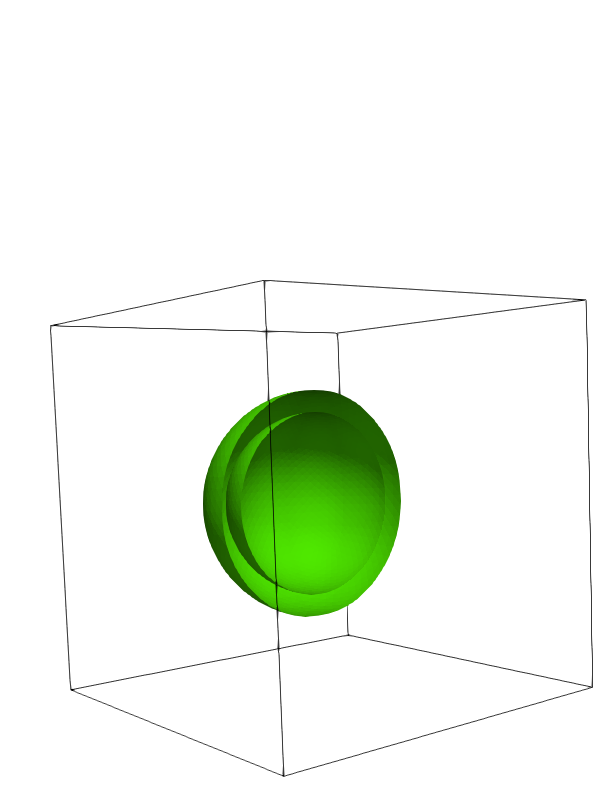}
\caption{($\Omega = (-5,5)^3$) 
$\alpha=0.07$, $m_\pm=1$, $\rho_\pm=2$,
$S_-=1$, $S_+=-4$, $S_I=0.42$ with $r(0)=(2.1, 3.1)$.
We show the solution at times $t=0$ and $t=10$.
}
\label{fig:3dthin0sq}
\end{figure}%

\subsection{Numerical simulations with pinch-off}

We start from a cigar-like initial shape of total dimension
$0.8\times 0.76 \times 0.76$, i.e., nearly a sphere that is slightly stretched
in the $x$-direction, centered at the origin within the spherical domain
$\Omega = \bB_4^3(0)$.
The physical parameters are chosen as $\alpha=0.05$,
$m_\pm=0.5$, $\rho_-=0.4$, $\rho_+=1$, $S_\mp=\pm1$, $S_I=-3.7$.
In Figure~\ref{fig:3dtopo1round},
we show the evolution until just before the first pinch-off.
\begin{figure}
\center
\mbox{
\hspace*{-1cm}
\includegraphics[angle=-0,width=0.2\textwidth]{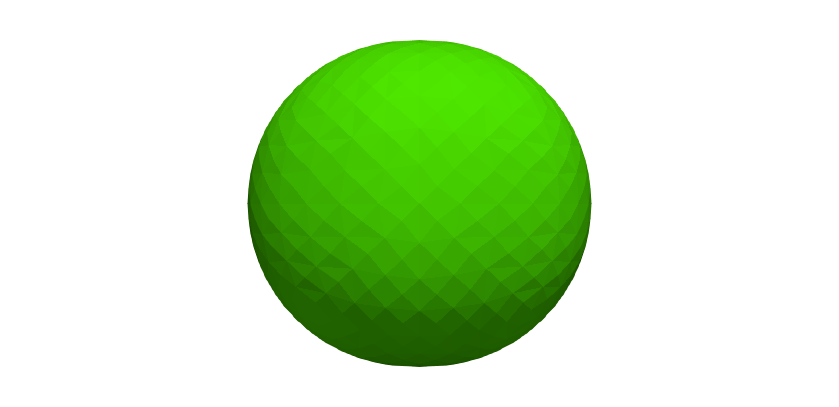}
\includegraphics[angle=-0,width=0.2\textwidth]{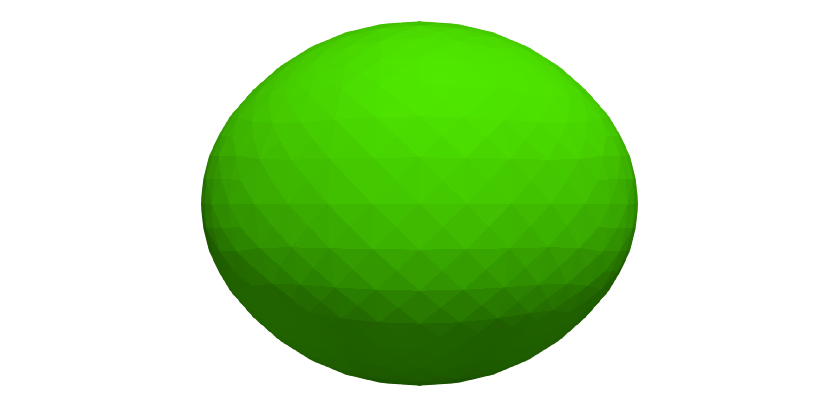}
\includegraphics[angle=-0,width=0.2\textwidth]{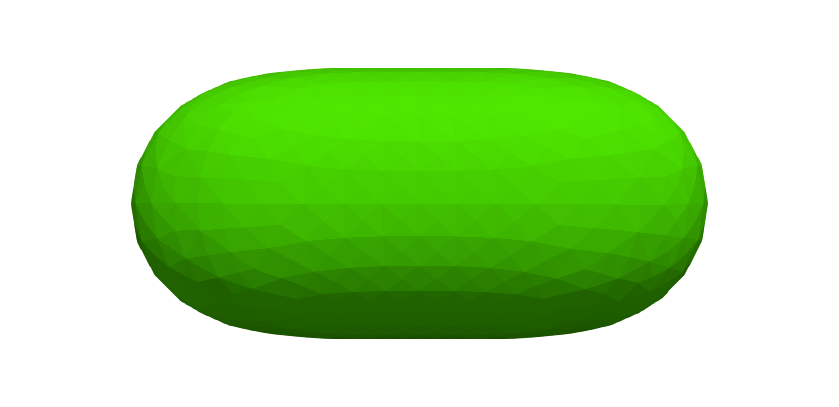}
\includegraphics[angle=-0,width=0.2\textwidth]{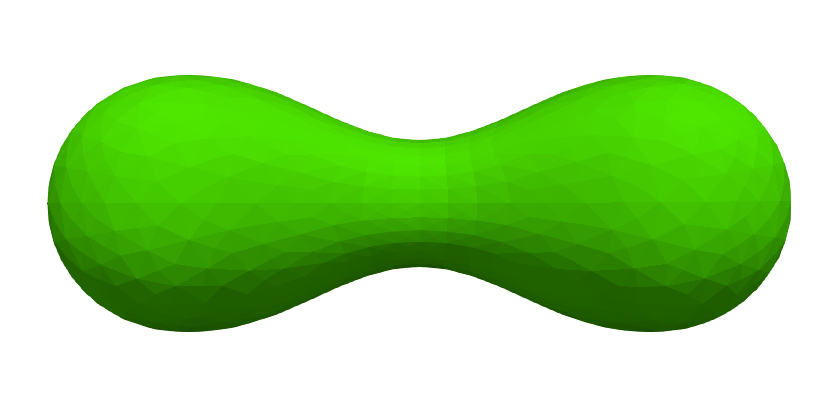}
\includegraphics[angle=-0,width=0.2\textwidth]{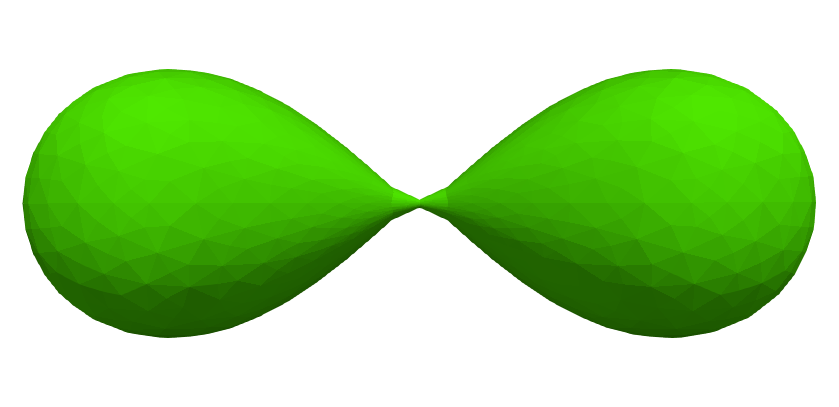}
}
\caption{($\Omega = \bB_4^3(0)$) Pinch-off for $\alpha=0.05$,
$m_\pm=0.5$, $\rho_-=0.4$, $\rho_+=1$, $S_\mp=\pm1$, $S_I=-3.7$.
We show the solution at times $t=0, 1, 2, 2.4, 2.5$.
}
\label{fig:3dtopo1round}
\end{figure}%
In order to break the symmetry of the evolution with respect to the origin,
we start another experiment with the same initial blob, but now centered 
at $(-2,-2,-2)^\top$ within the cube $\Omega = (0,8)^3$.
In the evolution shown in Figures~\ref{fig:3dtopo2offset} and 
\ref{fig:3dtopo2offset2} we see
that a total of 23 pinch-offs occur, leading to 24 blobs co-existing in the
domain at the final time. The fact that the last pinch-off occurs at time 
$t=51.23$ indicates that afterwards there is no more space available for 
blobs to grow into.
\begin{figure}
\center
\includegraphics[angle=-0,width=0.24\textwidth]{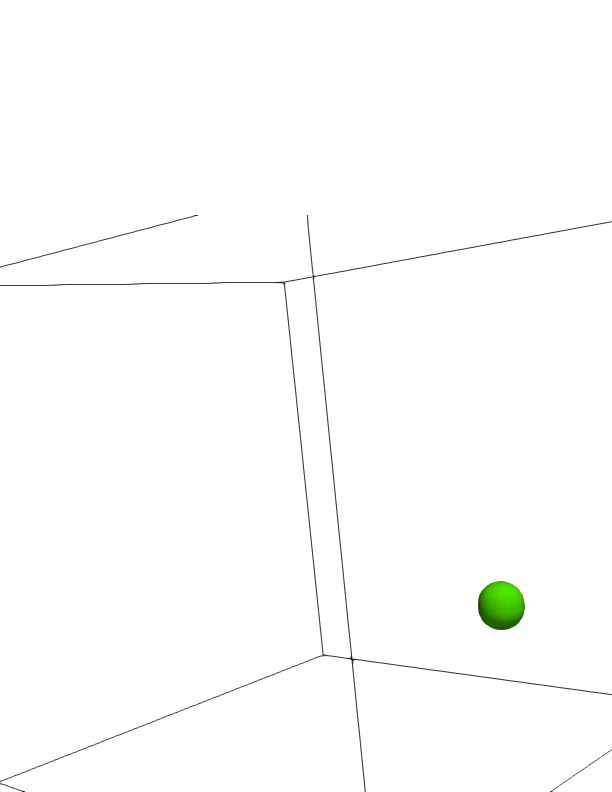}
\includegraphics[angle=-0,width=0.24\textwidth]{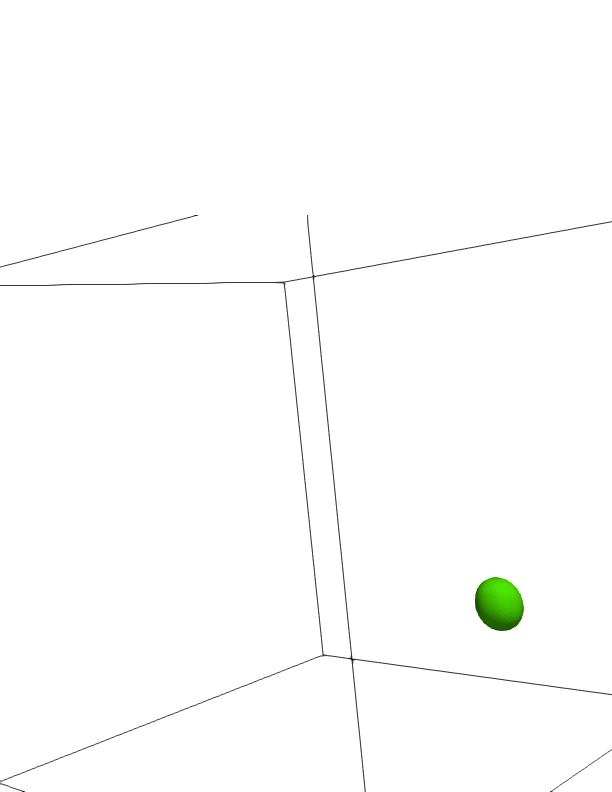}
\includegraphics[angle=-0,width=0.24\textwidth]{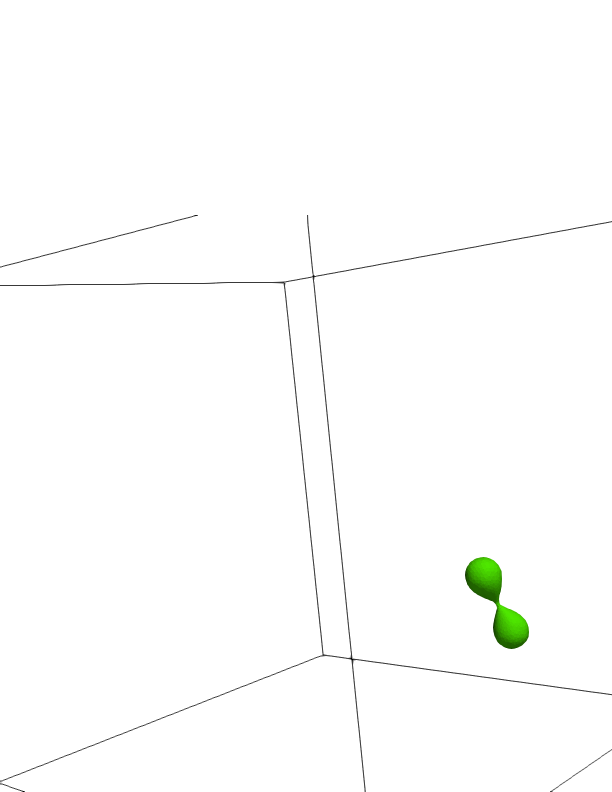}
\includegraphics[angle=-0,width=0.24\textwidth]{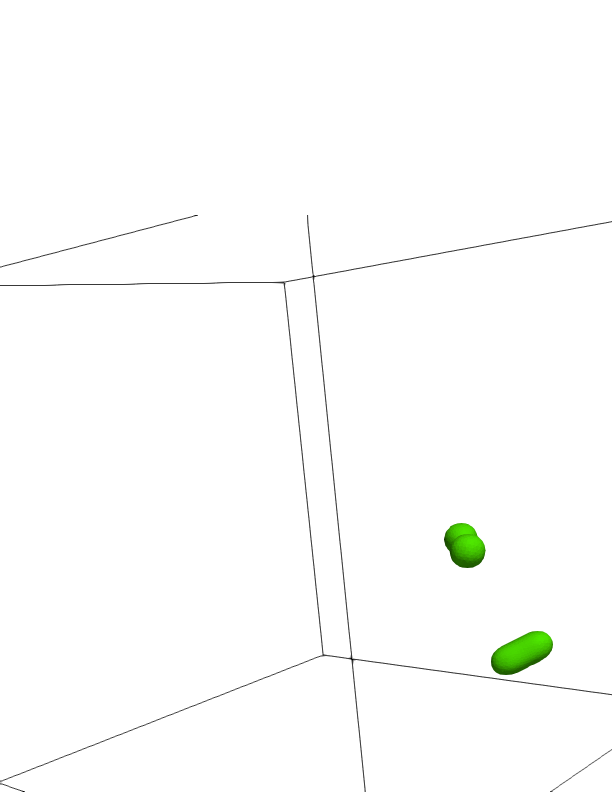}
\includegraphics[angle=-0,width=0.24\textwidth]{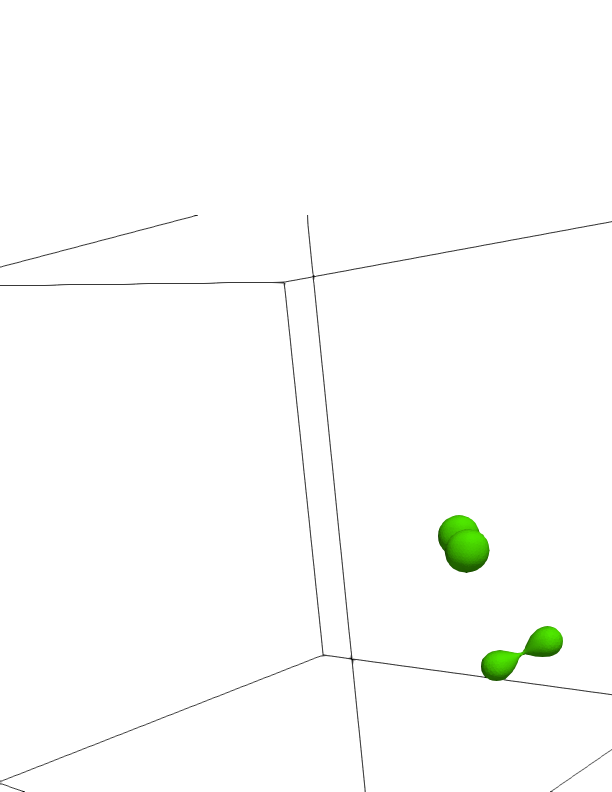}
\includegraphics[angle=-0,width=0.24\textwidth]{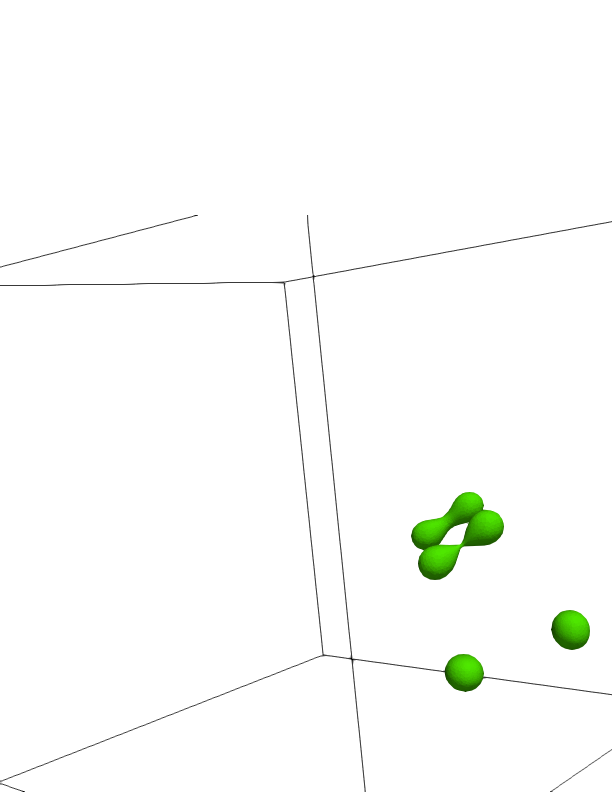}
\includegraphics[angle=-0,width=0.24\textwidth]{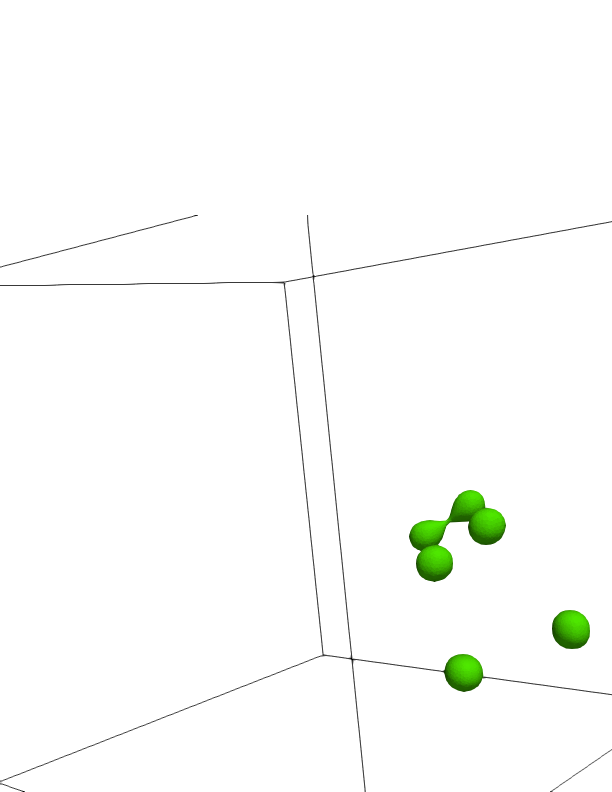}
\includegraphics[angle=-0,width=0.24\textwidth]{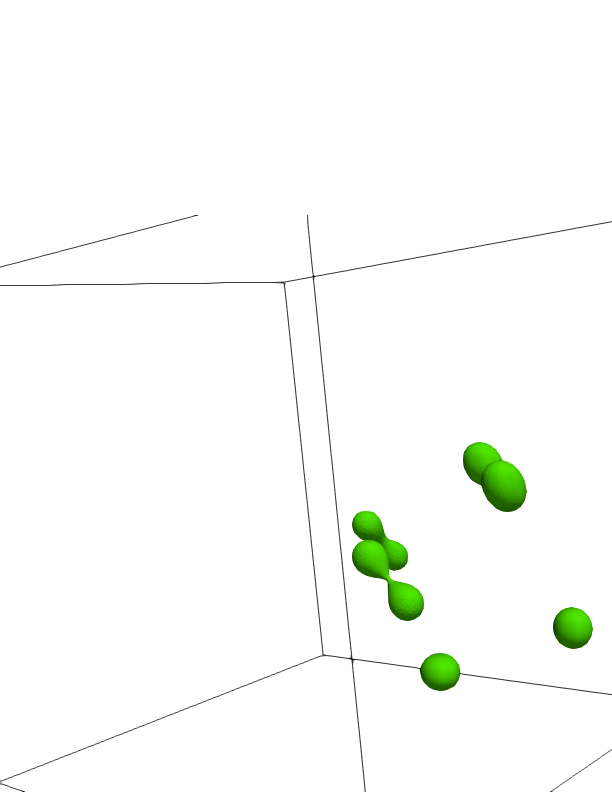}
\includegraphics[angle=-0,width=0.24\textwidth]{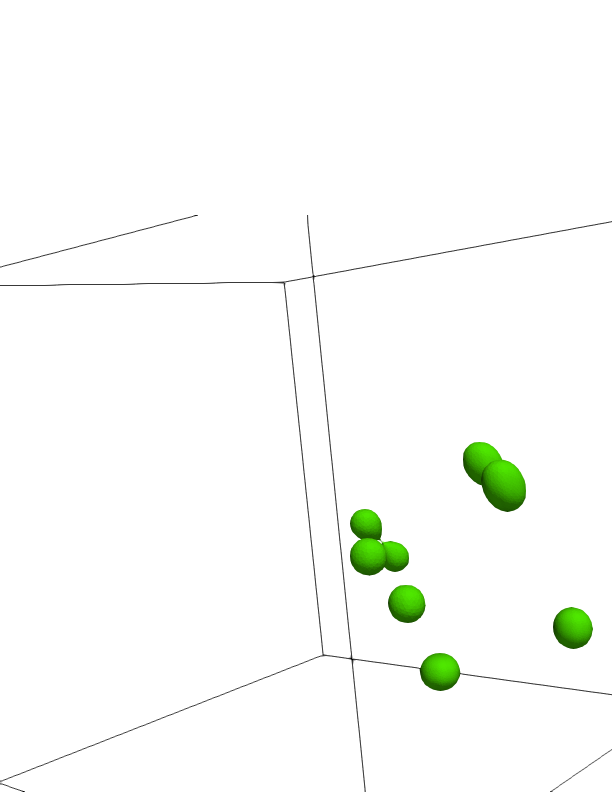}
\includegraphics[angle=-0,width=0.24\textwidth]{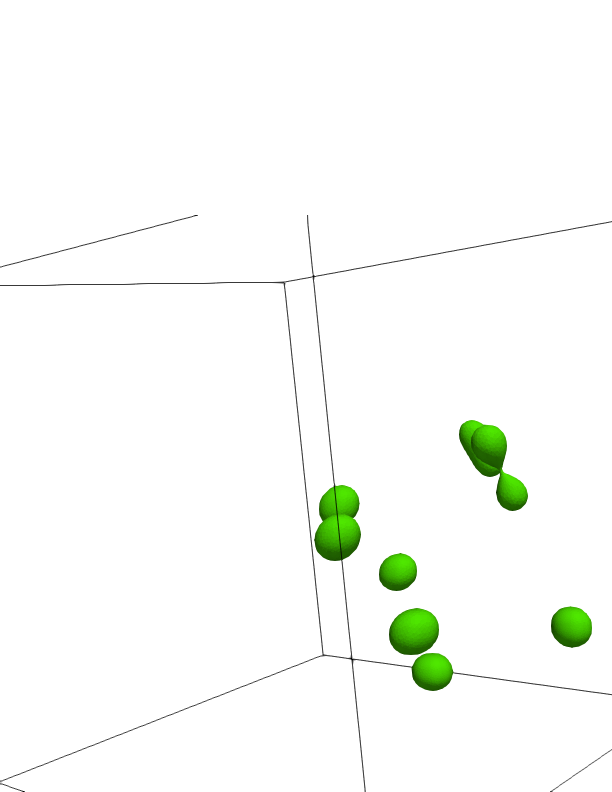}
\includegraphics[angle=-0,width=0.24\textwidth]{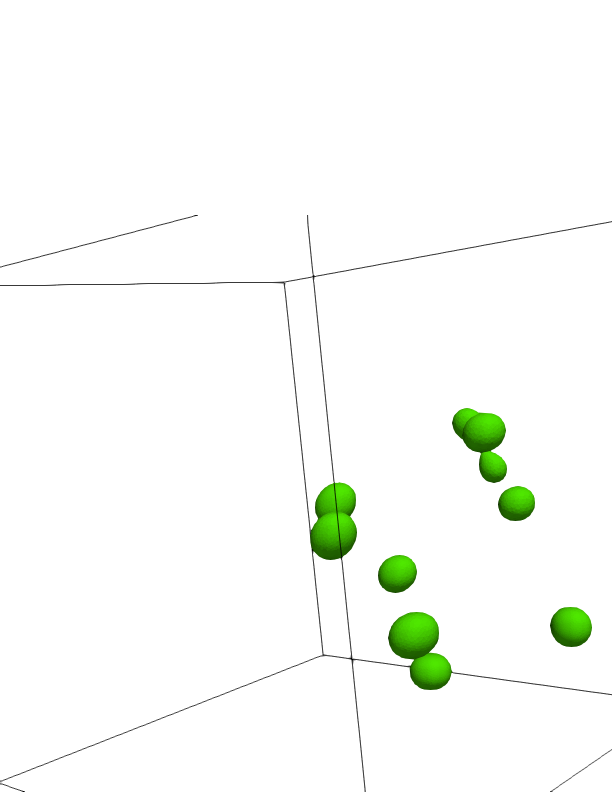}
\includegraphics[angle=-0,width=0.24\textwidth]{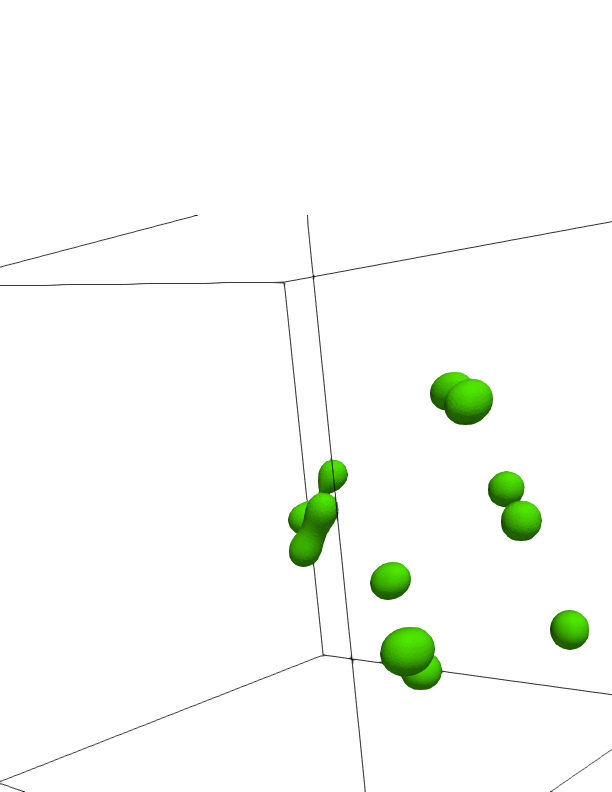}
\includegraphics[angle=-0,width=0.24\textwidth]{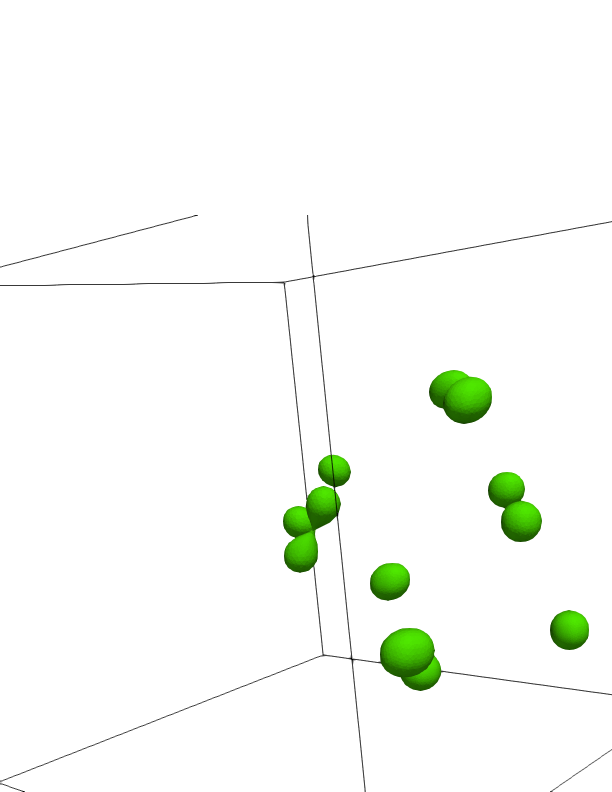}
\includegraphics[angle=-0,width=0.24\textwidth]{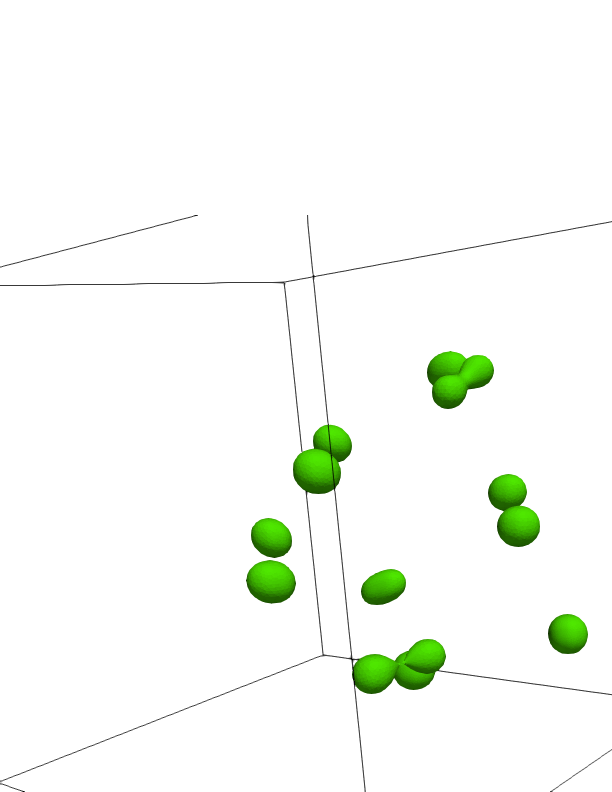}
\includegraphics[angle=-0,width=0.24\textwidth]{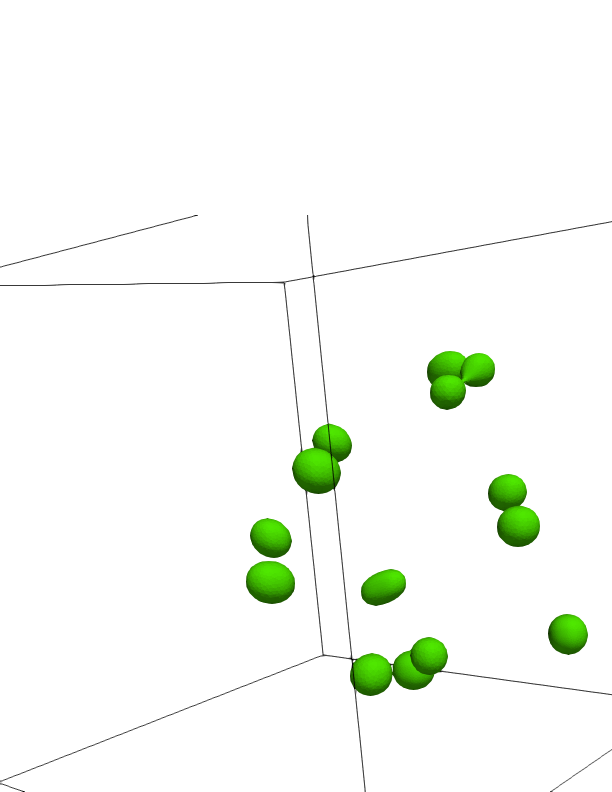}
\includegraphics[angle=-0,width=0.24\textwidth]{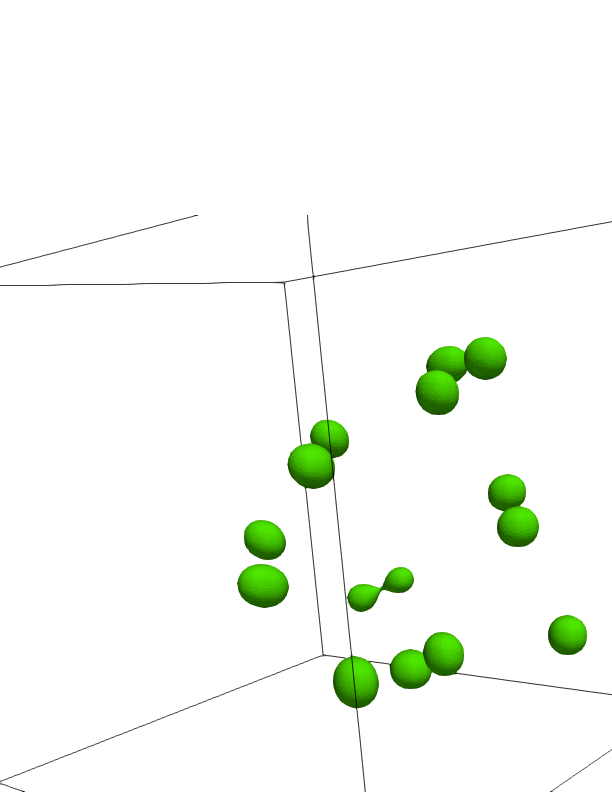}
\caption{($\Omega = (0,8)^3$) Pinch-offs for 
$\alpha=0.05$, $m_\pm=0.5$, $\rho_-=0.4$, $\rho_+=1$, $S_\mp=\pm1$, $S_I=-3.7$.
We show the evolution at times $t= 0$, 2, 3.7, 7.95, 8.45, 11.27, 11.35, 16.25,
16.33, 18.92, 19.45, 22.47, 22.72, 25.83, 25.9, 26.94. All but the first two
snap shots show a time just before a pinch-off occurs.
}
\label{fig:3dtopo2offset}
\end{figure}%
\begin{figure}
\center
\includegraphics[angle=-0,width=0.24\textwidth]{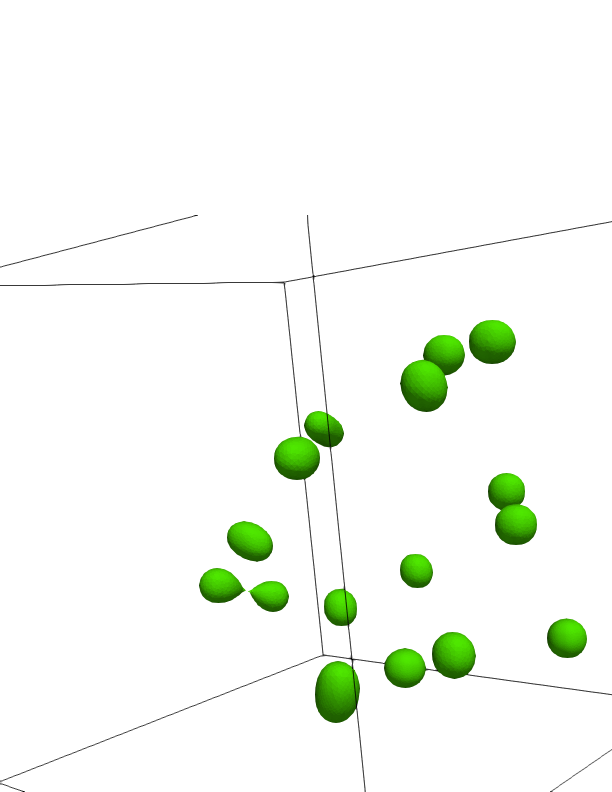}
\includegraphics[angle=-0,width=0.24\textwidth]{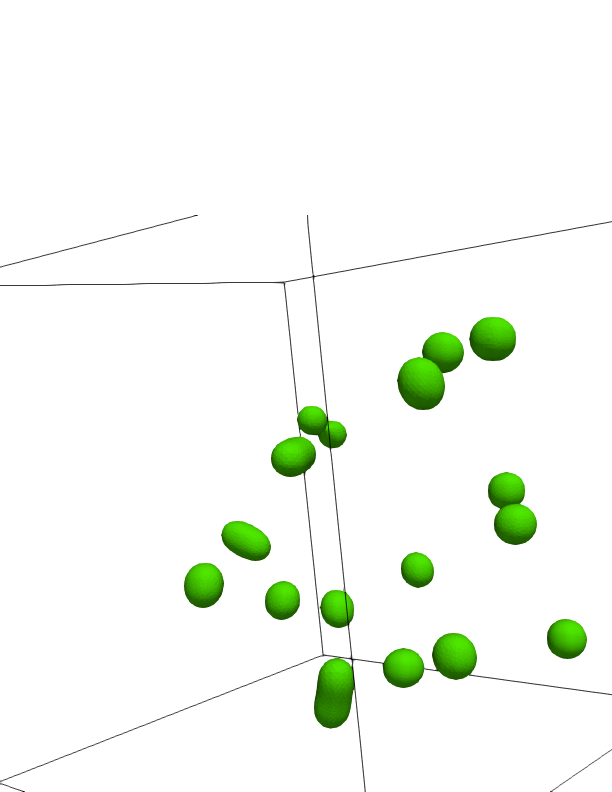}
\includegraphics[angle=-0,width=0.24\textwidth]{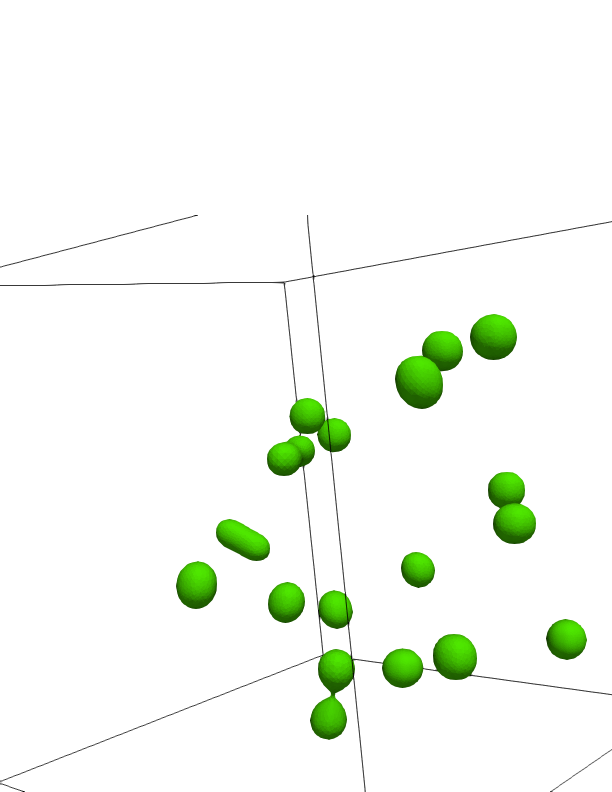}
\includegraphics[angle=-0,width=0.24\textwidth]{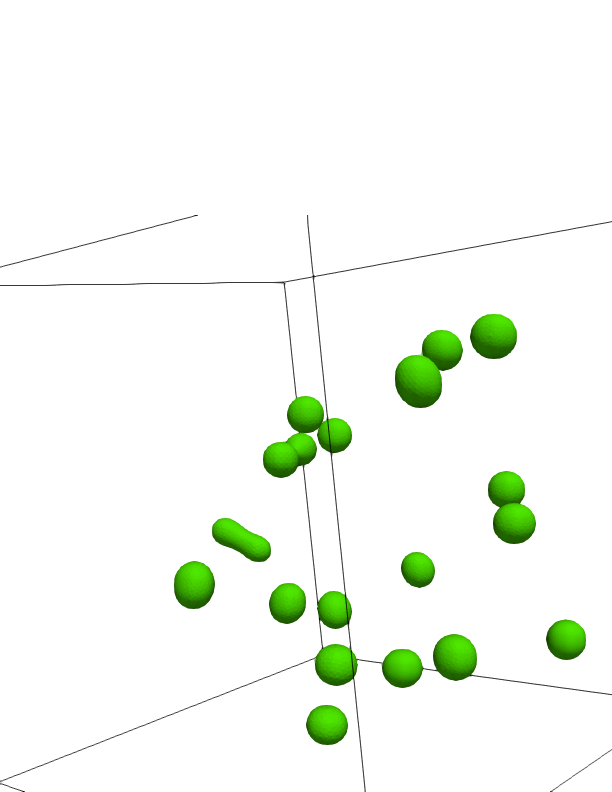}
\includegraphics[angle=-0,width=0.24\textwidth]{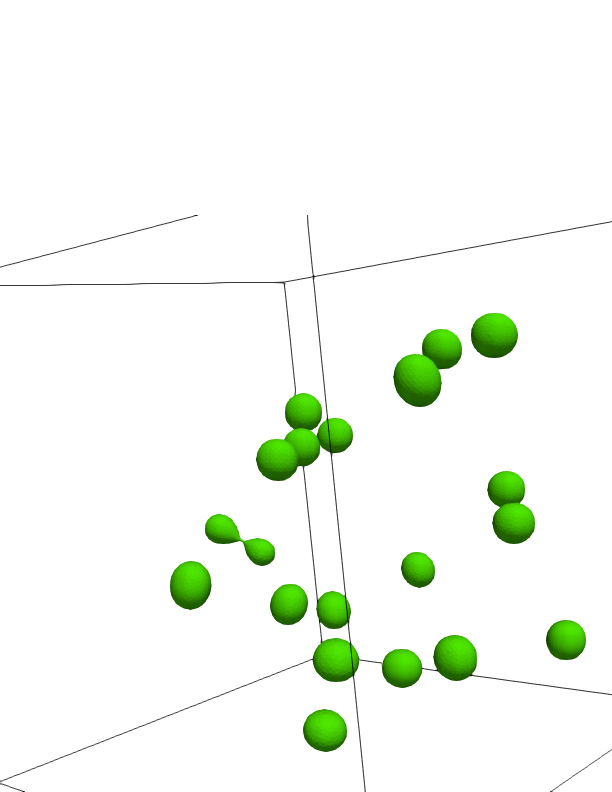}
\includegraphics[angle=-0,width=0.24\textwidth]{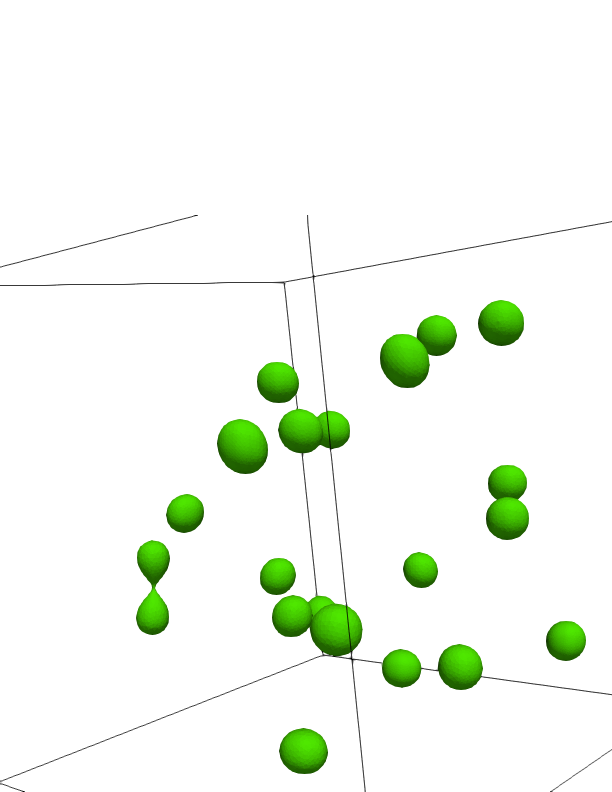}
\includegraphics[angle=-0,width=0.24\textwidth]{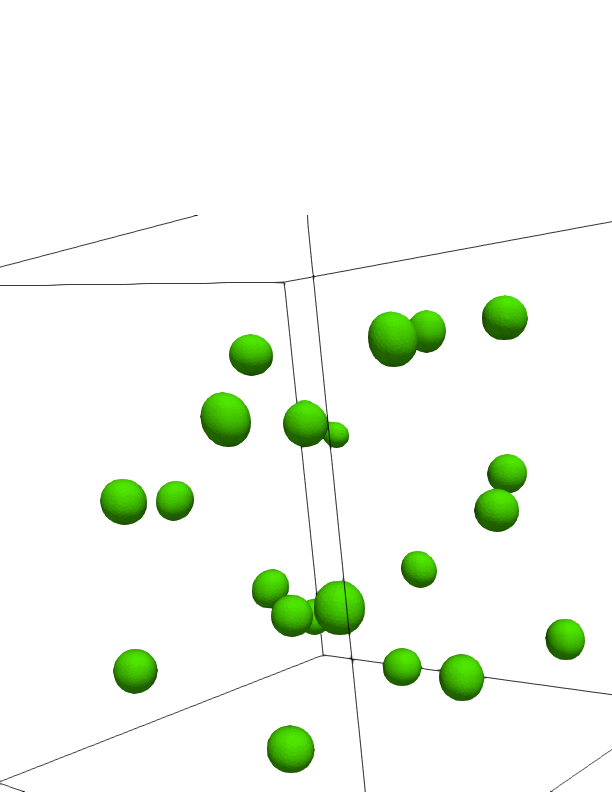}
\includegraphics[angle=-0,width=0.24\textwidth]{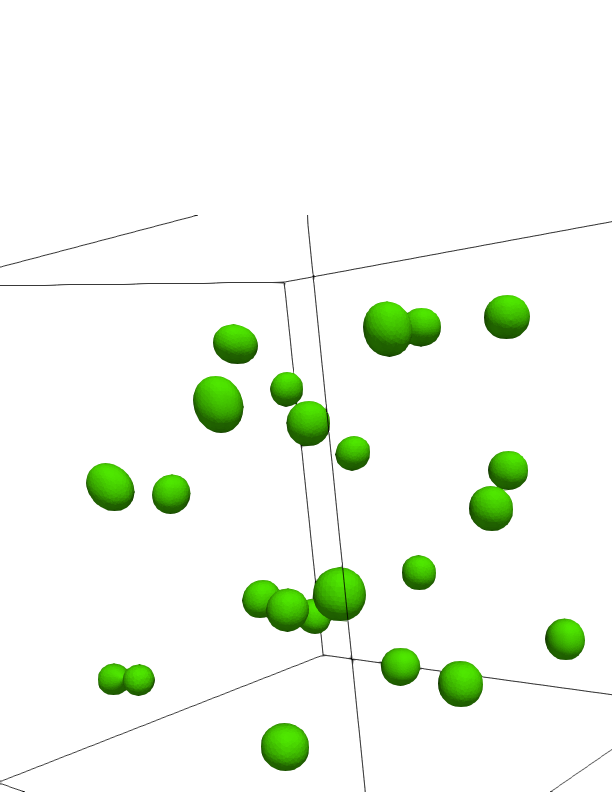}
\includegraphics[angle=-0,width=0.24\textwidth]{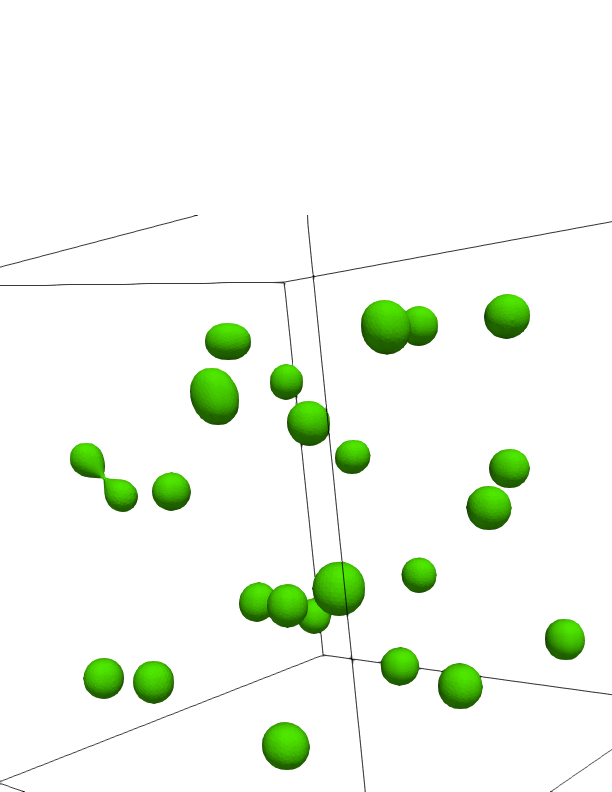}
\includegraphics[angle=-0,width=0.24\textwidth]{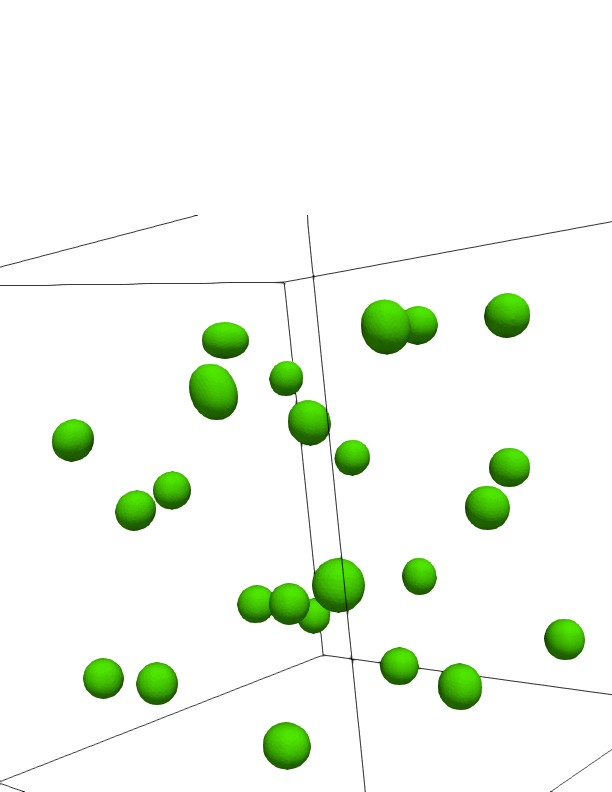}
\includegraphics[angle=-0,width=0.24\textwidth]{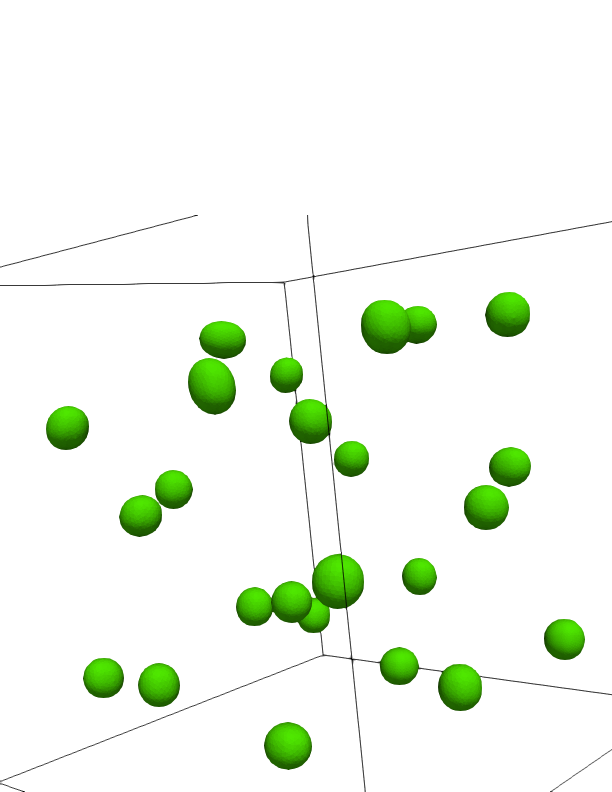}
\includegraphics[angle=-0,width=0.24\textwidth]{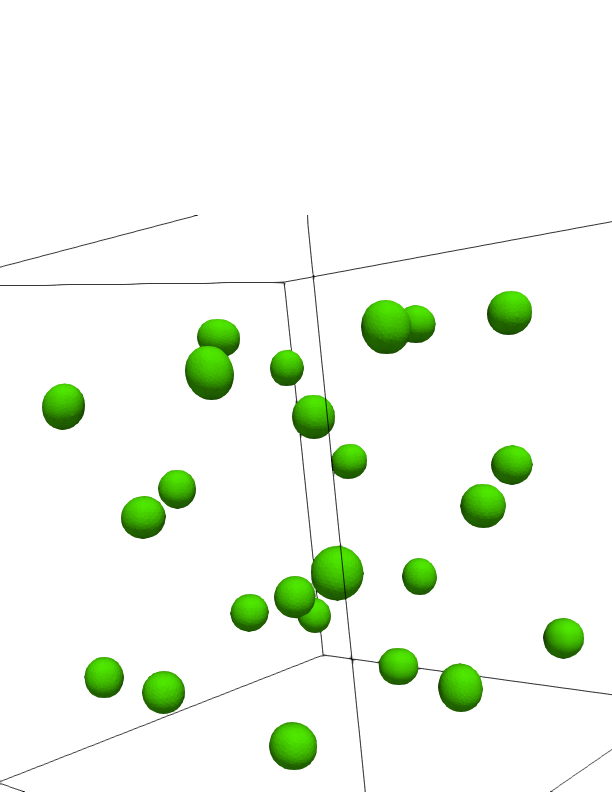}
\caption{Continuation of Figure~\ref{fig:3dtopo2offset}, 
at times $t= 29.57$, 30.26, 30.75,
30.97, 31.27, 36.75, 41.81, 48, 51.23, 53, 55, 60.
All but the last three snap shots show a time just before a pinch-off occurs.
}
\label{fig:3dtopo2offset2}
\end{figure}%

In our final two experiments, we take the same initial data as in
Figure~\ref{fig:3dtopo1round}, but now choose different physical parameters
in order to induce unstable growth. In particular, we attempt simulations that
either favour growth for modes with $\ell=3$ or with $\ell = 4$.
For the former case, we let
$\alpha = 0.05$, $m_\pm = 0.5$, $\rho_- =0.29$, $\rho_+ = 1$,
$S_- = 1$, $S_+ = -1$, $S_I = -3.7$ and observe the evolution shown in 
Figure~\ref{fig:3dell3round}.
\begin{figure}
\center
\mbox{
\hspace*{-1cm}
\includegraphics[angle=-0,width=0.2\textwidth]{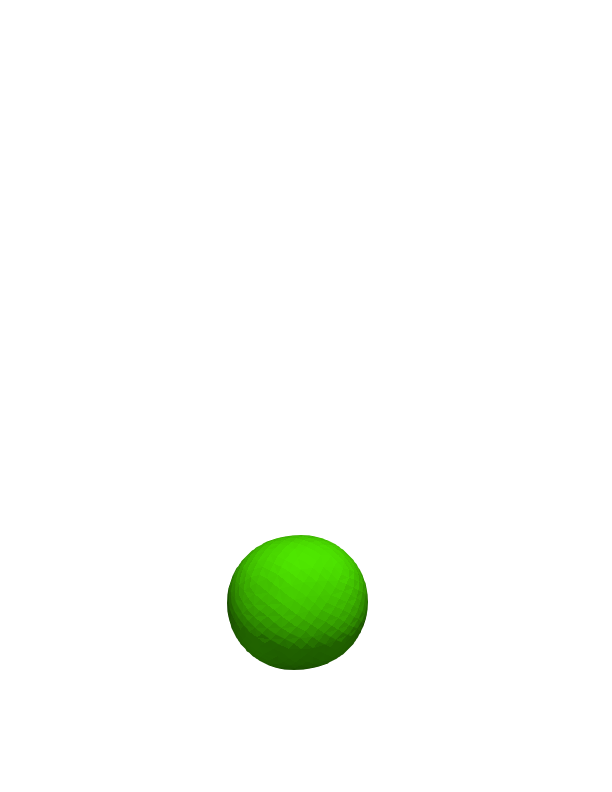}
\includegraphics[angle=-0,width=0.2\textwidth]{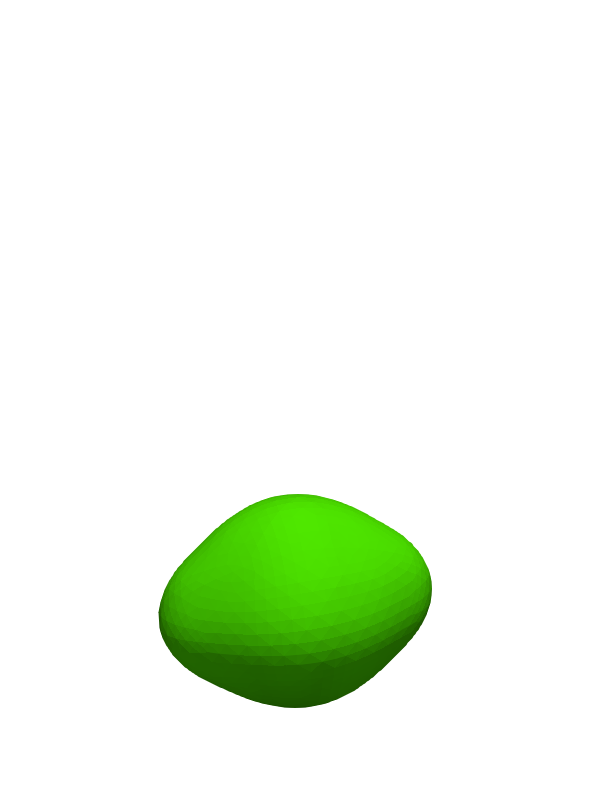}
\includegraphics[angle=-0,width=0.2\textwidth]{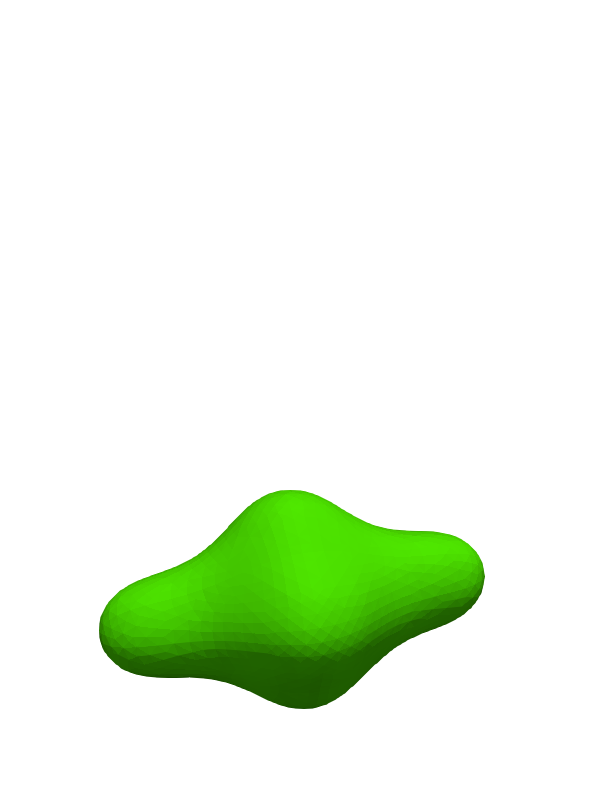}
\includegraphics[angle=-0,width=0.2\textwidth]{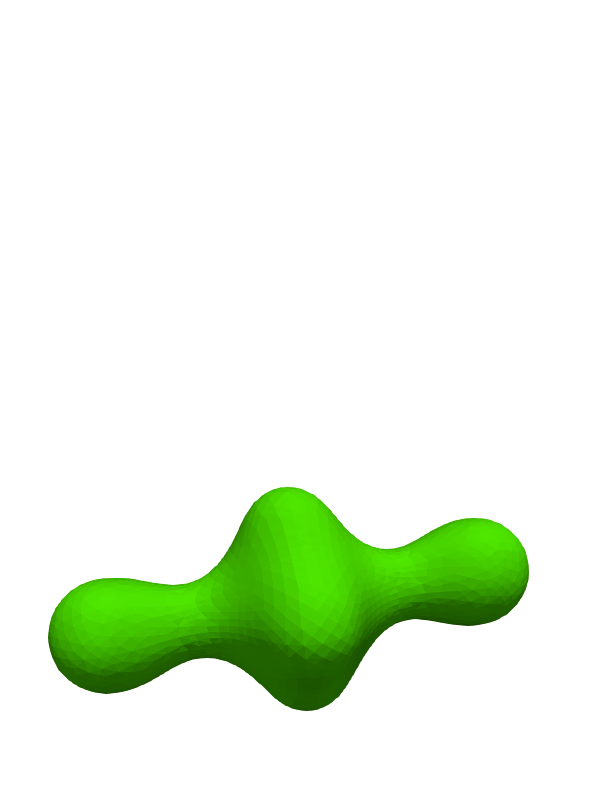}
\includegraphics[angle=-0,width=0.2\textwidth]{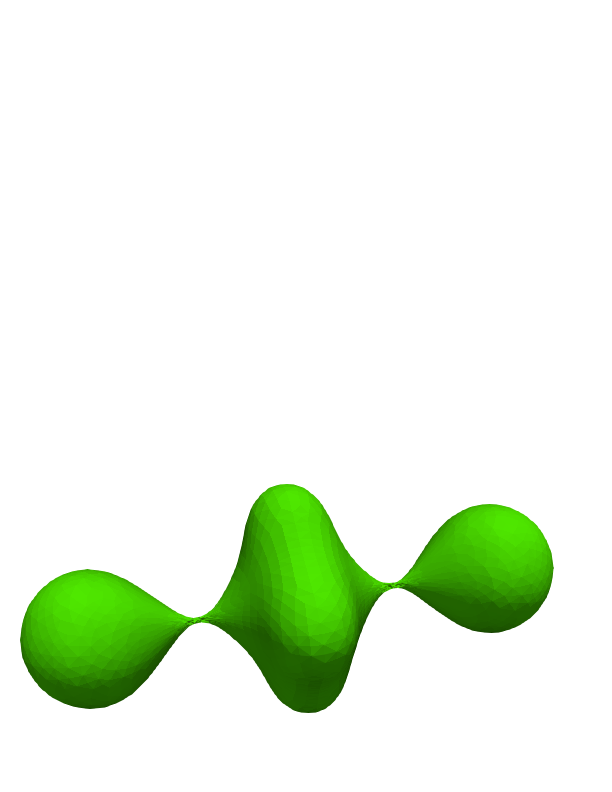}
}
\caption{($\Omega = \bB_4^3(0)$)
$\alpha = 0.08$, $m_\pm = 1$, $\rho_- =0.16$, $\rho_+ = 1$,
$S_- = 1$, $S_+ = -4$, $S_I = -5$.
We show the solution at times $t=0, 0.8, 1.2, 1.4, 1.53$.
}
\label{fig:3dell3round}
\end{figure}%
For the case $\ell=4$, on the other hand, we let
$\alpha = 0.08$, $m_\pm = 1$, $\rho_- =0.16$, $\rho_+ = 1$,
$S_- = 1$, $S_+ = -4$, $S_I = -5$ 
we observe the evolution shown in Figure~\ref{fig:3dell4round}.
In each case we show the evolution until just before the first simultaneous
pinch-offs occur. In Tables~\ref{tbl:3Dsingleradial_eigenl3} and \ref{tbl:3Dsingleradial_eigenl4}, we report on the values of the right-hand side of \eqref{radial:ode:perturb} for perturbation modes $\ell = 1, \dots, 10$ for the parameter settings of Figures~\ref{fig:3dell3round} and \ref{fig:3dell4round}, where we see that mode $\ell = 3$ (resp.~$\ell = 4$) has the largest amplification of the perturbations for the settings of Figure~\ref{fig:3dell3round} (resp.~Figure~\ref{fig:3dell4round}).
\begin{table}[h]
\centering
\begin{tabular}{|c|c|c|c|c|c|}
\hline
$\ell$ & $1$ & $2$ & $3$ & $4$ & 5 \\
\hline
& 0.2050 &   4.3153 &   \textbf{5.4376} &    1.9752 &   $-7.4842$  \\
\hline\hline
$\ell$ & 6 & 7 & 8 & 9 & 10 \\
\hline
& $-24.2574$ & $-49.7308$ & $-85.2324$ &$-132.1021$ & $-191.6791$ \\
\hline
\end{tabular}
\caption{Values of the right-hand side of \eqref{radial:ode:perturb} for perturbation modes $\ell = 1, \dots, 10$ for the setting of Figure~\ref{fig:3dell3round}. The largest value highlighted in bold corresponds to mode $\ell = 3$.}
\label{tbl:3Dsingleradial_eigenl3}
\end{table}
\begin{table}[h]
\centering
\begin{tabular}{|c|c|c|c|c|c|}
\hline
$\ell$ & $1$ & $2$ & $3$ & $4$ & 5 \\
\hline
 & $-0.0374$ &   5.8266  & 10.0378  & \textbf{11.9071} &  10.7068  \\
\hline\hline
$\ell$ & 6 & 7 & 8 & 9 & 10 \\
\hline
& 5.7831  & $-3.6066$ & $-18.1369$ & $-38.4965$ & $-65.3730$ \\
\hline
\end{tabular}
\caption{Values of the right-hand side of \eqref{radial:ode:perturb} for perturbation modes $\ell = 1, \dots, 10$ for the setting of Figure~\ref{fig:3dell4round}. The largest value highlighted in bold corresponds to mode $\ell = 4$.}
\label{tbl:3Dsingleradial_eigenl4}
\end{table}
\begin{figure}
\center
\mbox{
\hspace*{-1cm}
\includegraphics[angle=-0,width=0.2\textwidth]{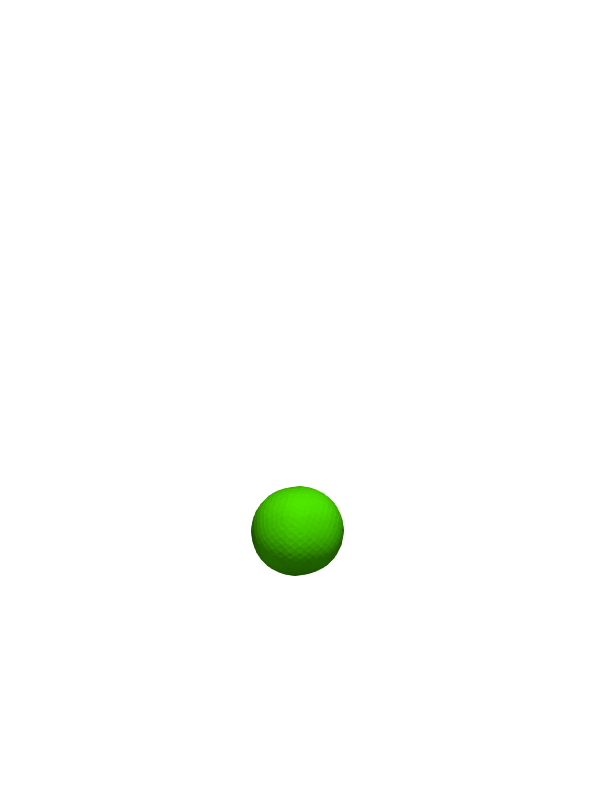}
\includegraphics[angle=-0,width=0.2\textwidth]{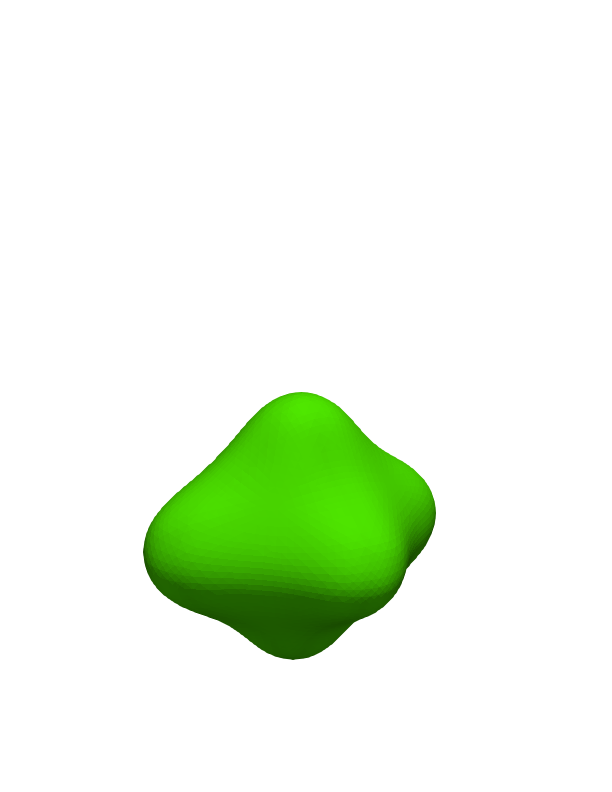}
\includegraphics[angle=-0,width=0.2\textwidth]{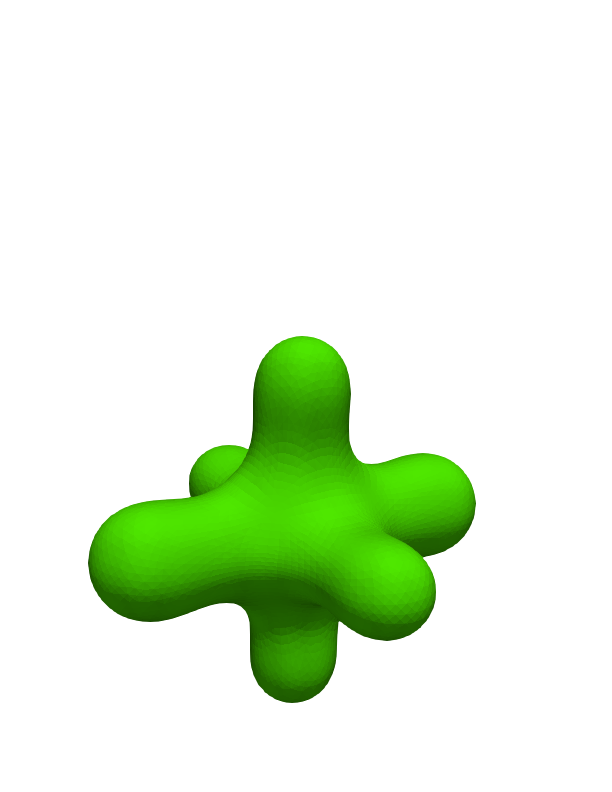}
\includegraphics[angle=-0,width=0.2\textwidth]{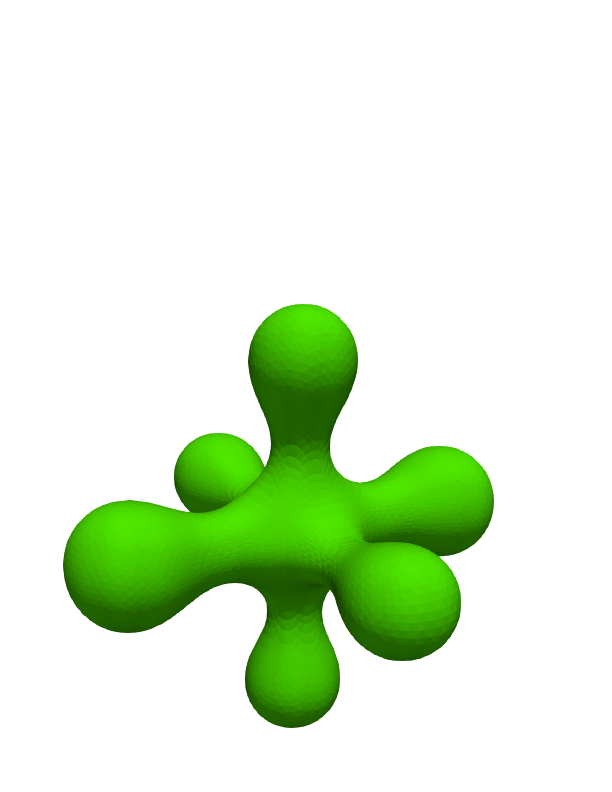}
\includegraphics[angle=-0,width=0.2\textwidth]{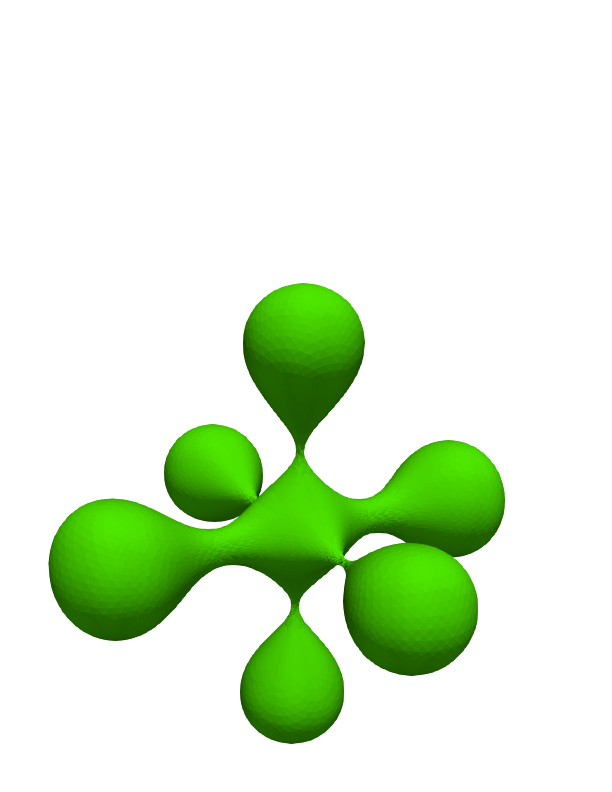}
}
\caption{($\Omega = \bB_4^3(0)$)
$\alpha = 0.08$, $m_\pm = 1$, $\rho_- =0.16$, $\rho_+ = 1$,
$S_- = 1$, $S_+ = -4$, $S_I = -5$.
We show the solution at times $t=0, 0.5, 0.7, 0.8, 0.88$.
}
\label{fig:3dell4round}
\end{figure}%


\section*{\bf Acknowledgements}
\noindent
AS gratefully acknowledge some support
from the GNAMPA (Gruppo Nazionale per l'Analisi Matematica, la Probabilit\`a e le loro Applicazioni) of INdAM (Isti\-tuto Nazionale di Alta Matematica)
project CUP E5324001950001,
from ``MUR GRANT Dipartimento di Eccellenza'' 2023-2027, and from the Alexander von Humboldt Foundation.
Additionally, AS appreciates affiliation with GNAMPA (Gruppo Nazionale per l'Analisi Matematica, la Probabilit\`a  e le loro Applicazioni) of INdAM (Istituto Nazionale di Alta Matematica).
{HG, KFL and AS benefited from stimulating discussions related to this work during their stay at the Erwin Schr\"odinger International Institute for Mathematics and Physics (ESI) in Vienna, and gratefully acknowledge the support and hospitality of the Institute.}

\appendix
\section{Radial symmetric solutions}
\subsection{Analytical formula for solutions}\label{app:radial:soln}
Suppressing the explicit dependence on $t$, we consider the ansatz where solutions to \eqref{linstab}--\eqref{linstab:6} are of the form
\begin{alignat}{3}
	\label{muplus:rad}
	\mu_+(r) & = b_{d,+}(q) \Iz \Lpr + e_+, \quad && r \in (0,q),
	\\
	\mu_-(r) & = b_{d,-}(q) \Iz \Lmr + c_{d,-}(q)  \Kz \Lmr + e_-, \quad && r \in (q,R),
	\label{muminus:rad}
\end{alignat}
for suitable constants $e_{\pm}$ and functions $b_{d,\pm}(q)$ and $c_{d,-}(q)$ of the interfacial location $q(t) \in (0,R)$ to be determined. We use the notation $\mu_+$ to denote $\mu$ in ${\Omega^+(t)}$ and $\mu_-$ to denote $\mu$ in ${\Omega^-(t)}$ for better presentation, and point out that the expressions for the coefficient functions $b_{d,\pm}$ and $c_{d,-}$ depend on the spatial dimension $d$. From the condition \eqref{sharp:in:cond} we note that the ansatz for $\mu_+$ in \eqref{muplus:rad} does not contain a term like $c_{d,+}(q)\Kz(\Lp r)$ due to its singular behavior at $r = 0$. \\

\noindent Substituting these into \eqref{lin:stab:1} and \eqref{lin:stab:1:bis} we infer that
\begin{align*}
- \frac{S_{+}}{m_+} + \Lp^2 \mu_+ & = \mu_{+}''(r) + \frac{d-1}{r} \mu_{+}'(r) = b_{d,+}(q) \Lp^2 \Iz(\Lp r) = \Lp^2(\mu_+(r) - e_+), \\
- \frac{S_{-}}{m_-} + \Lm^2 \mu_- & = \mu_{-}''(r) + \frac{d-1}{r} \mu_{-}'(r) = \Lm^2(\mu_-(r) - e_-),
\end{align*}
which leads to the identification of the constants $e_{\pm}$ as
\[
e_{\pm} = \frac{S_{\pm}}{\lambda_{\pm}^2 m_{\pm}} = \frac{S_{\pm}}{\rho_{\pm}}.
\]
Next, the external boundary condition \eqref{lin:stab:5} yields that
\begin{equation}\label{a11a12}
\begin{aligned}
0 & = \mu_-'(r = R) = \Lm \big ( b_{d,-}(q) \Iz' \LmR + c_{d,-}(q) \Kz' \LmR) \\
& = \lambda_- (b_{d,-}(q) \III_1(\lambda_-R) - c_{d,-}(q) \KKK_1(\lambda_-R)).
\end{aligned}
\end{equation}
From the interface condition \eqref{lin:stab:2}{,} we infer that
\begin{align}
	\label{sharp:proof:1:combined}
	b_{d,+} (q) \Iz \Lpq + e_+ &  = \alpha \frac {d-1}q,
	\\
	\label{sharp:proof:2:combined}
	 b_{d,-} (q) \Iz \Lmq 
	+ c_{d,-}(q)  \Kz \Lmq 	
	+ e_-  &=  \alpha \frac {d-1}q.
\end{align}
We can express \eqref{a11a12} and \eqref{sharp:proof:2:combined} as a linear system:
\begin{align*}
	\begin{pmatrix}
	\III_1(\lambda_-R) & -\KKK_1(\lambda_-R) \\
	\Iz(\lambda_-q) & \Kz(\lambda_-q)
	\end{pmatrix}
	\begin{pmatrix}
	b_{d,-} (q) \\
	c_{d,-} (q)
	\end{pmatrix}
	=
	\begin{pmatrix}
	0  \\
	\alpha \frac {d-1}q- e_-
	\end{pmatrix},
\end{align*}
while noting that \eqref{sharp:proof:1:combined} provides a formula for $b_{d,+}(q)$, leading to 
\begin{subequations} \label{bc:12:combined}
\begin{align}
	\label{bc:1:combined}
	b_{d,+}(q)  & = \Big( \alpha \frac {d-1}q - e_+\Big)  \frac 1 {\Iz \Lpq},
	\quad
	b_{d,-}(q)  =
	\frac {\KKK_1 \LmR}{\III_1 \LmR}c_{d,-}(q),
	\\ \label{bc:2:combined}
		c_{d,-}(q) & = \Big(\alpha \frac {d-1}q- e_-\Big)
	\frac {1}{\Iz \Lmq \frac {\KKK_1 \LmR}{\III_1 \LmR} + \Kz \Lmq}.
\end{align}
\end{subequations}
Therefore,
recalling that $e_\pm = \frac {S_\pm}{\rho_\pm}$, the identities {\eqref{bc:12:combined}}  yield the expressions in \eqref{radial:mu:sol}, and
we can express the ordinary differential equation \eqref{lin:stab:4} as
\begin{align}
\non        2 \dot{q}
	& =
	-\jump{{m_\pm} (\mu')(q)} +  S_I
	= -{m_+}\mu'_+ (q)
	+{m_-} \mu'_- (q)
	+ S_I
	\\ 
 \non       &
	=-m_+ \Lp  b_{d,+} (q)	\III_1 \Lpq
	+m_- \Lm  \big(
	b_{d,-}	\III_1 \Lmq
	- c_{d,-}	\KKK_1 \Lmq \big)
	+ S_I
	\\ 
 \non       & = -m_+ \Lp  \Big(
	\alpha \frac {d-1}q  -\frac{S_+}{\rho_+}
	\Big) \frac{\III_1 \Lpq}{\Iz \Lpq}
	\\ 
 & \quad
	+m_- \Lm  \Big(
	\alpha \frac {d-1}q  - \frac{S_-}{\rho_-}
	\Big) \frac{\III_1\Lmq \frac {\KKK_1\LmR}{\III_1\LmR} -\KKK_1\Lmq }{\Iz\Lmq  \frac {\KKK_1\LmR}{\III_1\LmR}+ \Kz \Lmq}
	+ S_I,
	\label{ode:combined}
\end{align}
which coincides with \eqref{radial:ode}.

\subsection{Solutions to the perturbed system on finite domains}\label{app:radial:perturb}
Here, we denote the stationary solution of the system by 
$\qstar$, as anticipated in {Section \ref{SEC:RAD}}.
Similar to $\mu_{\pm}$, we look for general solutions of the form
\begin{align*}
	U_+ (r) & =
	B_{d,+}(\qstar) \III_\ell(\Lp  r) ,\\
	U_- (r)  & =
	B_{d,-}(\qstar)  \III_\ell(\Lm r) + C_{d,-}(\qstar)  \KKK_\ell(\Lm r),
\end{align*}
where $B_{d,\pm}$ and $C_{d,-}$ indicate the unknown coefficients to be determined. We note that there is no term like $C_{d,+}(\qstar) \KKK_{\ell}(\Lp r)$ in $U_+$ due to the condition \eqref{linstab:6}.
We recall that $\III_{\ell}$ fulfill the ODE
\begin{align*}
	\Il''(r) + \frac {d-1} r \Il'(r)
+\frac{\zeta_{\ell,d}}{r^2} \Il(r) =  \Il(r) ,
	\quad
	 \Il(0)<\infty,
	 \quad
	 \ell\geq 0,
\end{align*}
where $\zeta_{\ell,d}$ is defined in \eqref{lap}. First, from the relation \eqref{linstab:extra} and the formula \eqref{radial:mu:sol} for $\mu_{\pm}$ we infer the relation
\begin{align}
	& \non
	B_{d,-}(\qstar) \Il \Lmqs  + C_{d,-}(\qstar) \Kl \Lmqs
	- B_{d,+}(\qstar) \Il \Lpqs
	\\ & \quad
	=  \Lp b_{d,+}(\qstar)  \III_1\Lpqs-
	 \Lm b_{d,-} (\qstar)\III_1\Lmqs
	+
	\Lm c_{d,-}(\qstar)  \KKK_1\Lmqs.
	\label{jump:lin:stab}
\end{align}
Next, from \eqref{linstab:3} we have that
\begin{subequations} \label{curv:lin:stab:12}
\begin{align}
	\label{curv:lin:stab:1}
	& \Lp b_{d,+} (\qstar)\III_1\Lpqs
	+ B_{d,+}(\qstar) \Il \Lpqs
	=
	- \alpha \frac {d-1}{(\qstar)^2} \Big( 1+ \frac {\zeta_{\ell,d}}{d-1}\Big),
	\\ \non
	& \Lm  b_{d,-}(\qstar) \III_1\Lmqs
	- \Lm c_{d,-}(\qstar)  \KKK_1\Lmqs
	+ B_{d,-} (\qstar)\Il \Lmqs
	+ C_{d,-}(\qstar) \Kl \Lmqs
	\\ & \quad
	=
	- \alpha \frac {d-1}{(\qstar)^2} \Big( 1+ \frac {\zeta_{\ell,d}}{d-1}\Big),
	\label{curv:lin:stab:2}
\end{align}
\end{subequations}
while from \eqref{linstab:5}, we also have
\begin{align}
	\label{bdcondR:lin:stab}
	 B_{d,-} (\qstar)\Il' \LmR +  C_{d,-}(\qstar) \Kl' \LmR =0.
\end{align}
Since the formula for $b_{d,\pm}(\qstar)$, $c_{d,-}(\qstar)$ are known and given in {\eqref{bc:12:combined}}, we can solve for $B_{d,+}(\qstar)$ as follows:
\begin{align}
	\label{def:Bplus:d}
	B_{d,+} (\qstar)
	=
	\Big(- \alpha \frac {d-1}{(\qstar)^2} \Big( 1+ \frac {\zeta_{\ell,d}}{d-1}\Big)
	- \Lp  b_{d,+}(\qstar) \III_1\Lpqs \Big) \frac 1{\Il \Lpqs}.
\end{align}
From \eqref{bdcondR:lin:stab} we find that
\begin{align*}
	 B_{d,-} (\qstar)= -  C_{d,-}(\qstar)\frac{\Kl' \LmR }{\Il' \LmR },
\end{align*}
and thus, using \eqref{curv:lin:stab:2} we find that
\begin{align}\label{def:Cminus:d}
	C_{d,-}  (\qstar)
	& = \frac{
	- \alpha \frac {d-1}{(\qstar)^2} \Big( 1+ \frac {\zeta_{\ell,d}}{d-1}\Big)
	-\Lm  b_{d,-}(\qstar) \III_1\Lmqs
	+ \Lm c_{d,-} (\qstar) \KKK_1\Lmqs }{	\Kl \Lmqs-  \frac{\Kl' \LmR }{\Il' \LmR } \Il \Lmqs },
\end{align}
which in turn provides an explicit formula for $B_{d,-}(\qstar)$.

Lastly, by substituting the explicit form \eqref{radial:mu:sol} into \eqref{linstab:4}, and using the properties {\eqref{Bessel2d3d}} we obtain
\begin{align*}
\non \frac{2\dot{\delta}}{\delta}  & = -\Big(m_+ (\mu^\star_+)''	-m_- (\mu^\star_-)''	\Big)\Big|_{r=\qstar} - \Big( m_+U_+'	- m_-U_-' \Big)\Big|_{r=\qstar}	\\ 
\non & = -m_+ \Lp^2 b_{d,+}(\qstar) \Iz'' \Lpqs+ m_-\Lm^2 [b_{d,-}(\qstar) \Iz'' \Lmqs+ c_{d,-}(\qstar) \Kz'' \Lmqs] \\ 
\non & \quad - m_+ \Lp B_{d,+} (\qstar)\Il' \Lpqs	+ m_-\Lm B_{d,-} (\qstar)\Il' \Lmqs+ m_-\Lm C_{d,-} (\qstar)\Kl' \Lmqs	\\ 
\non & =	-m_+ \Lp^2 b_{d,+}(\qstar) \III_1' \Lpqs	+ m_-\Lm^2 [b_{d,-}(\qstar) \III_1'  \Lmqs - c_{d,-}(\qstar) \KKK_1' \Lmqs] \\ 
\non & \quad - m_+ \Lp B_{d,+} (\qstar)\Il' \Lpqs	+ m_-\Lm B_{d,-} (\qstar)\Il' \Lmqs	+ m_-\Lm C_{d,-} (\qstar)\Kl' \Lmqs	\\ 
\non & =	-m_+ \Lp \Big( \Lp b_{d,+}(\qstar) \III_1' \Lpqs	+ B_{d,+} (\qstar)\Il' \Lpqs \Big)	\\ 
\non & \quad + m_-\Lm \Big( \Lm [b_{d,-}(\qstar) \III_1'  \Lmqs- c_{d,-}(\qstar) \KKK_1' \Lmqs] \Big)\\ 
& \quad	+ m_-\Lm \Big( B_{d,-} (\qstar)\Il' \Lmqs+C_{d,-} (\qstar)\Kl' \Lmqs \Big),
\end{align*}
which coincides with \eqref{radial:ode:perturb} when we insert the expressions for $b_{d,\pm}$, $c_{d,-}$, $B_{d,\pm}$ and $C_{d,-}$.

\subsection{On the ratio of modified Bessel functions and their derivatives}
\begin{lem}\label{lem:Bessel}
For a fixed index $j \geq 0$, let $\III_j$ and $\KKK_j$ denote the modified Bessel functions of the first and second kind, respectively. Then, it holds that for all $x > 0$,
\[
\frac{\III'_{j+1}(x)}{\III_{j+1}(x)} > \frac{\III'_j(x)}{\III_j(x)}, \quad \frac{\KKK'_{j+1}(x)}{\KKK_{j+1}(x)} < \frac{\KKK'_j(x)}{\KKK_j(x)}. \]
\end{lem}
\begin{proof}
We follow the argument {on} page 242 of \cite{Amos}. For a fixed index $j \geq 0$, let $y_j$ denote a solution to the modified Bessel differential equation. {A} short calculation verifies the following identity:
\[
y_{j} (r y_{j+1}')' - y_{j+1} (r y_j')' = \frac{Q_{j,d}}{r} y_{j} y_{j+1}, \quad Q_{j,d} = \begin{cases} 2 j + 1& \text{ if } d = 2, \\
2 (j + 1) & \text{ if } d = 3.
\end{cases}
\]
For modified Bessel functions of the first kind, i.e., $y_j = \III_j$, integrating and applying Green's theorem on the left-hand side results in
\[
\Big [r(y_j y_{j+1}' - y_j' y_{j+1} ) \Big ]_{r=0}^{r=x} = \int_0^x  y_j (r y_{j+1}')' - y_{j+1} (r y_j')'  = \int_0^x \frac{Q_{j,d}}{r} y_j y_{j+1},
\]
{which} thanks to $\III_j(0) = \III_j'(0) = 0$ leads to 
\[
(y_{j} y_{j+1}' - y_j' y_{j+1})(x) = \frac{Q_{j,d}}{x} \int_0^x \frac{y_j y_{j+1}}{r}.
\]
As $\{\III_j\}_{j \geq 0}$ are positive functions, the right-hand side is also positive. Hence, we deduce the inequality
\[
(y_j y_{j+1}' - y_j' y_{j+1})(x) >0, \quad x > 0,
\]
which in turn implies that 
\[
(\III_{j+1} \III_j' - \III_{j} \III_{j+1}')(x) < 0, \quad x > 0.
\]
Since $\III_{j}$ and $\III_{j+1}$ are positive over $(0,\infty)$ we obtain
\[
0 > \frac{(\III_{j+1} \III_j' - \III_{j} \III_{j+1}')(x)}{\III_j(x) \III_{j+1}(x)} = \frac{\III'_j(x)}{\III_j(x)} - \frac{\III'_{j+1}(x)}{\III_{j+1}(x)}.
\]
Meanwhile, for the modified Bessel functions of the second kind, i.e., $y_j = \KKK_j$, we note that $\lim_{x \to \infty} \KKK_j(x)= 0$ and $\lim_{x \to \infty} \KKK_j'(x) = 0$. Hence, integrating instead from $x$ to $\infty$ we find analogously
\[
- (y_j y_{j+1}' - y_j' y_{j+1})(x) = \frac{Q_{j,d}}{x} \int_{x}^\infty \frac{y_j y_{j+1}}{r} > 0.
\]
This indicates that
\[
 (\KKK_{j+1} \KKK_{j}' - \KKK_{j} \KKK_{j+1}')(x)> 0.
\]
Since $\KKK_{j}$ and $\KKK_{j+1}$ are positive over $(0,\infty)${,} we deduce that 
\[
0 < \frac{(\KKK_{j+1} \KKK_{j}' - \KKK_{j} \KKK_{j+1}')(x)}{\KKK_j(x) \KKK_{j+1}(x)} = \frac{\KKK_j'(x)}{\KKK_{j}(x)} - \frac{\KKK_{j+1}'(x)}{\KKK_{j+1}(x)}.
\]
\end{proof}
\begin{cor}\label{cor:Bessel}
For any $\ell \in \NNN$, $\ell > 1$, it holds that
\[
\III_1'(x) \III_{\ell}(x) - \III_1(x) \III_{\ell}'(x) < 0, \quad \KKK_1'(x) \KKK_{\ell}(x) - \KKK_1(x) \KKK_{\ell}'(x) > 0, \quad \forall \ell > 1, \quad \forall x > 0.
\]
\end{cor}
\begin{proof}
Via an induction argument with Lemma \ref{lem:Bessel}, we see that
\[
\frac{\III_1'(x)}{\III_1(x)} < \frac{\III_{\ell}'(x)}{\III_{\ell}(x)},  \quad \frac{\KKK_1'(x)}{\KKK_1(x)} > \frac{\KKK_{\ell}'(x)}{\KKK_{\ell}(x)}, \quad \forall \ell > 1, \quad \forall x> 0,
\]
and hence upon rearranging we obtain
\[
\III_1'(x) \III_{\ell}(x) - \III_1(x) \III_{\ell}'(x) < 0, \quad \KKK_1'(x) \KKK_{\ell}(x) - \KKK_1(x) \KKK_{\ell}'(x) > 0, \quad \forall \ell > 1, \quad \forall x > 0.
\]
\end{proof}

\section{Radial multilayered solutions}
\subsection{Analytical formula for solutions}\label{app:radialshell}
Suppressing the explicit dependence on $t$, we postulate the following ansatz for the solutions
\begin{subequations} \label{ansatz:shell:mu}
\begin{align}
	\label{ansatz:shell:mumin}
	\mumin(r) & =  \bmin(q_1) \Iz (\Lm r) 
	+\emin,
	\quad 
	&& r \in (0,q_1(t)),
	\\
	\label{ansatz:shell:mup}
	\mu_+(r) & = b_{d,+}\qud \Iz(\Lp r) 
	+c_{d,+}\qud \Kz(\Lp r)  
	+ e_+,
	\quad 
	&& r \in (q_1(t),q_2(t)),
	\\
	\label{ansatz:shell:mumex}
	\mumex (r) & = \bmex (q_2) \Iz(\Lm r)  
	+	\cmex (q_2) \Kz(\Lm r) 
	+\emex ,
	\quad 
	&& r \in (q_2(t),R),
\end{align}
\end{subequations}
where, similar to Appendix \ref{app:radial:soln}, in accordance with \eqref{rad:shell:11}  we highlight that the ansatz for $\mumin$ does not contain a term like $c^{{\rm in}}_{d,+} (q_1) \Kz (\Lm r)$ as this is singular at the origin.
We then proceed as before by inserting the expressions {\eqref{ansatz:shell:mu}} into the system  \eqref{rad:shell} to derive the expressions of the unknown coefficients and functions. Substituting the ansatz for $\mumin$, $\mu_+$ and $\mumex$ into \eqref{rad:shell:1}--\eqref{rad:shell:3} yields the identification of the constants
\[
\emin = \emex = \frac{S_-}{\rho_-}, \quad e_+ = \frac{S_+}{m_+}.
\]
Next, using the conditions \eqref{rad:shell:4} at the free boundary $\{r = q_1(t)\}$ we obtain from the expressions \eqref{ansatz:shell:mumin} and \eqref{ansatz:shell:mup} the following relations
\begin{align}
\label{rad:bmin} & \bmin(q_1) \Iz(\Lm q_1) = \mumin(q_1) - \frac{S_-}{\rho_-} = - \alpha \frac{d-1}{q_1}  - \frac{S_-}{\rho_-}, \\
\label{rad:bcplus} & b_{d,+}\qud \Iz(\Lp q_1)
	+c_{d,+}\qud \Kz(\Lp q_1) + \frac{S_+}{\rho_+} = - \alpha \frac{d-1}{q_1}.
\end{align}
Notice that $\bmin(q_1)$ can be solved directly:
\begin{align}\label{bmin}
	\bmin (q_1)  = \Big(-\alpha\frac{d-1}{q_1} - \frac{S_-}{\rho_-} \Big) \frac 1{\Iz(\Lm  q_1)}.
\end{align}
Then, from \eqref{rad:shell:5} at the free boundary $\{ r = q_2(t)\}$ we obtain 
\begin{align}
	\label{rad:secondstep:1}
	& b_{d,+}\qud \Iz (\Lp q_2) 
	+c_{d,+}\qud \Kz(\Lp q_2)
	 = 
	\alpha\frac{d-1}{q_2}
	- \frac{S_+}{\rho_+},
	\\
    	\label{shell:comp:cond}
	& \bmex (q_2) \Iz(\Lm q_2) 
	+	\cmex (q_2) \Kz(\Lm q_2)  
	= \alpha\frac{d-1}{q_2} 
	- \frac{S_-}{\rho_-}.
\end{align}
We are now in a position to determine $b_{d,+} $ and $c_{d,+}$, which can be achieved by combining \eqref{rad:bcplus} with \eqref{rad:secondstep:1}.
First, we point out the  useful identities 
\begin{align*}
& \Iz (\Lp q_1)\Kz (\Lp q_2) - \Kz (\Lp q_1)  \Iz (\Lp q_2)
\\ & \quad
= \Big( \frac{\Iz (\Lp q_1)}{ \Kz (\Lp q_1)}
- \frac {\Iz (\Lp q_2)}{\Kz (\Lp q_2)}\Big) \big(\Kz (\Lp q_1)\Kz(\Lp q_2)\big)
\\ & \quad
= \Big( \frac{\Kz (\Lp q_2)}{ \Iz (\Lp q_2)}
- \frac {\Kz (\Lp q_1)}{\Iz (\Lp q_1)}\Big) \big(\Iz (\Lp q_1)\Iz (\Lp q_2)\big).
\end{align*}
Then, expressing \eqref{rad:bcplus} and \eqref{rad:secondstep:1} in matrix form:
\begin{align*}
\begin{pmatrix}
\Iz(\Lp q_1) & \Kz(\Lp q_1) \\
\Iz(\Lp q_2) & \Kz(\Lp q_2)
\end{pmatrix}
\begin{pmatrix}
b_{d,+} \qud \\ c_{d,+} \qud
\end{pmatrix} = \begin{pmatrix} - \alpha \frac{d-1}{q_1} - \frac{S_+}{\rho_+} \\[1ex] 
\alpha \frac{d-1}{q_2} - \frac{S_+}{\rho_+}
\end{pmatrix},
\end{align*}
we find
\begin{align*}
	b_{d,+}\qud & 
	 = \Big( \frac{\Iz (\Lp q_1)}{ \Kz (\Lp q_1)} 
- \frac {\Iz (\Lp q_2)}{\Kz (\Lp q_2)}\Big)^{-1}\\
& \quad \times  \Big(\frac 1{\Kz(\Lp q_1)} \Big(-\alpha\frac{d-1}{q_1} - \frac{S_+}{\rho_+} \Big) - \frac1{\Kz(\Lp q_2)}  \Big(\alpha\frac{d-1}{q_2} - \frac{S_+}{\rho_+} \Big)\Big),
	\\
	c_{d,+}\qud & =
	\Big( \frac{\Kz (\Lp q_2)}{ \Iz (\Lp q_2)}
- \frac {\Kz (\Lp q_1)}{\Iz (\Lp q_1)}\Big)^{-1} \\
& \quad \times  \Big(\frac 1{\Iz(\Lp q_2)} \Big(\alpha\frac{d-1}{q_2} - \frac{S_+}{\rho_+} \Big) - \frac1{\Iz(\Lp q_1)}  \Big(-\alpha\frac{d-1}{q_1} - \frac{S_+}{\rho_+} \Big)\Big).
\end{align*}
Lastly, we make use of the external boundary condition \eqref{rad:shell:10} to see that 
\begin{align*}
	0& =(\mumex)'(r=R) = \Lm \bmex (q_2) \Iz'(\Lm R) 
	+	\Lm \cmex (q_2) \Kz'(\Lm R) \\
    & =  \Lm \bmex (q_2) \III_1(\Lm R)  -\Lm \cmex (q_2) \KKK_1(\Lm R).
	\end{align*}
Expressing this and \eqref{shell:comp:cond} in matrix form:
\begin{align*}
	\begin{pmatrix}
	\Iz (\Lm q_2) & \Kz (\Lm q_2) \\
	\III_1 (\Lm R) & -\KKK_1 (\Lm R) 
	\end{pmatrix}
	\begin{pmatrix}
	\bmex (q_2)\\
	\cmex (q_2)
	\end{pmatrix}
	=
	\begin{pmatrix}
	\alpha\frac{d-1}{q_2} 
	- \frac{S_-}{\rho_-}
	\\
	0
	\end{pmatrix},
\end{align*}
and thus, we infer 
\begin{align}
	\label{bmex:sol}
	\bmex (q_2) & = \Big( \frac{\Iz (\Lm q_2)}{ \Kz (\Lm q_2)}
+ \frac {\III_1 (\Lm R)}{\KKK_1 (\Lm R)}\Big)^{-1} \Big( \alpha\frac{d-1}{q_2} 
	- \frac{S_-}{\rho_-} \Big) \frac 1 {\Kz (\Lm q_2)} ,
	\\
	\label{cmex:sol}
	\cmex (q_2) & =  \Big( \frac{\KKK_1 (\Lm R)}{ \III_1 (\Lm R)}
+ \frac {\Kz (\Lm q_2)}{\Iz (\Lm q_2)}\Big)^{-1}
	\Big( \alpha\frac{d-1}{q_2} 
	- \frac{S_-}{\rho_-} \Big)\frac 1{\Iz (\Lm q_2)},
\end{align}
where we again owe to a suitable 
factorizations of the following terms
\begin{align*}
& \Iz (\Lm q_2)\KKK_1 (\Lm R) + \Kz (\Lm q_2)  \III_1 (\Lm R)
\\ & \quad
= \Big( \frac{\Iz (\Lm q_2)}{ \Kz (\Lm q_2)}
+ \frac {\III_1 (\Lm R)}{\KKK_1 (\Lm R)}\Big) \big(\Kz (\Lm q_2)\KKK_1 (\Lm R)\big)
\\ & \quad
= \Big( \frac{\KKK_1 (\Lm R)}{ \III_1 (\Lm R)}
+ \frac {\Kz (\Lm q_2)}{\Iz (\Lm q_2)}\Big) \big(\Iz (\Lm q_2)\III_1 (\Lm R)\big).
\end{align*}
With these expression we obtain the analytical formulas in {\eqref{def:shell:mu}}. 

\subsection{Solutions to the perturbed system on finite domains}\label{app:radialshell:perturb}
Below, we denote the stationary solutions of the system by 
$\qstaru$ and 
$\qstard$, as anticipated in {Section \ref{SEC:RAD:MULTI}}.
We look for general solutions of the form
\begin{align*}
	\Uin (r) & =
	B^{{\rm in}}_{d,-}(\qstaru) \Iell(\Lm  r) ,
	\\
	\Upl (r)  & =
	B_{d,+}\qudst  \Iell(\Lp r) + C_{d,+}\qudst  \Kell(\Lp r),
	\\
	\Uext (r)  & =
	B_{d,-}^{{\rm ext}}(\qstard)  \Iell(\Lm r) + C_{d,-}^{{\rm ext}}(\qstard)  \Kell(\Lm r),
\end{align*}
where $B^{ {\rm in}}_{d,-}(\qstaru), B_{d,+}\qudst , C_{d,+}\qudst, B^{{\rm ext}}_{d,-}(\qstard)$, and $C^{{\rm ext}}_{d,-}(\qstard)$ indicate the unknowns to be determined. As before {the fact that} the ansatz for $\Uin(r)$ does not contain a term like $C^{{\rm in}}_{d,-}(\qstaru) \Kell(\Lm r)$ is due to the boundary condition \eqref{linstab:shell:9} at the origin.
First, from the relations \eqref{linstab:shell:extra:1} and \eqref{linstab:shell:extra:2}, the formulas {\eqref{ansatz:shell:mu}}, as well as $\III_0'(r) = \III_1$ and $\KKK_0'(r) = - \KKK_1(r)$, we obtain 
\begin{align}
	& \non
	B_{d,+}\qudst  \Iell(\Lp \qstaru) + C_{d,+}\qudst  \Kell(\Lp \qstaru)
	- B^{{\rm in}}_{d,-}(\qstaru) \Iell(\Lm  \qstaru) 
	\\ & \quad
	=  [\Lm b_{d,-}^{{\rm in}}(\qstaru)  \III_1(\Lm \qstaru)-
	 \Lp b_{d,+} \qudst\III_1 (\Lp \qstaru)
	+
	\Lp c_{d,+} \qudst  \KKK_1 (\Lp \qstaru)]y_1, 
	\label{jump:shell:stab:1} \\
		& \non
	B_{d,+}\qudst  \Iell(\Lp \qstard) + C_{d,+}\qudst  \Kell(\Lp \qstard)
	- B^{{\rm ext}}_{d,-}(\qstard) \Iell(\Lm  \qstard) - C^{{\rm ext}}_{d,-}(\qstard) \Kell(\Lm \qstard)
	\\ & \quad
	\notag =  [\Lm b_{d,-}^{{\rm ext}}(\qstard)  \III_1(\Lm \qstard) - \Lm c_{d,-}^{{\rm ext}}(\qstard) \KKK_1(\Lm \qstard)] y_2 \\
	& \qquad  -[ 
	 \Lp b_{d,+} \qudst\III_1 (\Lp \qstard)
	-
	\Lp c_{d,+}\qudst  \KKK_1 (\Lp \qstard)]y_2.
	\label{jump:shell:stab:2} 
\end{align} 
Next, from \eqref{linstab:shell:4} and \eqref{linstab:shell:5} we have that
\begin{align}
	& 
	\Lm  b_{d,-}^{{\rm in}} (\qstaru)\III_1(\Lm \qstaru) y_1
	+ B_{d,-}^{{\rm in}}(\qstaru) \Iell(\Lm  \qstaru) 
		\label{lin:shell:1}
	= \alpha \frac {d-1}{(\qstaru)^2} \Big( 1 + \frac {\zeta_{\ell,d}}{d-1}\Big) y_1,
	\\
		& \non
	[\Lp  b_{d,+} \qudst\III_1 (\Lp \qstaru) 
		- \Lp  c_{d,+} \qudst\KKK_1(\Lp\qstaru) ]y_1 \\
		& \label{lin:shell:2} \qquad 
	+ B_{d,+} \qudst \Iell(\Lp \qstaru) + C_{d,+}\qudst  \Kell(\Lp \qstaru)	
	= \alpha \frac {d-1}{(\qstaru)^2} \Big( 1 + \frac {\zeta_{\ell,d}}{d-1}\Big)  y_1,
	\\
		& \non
	[\Lp  b_{d,+} \qudst\III_1 (\Lp \qstard)
	- \Lp  c_{d,+} \qudst\KKK_1(\Lp\qstard)] y_2
	\\ &  	\label{lin:shell:3} \qquad + B_{d,+} \qudst \Iell(\Lp \qstard)+ C_{d,+}\qudst  \Kell(\Lp \qstard)
	= -\alpha \frac {d-1}{(\qstard)^2} \Big( 1 + \frac {\zeta_{\ell,d}}{d-1}\Big)  y_2,
	\\
		& \non
	[\Lm  b_{d,-}^{{\rm ext}} (\qstard)\III_1(\Lm \qstard)
	- \Lm  c_{d,-}^{{\rm ext}} (\qstard)\KKK_1(\Lm \qstard)] y_2
	\\ &  	\label{lin:shell:4} \qquad + B_{d,-}^{{\rm ext}}(\qstard) \Iell(\Lm  \qstard)
	+ C_{d,-}^{{\rm ext}}(\qstard) \Kell(\Lm  \qstard) 
	= -\alpha \frac {d-1}{(\qstard)^2} \Big( 1 + \frac {\zeta_{\ell,d}}{d-1}\Big) y_2,
\end{align}
while from \eqref{linstab:shell:8}, we also have
\begin{align}
	\label{shell:lin:stab:bdcond}
	 B_{d,-}^{{\rm ext}} (\qstard)\Il' \LmR +  C_{d,-}^{{\rm ext}}(\qstard) \Kl' \LmR =0.
\end{align}
Since $b_{d,-}^{{\rm in}} (\qstaru)$ is known, we can solve \eqref{lin:shell:1} in terms of $B_{d,-}^{{\rm in}}(\qstaru)$
and obtain
\begin{align}\label{Bd:in:min}
B_{d,-}^{{\rm in}}(\qstaru)
	= \Big(\alpha \frac {d-1}{(\qstaru)^2} \Big( 1 + \frac {\zeta_{\ell,d}}{d-1}\Big)
	- \Lm  b_{d,-}^{{\rm in}} (\qstaru)\III_1(\Lm \qstaru)\Big) \frac {y_1} { \Iell(\Lm  \qstaru)}.
\end{align}
To identify $B_{d,+}\qudst$ and $C_{d,+} \qudst$ we can use \eqref{lin:shell:2} and \eqref{lin:shell:3} to infer that 
\begin{align*}
	 & 
	B_{d,+} \qudst \Iell(\Lp \qstaru) + C_{d,+} \qudst  \Kell(\Lp \qstaru)
	\\ & \quad 	
	= \Big(\alpha \frac {d-1}{(\qstaru)^2} \Big( 1 + \frac {\zeta_{\ell,d}}{d-1}\Big)
	-\Lp  b_{d,+} \qudst\III_1 (\Lp \qstaru)
	+ \Lp  c_{d,+} \qudst\KKK_1(\Lp\qstaru)	
	\Big) y_1\\
	& \quad 	=: y_1 F_1\qudst,
	\\
	& \non
	B_{d,+} \qudst \Iell(\Lp \qstard)+ C_{d,+} \qudst  \Kell(\Lp \qstard)
	\\ & \quad 
	= \Big(-\alpha \frac {d-1}{(\qstard)^2} \Big( 1 + \frac {\zeta_{\ell,d}}{d-1}\Big)
	-\Lp  b_{d,+} \qudst\III_1 (\Lp \qstard)
	+ \Lp  c_{d,+} \qudst\KKK_1(\Lp\qstard)\Big)  y_2\\
	& \quad 
	=: y_2 F_2\qudst,
\end{align*}
which can be expressed as
\begin{align*}
	\begin{pmatrix}
	\Iell(\Lp \qstaru) & \Kell(\Lp \qstaru) \\
	\Iell(\Lp \qstard) & \Kell(\Lp \qstard)
	\end{pmatrix}
	\begin{pmatrix}
	B_{d,+} \qudst \\			
	C_{d,+} \qudst
	\end{pmatrix}
	=
	\begin{pmatrix}
	 y_1 F_1 \qudst \\
	 y_2 F_2\qudst
	\end{pmatrix}.
\end{align*}
Solving the linear system leads to 
\begin{align*}
	B_{d,+} \qudst & = \frac  { y_1 F_1 \qudst\Kell(\Lp\qstard) -    y_2 F_2 \qudst\Kell(\Lp\qstaru)}{\Iell(\Lp \qstaru)\Kell(\Lp \qstard)-\Iell(\Lp \qstard)\Kell(\Lp \qstaru)},
	\\
	C_{d,+} \qudst & = \frac  {\Iell(\Lp\qstaru) y_2 F_2\qudst -   \Iell (\Lp \qstard) y_1 F_1 \qudst}{\Iell(\Lp \qstaru)\Kell(\Lp \qstard)-\Iell(\Lp \qstard)\Kell(\Lp \qstaru)}.
\end{align*}
Next, observe from \eqref{shell:lin:stab:bdcond} that 
\begin{align*}
	 B_{d,-}^{{\rm ext}} (\qstard) = -C_{d,-}^{{\rm ext}}(\qstard) \frac{\Kl' \LmR}{\Il' \LmR},
\end{align*}
and upon substituting into \eqref{lin:shell:4} we can infer that
\begin{align}
	\non
	C_{d,-}^{{\rm ext}}(\qstard)
	& =
	\Big(- \alpha \frac {d-1}{(\qstard)^2} \Big( 1 + \frac {\zeta_{\ell,d}}{d-1}\Big)
	- \Lm  b_{d,-}^{{\rm ext}} (\qstard)\III_1(\Lm \qstard)
	+ \Lm  c_{d,-}^{{\rm ext}} (\qstard)\KKK_1(\Lm \qstard) 
	\Big) y_2
	\\ & \quad 	\label{Cd:extmin} 
	\times \Big( \Kell(\Lm  \qstard)  - \Big( \frac{\Kl' \LmR}{\Il' \LmR} \Big)\Iell(\Lm  \qstard) \Big)^{-1},
\end{align}
which in turn provides an explicit formula for $B^{d,{\rm in}}_-(\qstaru)$ as
\begin{align}
	\non
	B_{d, -}^{{\rm ext}}(\qstard)
	& =
	\Big(- \alpha \frac {d-1}{(\qstard)^2} \Big( 1 + \frac {\zeta_{\ell,d}}{d-1}\Big)
	- \Lm  b_{d,-}^{{\rm ext}} (\qstard)\III_1(\Lm \qstard)
	+ \Lm  c_{d,-}^{{\rm ext}} (\qstard)\KKK_1(\Lm \qstard) 
	\Big) y_2
	\\ & \quad 	\label{Bd:extmin} 
	\times \Big( \Iell(\Lm  \qstard)  - \Big( \frac{\Il' \LmR}{\Kl' \LmR} \Big)\Kell(\Lm  \qstard) \Big)^{-1}.
\end{align}
Recalling the solution {formulas \eqref{def:shell:mu}}, we calculate 
\begin{align*}
(\mustin)''(\qstaru) & = \lambda_-^2 \bmin(\qstaru) \III_0''(\lambda_- \qstaru), \\
(\mustpl)''(r) & = \lambda_+^2 b_{d,+}(\qstaru,\qstard) \III_0''(\Lp r) + \lambda_+^2 c_{d,+}(\qstaru,\qstard) \KKK_0''(\Lp r), \\
(\mumex)''(\qstard) & = \lambda_-^2 \bmex (\qstard) \III_0''(\Lm \qstard) + \lambda_-^2 \cmex(\qstard) \KKK_0''(\Lm \qstard).
\end{align*}
Then, upon substituting these explicit expressions as well as the {formulas} for $U_+$ and $U_-$ into \eqref{linstab:shell:6} and \eqref{linstab:shell:7}, and recalling the prime notation to indicate the derivative with respect to the radial component, we arrive at
\begin{align*}
	 \frac{2 \dot{\delta}}{{\delta}} y_1 
	&=
	-\big(m_+(\mustpl)'' - m_-(\mustin)''\big ) \Big|_{r=\qstaru} y_1
	-  \big(m_+(\Upl)'  - m_-(\Uin)' \big) \Big|_{r=\qstaru}
	\\
	& 
	= 
	- m_+ \Lp^2 [b_{d,+} \qudst  \Iz'' (\Lp \qstaru)
	+c_{d,+} \qudst  \Kz'' (\Lp \qstaru)] y_1
	\\
	& \quad + m_- \Lm^2 b_{d,-}^{{\rm in}} (\qstaru)  \Iz'' (\Lm \qstaru) y_1 	- m_+ \Lp B_{d,+} \qudst  \III_\ell' (\Lp \qstaru)\\
	& \quad 
	- m_+ \Lp C_{d,+} \qudst  \KKK_\ell' (\Lp \qstaru)
	+ m_- \Lm B_{d,-}^{{\rm in}} (\qstaru)  \III_\ell' (\Lm \qstaru),
	\\
	 \frac{2 \dot{\delta}}{{\delta}} y_2 
	& =
	-\big(m_+(\mu^\star_+)''
	- m_-(\mustext)''\big)\Big|_{r=\qstard} y_2
	-   \big(m_+(U_+)'-m_-(\Uext)')\big)\Big|_{r=\qstard}
	\\
	& 
	= 
	- m_+ \Lp^2 [b_{d,+} \qudst  \Iz'' (\Lp \qstard)
	+c_{d,+} \qudst  \Kz'' (\Lp \qstard)]y_2
	\\
	& \quad
		+  m_- \Lm^2 [b_{d,-}^{{\rm ext}} (\qstard) \Iz'' (\Lm \qstard)
	+ c_{d,-}^{{\rm ext}} (\qstard)  \Kz'' (\Lm \qstard)]y_2
	\\
	& \quad
	- m_+ \Lp [B_{d,+} \qudst \III_\ell' (\Lp \qstard)
	+C_{d,+} \qudst  \KKK_\ell' (\Lp \qstard)]
		\\
	& \quad
	+ m_- \Lm [B_{d,-}^{{\rm ext}} (\qstard)  \III_\ell' (\Lm \qstard)
	+  C_{d,-}^{{\rm ext}} (\qstard)  \KKK_\ell' (\Lm \qstard)],
\end{align*}
where we recall that $B_{d,-}^{{\rm in}}$, $B_{d,+}$, $C_{d,+}$, $B_{d,-}^{{\rm ext}}$ and $C_{d,-}^{{\rm ext}}$ contain $y_1$ and $y_2$. Then, we can express this system with a matrix structure
\begin{align*}
	& 2 \dot{\delta} 
	\begin{pmatrix}
	y_{1 }\\
	y_{2 }
	\end{pmatrix}
	=
	\begin{pmatrix}
	A_{11}  & A_{12} \\
	A_{21} & A_{22} 
	\end{pmatrix}
	\begin{pmatrix}
	y_{1 }\\
	y_{2 }
	\end{pmatrix}
	\delta = \delta \mathbf{A} \begin{pmatrix}
	y_{1 }\\
	y_{2 }
	\end{pmatrix}{,}
\end{align*}
where the entries of the matrix $\mathbf{A}$ are
\begin{equation}\label{app:radialshell:matrix}
\begin{aligned}
A_{11} & = - m_+ \Lp^2 [b_{d,+} \qudst  \Iz'' (\Lp \qstaru)
	+c_{d,+} \qudst  \Kz'' (\Lp \qstaru)]  \\
	& \quad + m_- \Lm^2 b_{d,-}^{{\rm in}} (\qstaru)  \Iz'' (\Lm \qstaru) \\
	& \quad -m_+ \lambda_+ \frac{ \Iell'(\lambda_+ \qstaru) \Kell(\lambda_+ \qstard) -  \Kell'(\lambda_+ \qstaru)  \Iell(\lambda_+ \qstard)}{\Iell(\Lp \qstaru) \Kell(\Lp \qstard) - \Iell(\Lp \qstard) \Kell(\Lp \qstaru)}F_1(\qstaru,\qstard) \\
	& \quad + m_- \lambda_- \Big(\alpha \frac {d-1}{(\qstaru)^2} \Big( 1 + \frac {\zeta_{\ell,d}}{d-1}\Big)
	- \Lm  b_{d,-}^{{\rm in}} (\qstaru)\III_1(\Lm \qstaru)\Big)\frac{ \Iell'(\lambda_- \qstaru)}{ \Iell(\Lm  \qstaru)}, \\
A_{12} & = m_+ \Lp \frac{ \Iell'(\Lp \qstaru) \Kell(\Lp \qstaru) - \Iell(\Lp \qstaru) \Kell'(\Lp \qstaru)}{\Iell(\Lp \qstaru) \Kell(\Lp \qstard) - \Iell(\Lp \qstard) \Kell(\Lp \qstaru)} F_2(\qstaru,\qstard), \\
A_{21} & = - m_+ \lambda_+ \frac{\Iell'(\Lp \qstard) \Kell(\Lp \qstard) -  \Iell(\Lp \qstard) \Kell'(\Lp \qstard)}{\Iell(\Lp \qstaru) \Kell(\Lp \qstard) - \Iell(\Lp \qstard) \Kell(\Lp \qstaru)}F_1(\qstaru,\qstard), \\
A_{22} & = - m_+ \Lp^2 [b_{d,+} \qudst  \Iz'' (\Lp \qstard)
	+c_{d,+} \qudst  \Kz'' (\Lp \qstard)]\\
	& \quad
		+  m_- \Lm^2 [b_{d,-}^{{\rm ext}} (\qstard) \Iz'' (\Lm \qstard)
	+ c_{d,-}^{{\rm ext}} (\qstard)  \Kz'' (\Lm \qstard)] \\
	& \quad 
	+m_+ \Lp \frac{\Iell'(\Lp \qstard) \Kell(\Lp \qstaru) - \Kell'(\Lp \qstard) \Iell(\Lp \qstaru)}{\Iell(\Lp \qstaru) \Kell(\Lp \qstard) - \Iell(\Lp \qstard) \Kell(\Lp \qstaru)} F_2(\qstaru,\qstard) \\
	& \quad + m_- \Lm \frac{\Kell'(\lambda_- \qstard) \Iell'(\Lm R) - \Kell'(\Lm R) \Iell'(\Lm \qstard)}{\Kell(\Lm \qstard) \Iell'(\Lm R) - \Kell'(\Lm R) \Iell(\Lm \qstard)} \\
	& \qquad \times \Big(- \alpha \frac {d-1}{(\qstard)^2} \Big( 1 + \frac {\zeta_{\ell,d}}{d-1}\Big)
	- \Lm  b_{d,-}^{{\rm ext}} (\qstard)\III_1(\Lm \qstard)
	\\
	& \qquad \qquad + \Lm  c_{d,-}^{{\rm ext}} (\qstard)\KKK_1(\Lm \qstard) 
	\Big).
\end{aligned}
\end{equation}
\footnotesize
\bibliographystyle{abbrv} 
\bibliography{refs} 

\end{document}